\documentclass[11pt]{article}

\usepackage{graphicx}

\usepackage{amssymb,amsmath,amsthm,enumerate,bm}
\usepackage{fullpage}

\usepackage{tikz}

\RequirePackage[colorlinks,citecolor=blue,urlcolor=blue]{hyperref}

\newtheorem{lemma}{Lemma}
\newtheorem{remark}{Remark}
\newtheorem{proposition}[lemma]{Proposition}
\newtheorem{definition}[lemma]{Definition}
\newtheorem{theorem}[lemma]{Theorem}

\newtheorem{corollary}[lemma]{Corollary}

\newcommand{\BIN}{{\mathrm{Bin}}}

\newcommand{\tr}{{\rm tr}}

\newcommand{\dE}{\mathbb {E}}
\newcommand{\dP}{\mathbb{P}}

\newcommand{\dR}{\mathbb {R}}
\newcommand{\dC}{\mathbb {C}}
\newcommand{\cF}{\mathcal {F}}

\newcommand{\cE}{E}
\newcommand{\cG}{\mathcal {G}}

\newcommand{\cW}{\mathcal {W}}

\newcommand{\Bin}{ \mathrm{Bin}}

\newcommand{\SPAN}{ \mathrm{span}}

\newcommand{\ABS}[1]{{{\left| #1 \right|}}} 
\newcommand{\BRA}[1]{{{\left\{#1\right\}}}} 
\newcommand{\SBRA}[1]{{{\left[#1\right]}}} 
\newcommand{\PAR}[1]{{{\left(#1\right)}}} 

\newcommand{\uB}{{\underline B}}

\def\lp{(\hspace{-2pt}(}
\def\rp{)\hspace{-2pt})}

\newcommand{\1}{1\!\!{\sf I}}\newcommand{\IND}{\1}
\newcommand{\veps}{\varepsilon}
\newcommand{\si}{\sigma}

\newcommand{\Be}{\bm{e}}
\newcommand{\Bf}{\bm{f}}

\newcommand{\whp}{{w.h.p.~}}

\title{A new proof of Friedman's second eigenvalue Theorem and its extension to random lifts}
\author{Charles Bordenave\footnote{The author is supported by the research grants ANR-14-CE25-0014 and ANR-16-CE40-0024-01.}}

\begin{document}
\maketitle

\begin{abstract}
It was conjectured by Alon and proved by Friedman that a random $d$-regular graph has nearly the largest possible spectral gap, or, more precisely, the largest absolute value of the non-trivial eigenvalues of its adjacency matrix is at most $2\sqrt{d-1} +o(1)$ with probability tending to one as the size of the graph tends to infinity. We give a new  proof of this statement.  We also study related questions on random $n$-lifts of graphs  and improve a recent result by Friedman and Kohler.


\end{abstract}

\section{Introduction}

Consider a finite simple graph $G = (V,E)$ with $ n = |V|$ vertices.  Its adjacency matrix $A = A(G)$ is the matrix indexed by $V$ and defined for all $u,v \in V$ by $A_{uv} = \IND_{\{ u,v \} \in E}$ where $\IND$ denotes the indicator function. The matrix $A$ is symmetric, its eigenvalues $\mu_i = \mu_i (G)$ are real and we order them non-increasingly, 
$$
\mu_{n} \leq \ldots \leq \mu_1, 
$$
We assume further that, for some integer $d \geq 3$, the graph $G$ is $d$-regular, that is, all vertices have degree $d$. We then have that $\mu_1 = d$, that all eigenvalues have absolute value at most $d$, and $\mu_n = -d$ is equivalent to $G$ having a bipartite connected component. The absolute value of the largest non-trivial eigenvalues of $G$ is denoted by
$
\mu = \mu(G) = \max \{ |\mu_i | :  |\mu_i| < d \}.  
$
Classical statements such as Cheeger's isoperimetric inequality or Chung's diameter inequality relate small values of $\mu$ or $\mu_2$ with good expanding properties of the graph $G$, we refer for example to \cite{MR1421568,MR2247919}.  It turns out that $\mu$ cannot be made arbitrarily small. Indeed, a celebrated result of Alon-Boppana implies that for any $d$-regular graph with $n$ vertices, 
\begin{equation}\label{eq:AB}
\mu_2 (G) \geq 2 \sqrt{d-1}    -  \veps_d (n), 
\end{equation}
where, for some constant $c_d > 0$, $\veps_d (n) = c_d  /   (\log n)^2 $; see the above references and \cite{MR1124768,MR2437174,MR2679612}. Following \cite{MR939574,MR963118}, one may try to construct  graphs which achieve the Alon-Boppana bound. A graph is called Ramanujan if $\mu \leq 2 \sqrt{d-1}$. Proving the existence of Ramanujan graphs with a large number of vertices is a difficult task which has been solved for arbitrary $d \geq 3$ only recently \cite{MSS}. On the other end, it was conjectured by Alon \cite{MR875835} and proved by Friedman \cite{MR2437174} that most $d$-regular graphs are weakly Ramanujan.   More precisely, for integer $n \geq 1$, we define $\cG_d(n)$ as the set of simple $d$-regular graphs with vertex set $\{1 , \ldots , n\}$. If $nd$ is even and $d \leq n -1$, this set is non-empty (for $nd$ odd, a definition of $\cG_d(n)$ is given in \cite{MR2437174}). A uniformly sampled $d$-regular graph is then a random graph whose distribution is uniform on $\cG_d (n)$.

\begin{theorem}[Friedman's second eigenvalue Theorem \cite{MR2437174}]\label{th:Fr}
Let  $d \geq 3$ be an integer and $nd$ be even. If $G$ is uniformly distributed on $\cG_d(n)$, we have for any $\veps > 0$, 
$$
\lim_{ n \to \infty } \dP \PAR{  \mu_2 \vee |\mu_n| \geq 2 \sqrt{d-1} +\veps } = 0,  
$$
where $a \vee b = \max(a,b)$ and the limit is along any sequence going to infinity with $nd$ even.
\end{theorem}

The first aim of this paper is to give a new proof of this result. The argument detailed in Section \ref{sec:proofF} simplifies substantially the original proof. A careful reading of the proof actually gives the following quantitative statement: for any $0 < a < 1$, there exists $c>0$ (depending on $d$ and $a$) such that for all integers $n$ such that $\cG_d(n)$ is non-empty,
\begin{equation}\label{eq:QFried}
\dP \PAR{  \mu_2 \vee |\mu_n| \geq 2 \sqrt{d-1} +c \PAR{\frac{\log \log n }{ \log n}} ^2 } \leq  n^{-a}.  
\end{equation}

The method is robust and it has been recently applied in \cite{BLM} to random graphs with structure (stochastic block model).

The second aim of this paper is to apply this method to study similar questions on the eigenvalues of random lifts of graphs. This class of models  sheds a new light on Ramanujan-type  properties, and, since the work of Amit and Linial \cite{MR1883559,MR2216470} and Friedman \cite{MR1978881},  it has attracted a substantial attention \cite{MR2674623,AB-S,MR2799807,puder,FrKo}. To avoid any confusion in notation, we will postpone to Section \ref{sec:introbis} the precise definition of random lifts and the statement of the main results. In Section \ref{secM:proof} we will give a simpler proof of a recent result of Friedman and Kohler \cite{FrKo} and establish a weak Ramanujan property for the non-backtracking eigenvalues of a random lift of an arbitrary graph. 

 \paragraph{Notation.}
If $n$ is a positive integer, we set $[n ] = \{ 1, \ldots, n\}$. If $M \in M_n(\dR)$, $M^*$ denotes its conjugate transpose and we denote its operator norm by 
$$
\| M \| = \sup_{x \in \dR^n, x \ne 0} \frac{\| M x\|_ 2}{\|x \|_2}. 
$$

For positive sequences $a_n$, $b_n$, we will use the standard notation $a_n \sim b_n$ (if $\lim_{ n \to \infty} a_n / b_n =  1$),  $a_n = O(b_n)$ (if $\limsup_{n \to \infty} a_n / b_n < \infty$) and $a_n = o(b_n)$ (if $\lim_{n \to \infty} a_n / b_n  = 0$).  Finally, we shall write that an event $\Omega_n$ holds {\em with high probability}, \whp for short, if  $\dP ( \Omega_n^c)  =o(1)$.

\section{Proof of  Theorem \ref{th:Fr}}
\label{sec:proofF}

\subsection{Overview of the proof}
\label{subsec:ovpf}

Let us describe the strategy of proof of Theorem \ref{th:Fr} and its main difficulties.  Following F{\"u}redi and Koml{\'o}s  \cite{MR637828} and Broder and Shamir \cite{BrSh}, a natural strategy is to estimate the trace of a high power of the adjacency matrix. Namely, if we manage to prove that \whp 
\begin{equation}\label{eq:trAk}
\tr  ( A^k )  \leq  d^k +  n \PAR{ 2 \sqrt{d-1} + o(1) }^k. 
\end{equation}
for some even integer $k = k(n)$ such that $k \gg \log n $ then Theorem \ref{th:Fr} would follow. Indeed, from the spectral Theorem, \eqref{eq:trAk} implies that \whp 
$$
\mu_2^k + \mu_n ^k \leq \tr (A^k )  -d^k \leq n   \PAR{ 2 \sqrt{d-1} + o(1) }^k.
$$ 
Therefore, \whp
$$
 \mu_2 \vee |\mu_n| \leq n^{1/k} \PAR{ 2 \sqrt{d-1} + o(1) } = 2 \sqrt{d-1} + o(1),
$$
where the last equality comes from $n^{1/k} = 1 + o(1)$. From Serre \cite{MR1396897},  we note that for any $\veps > 0$, there is a positive proportion of the eigenvalues of $A$ which are larger than $2 \sqrt{d-1} - \veps$. This explains the necessary presence of the factor $n$ on the right-hand side of \eqref{eq:trAk}. Observe also that the entries of the matrix $A^k$ count the number of paths of length $k$ between two vertices. Since $k \gg \log n$, we are interested in the asymptotic number of closed paths of length $k$ when $k$ is much larger than the typical diameter of the graph. 

To avoid the presence of $d^k$ on the right-hand side of \eqref{eq:trAk}, we may project $A$ onto the orthogonal complement of the eigenspace associated to $\mu_1 =d$ and then compute the trace. If  $J$ is the $n\times n$ matrix with all entries equal to $1$, we should then prove that \whp for some even $k$, $k \gg \log n$, 
\begin{equation}\label{eq:trAk2}
\tr (A^k) - d^k  = \tr \PAR{ A - \frac{d}{n} J  } ^k \leq   n \PAR{ 2 \sqrt{d-1} + o(1) }^k.
\end{equation}

The main difficulty hidden behind Friedman's Theorem \ref{th:Fr} is that statements \eqref{eq:trAk}-\eqref{eq:trAk2} do not hold in expectation for $k \gg \log n$. This is due to the presence of subgraphs in the graph which occur with polynomially small probability. For example, it follows from McKay  \cite{MR681916} that for $n$ large enough, the graph contains as subgraph the complete graph with $d+1$ vertices with probability at least $n^{- c}$  for some explicit $c > 0$. On this event, say $\Omega$, the graph is disconnected and $\mu_2 = d$. Hence, for $k$ even, 
$$
\dE \tr (A^k) - d^k = \dE \tr \PAR{ A - \frac{d}{n} J  } ^k \geq d^k \dP ( \Omega) \geq d^k n^{-c}. 
$$
For $k \gg \log n$, the right-hand side is much larger than $n \PAR{ 2 \sqrt{d-1} + o(1) }^k$. The event $\Omega$ is only an example among other unlikely events which prevent statement  \eqref{eq:trAk}-\eqref{eq:trAk2} to hold in expectation, see \cite{MR2437174} for a more detailed treatment of this key issue. In \cite{MR2437174}, the subgraphs which are responsible for the large expectation of the trace are called {\em tangles}. In this paper, we will use a  simpler definition of the word tangle (Definition \ref{def2}).

The proof is organized as follows. First, as in the original Friedman's argument, we will study the spectrum of the non-backtracking matrix $B$ of the graph instead of its adjacency matrix $A$.  Through the Ihara-Bass formula,  the eigenvalues of $A$ and $B$ are related by a quadratic equation. It is easier from a combinatorial viewpoint to count the non-backtracking paths which will appear naturally when taking powers of the matrix $B$. This step will be performed in \S \ref{subsec:NB}: will restate Friedman's Theorem in terms of the second largest eigenvalue of the non-backtracking matrix of the random configuration model.

We will not directly apply the high trace method to $B$. We shall fix some integer $\ell$ of order $\log n$. The second largest eigenvalue of $B$ in absolute value, say $\lambda_2$, satisfies, 
$$
\ABS{\lambda_2}^\ell  \leq \max_{\langle x , \chi \rangle   = 0, x \ne 0} \frac{\| B^\ell x \|_2}{\|x\|_2}, 
$$
where $\chi$  (all entries equal to $1$) will be a common eigenvector of $B$ and $B^*$, its conjugate transpose, associated to their largest eigenvalue. 
We will then use the crucial fact that \whp the graph is free of tangles (forthcoming Lemma \ref{le:tangle}). On this event, we will have the matrix identity
$$
B^{\ell} = B^{(\ell)},
$$
where $B^{(\ell)}$ is the matrix
obtained from $B^\ell$ by discarding non-backtracking walks that encounter a
tangle.  Thanks to elementary linear algebra, we will then project the matrix $B^{(\ell)}$ onto the orthogonal complement of the vector $\chi$ and give a deterministic upper bound  of
 $$
  \max_{\langle x , \chi \rangle   = 0, x \ne 0} \frac{\| B^{(\ell)} x \|_2}{\|x\|_2}
 $$
in terms of the operator norms of new matrices which will be expressed as weighted paths of length at most $\ell$. This step is done in \S \ref{subsec:PD}. It is inspired from  Massouli\'e \cite{M13} and was further developed in \cite{BLM}.

In the remainder of the proof, we will aim at using the  high trace method to upper bound the operator norms of these new matrices of weighted paths of length at most $\ell$: if $C$ is such matrix, we will write
\begin{equation}\label{eq:illC}
\dE \| C \|^{2m}  = \dE \| C C^* \|^m  \leq \dE \tr \PAR{ C C^* }^m   
\end{equation}
for some integer $m$ of order $\log n / \log \log n$. By construction, the expression on the right-hand side is an expected contribution of some weighted paths of lengths $k = 2 m \ell$ of order $(\log n )^2 / \log \log n$.  We will thus haved reached paths of length of size $k = 2 m \ell \gg \log n$ by using an intermediary step where we modify the matrix $B^\ell$ in order that it vanishes on tangles.

The study of the expected contribution of weighted paths in \eqref{eq:illC} will have a probabilistic and a combinatorial part.  The necessary probabilistic computations on the configuration model are gathered in \S \ref{subsec:CM}. We will notably estimate the expectation of a single weighted path of polynomial length thanks to an exact representation in terms of a special function.  In  \S \ref{subsec:PC}, we will use these computations together with combinatorial bounds on non-backtracking paths to deduce sharp bounds on our operator norms. The success of this step will essentially rely on the fact that the contributions of tangles vanish in $B^{(\ell)}$. Finally, in \S  \ref{subsec:end}, we gather all these ingredients to conclude our proof of Theorem \ref{th:Fr}.

\subsection{The non-backtracking matrix of the configuration model}
\label{subsec:NB}

In this subsection, we restate Theorem \ref{th:Fr} in terms of the spectral gap of the non-backtracking matrix. For the forthcoming probabilistic analysis, we define in a slightly unusual way this non-backtracking matrix. It is tuned to the configuration model. This probabilistic model is closely related to the uniform distribution on $\cG_d(n)$ and it is simple enough to allow  explicit computation; we refer to \cite{MR1864966}.  To this end, we define the finite sets
$$
V= [n] \quad \hbox{ and }  \quad \vec E = [n] \times [d].
$$
An element of $V$ will be called a vertex and an element of $\vec E$, a half-edge. The subset of $\vec E$
\begin{equation}\label{eq:defEv}
\vec E(v) = \{v \} \times [d], 
\end{equation}
is thought as a set of half-edges attached to the vertex $v \in V$.
If  $X$ is a finite set of even cardinality, we define $M(X)$ as the set of perfect matchings of $X$, that is permutations $\si$ of $X$ such that  for all $x \in X$, $\si^2 (x)= x$ and $\si(x) \ne x$. If $\sigma \in M(\vec E)$, we can classically associate a multigraph $G = G (\sigma) $, where a {\em multigraph} can have multiple edges between the same pair of vertices and loops, edges that connect a vertex with itself. This multigraph $G(\sigma)$ is defined through its adjacency matrix $A \in M_n (\dR)$, by the formula for all $u, v \in V$,
$$
A_{uv} = A_{vu}  = \sum_{i = 1} ^d \sum_{j = 1} ^d \IND ( \si(u,i) =( v,j) ) = \sum_{i = 1} ^d \IND ( \si(u,i) \in \vec E(v)).
$$
Graphically, $G(\sigma)$ is the multigraph obtained by gluing the half-edges into edges according to the matching map $\sigma$.

Recall that $\cG_d(n)$ is the set of simple $d$-regular graphs on the vertex set $V$. Observe that $G  = G(\sigma)  \in \cG_d (n)$ if and only if for all $u \ne v \in V$, $A_{uu} = 0$ and $A_{uv} \in \{ 0,1\}$, equivalently   $\si(\vec E(u) ) \cap \vec E(u)= \emptyset$ (no loops) and $|\si(\vec E(u) ) \cap \vec E(v) | \in \{0,1\}$ (no multiple edges). It is easy to check that if $\sigma$ is uniformly distributed on $M(\vec E)$ then the conditional probability measure of $G(\sigma)$ given $\{ G(\sigma) \in \cG_d (n)\}$  is the uniform measure on $\cG_d(n)$ (that is, for any $g \in \cG_d (n)$, $
\dP \PAR{ G(\sigma ) = g }$ does not depend on $g$). Importantly, from \cite[Theorem 2.16]{MR1864966}, the following holds 
\begin{equation}\label{eq:Poiloop}
\lim_{n \to \infty} \dP ( G(\sigma ) \in \cG_d (n) ) =  e^{-(d^2-1)/4}.
\end{equation}

The Hashimoto's non-backtracking matrix $B$ of $G$ is an endomorphism of $\dR^{\vec E}$ defined in matrix form, for $e = (u,i)$, $f = (v,j)$ by
$$
B_{ef} = \IND ( \si (e) \in \vec E(v) \backslash \{ f \} ). 
$$
There is an alternative expression for $B$. Let $M = M(\sigma)$ be the permutation matrix associated to $\sigma$, defined for all $e, f \in \vec E$ by
\begin{equation}\label{eq:defM}
M_{e f} = M_ {f e} = \IND ( \sigma(e) =  f).
\end{equation}
Let $N$ be the endomorphism of $\dR^{\vec E}$ defined in matrix form, for $e = (u,i)$, $f = (v,j)$ by 
\begin{equation}\label{eq:defN}
N_{e f}  = N_{fe} = \IND ( u = v ; i \ne j).
\end{equation}
Since $(MN)_{ef} = \sum_g M_{eg} N_{gf}$, we get easily
\begin{equation} \label{eq:defBMN}
B = M N.
\end{equation}

If $m = nd$, we denote by $\lambda_1 \geq |\lambda_2| \geq \ldots \geq |\lambda_m|$ the eigenvalues of $B$ (we index the eigenvalues of $B$ of equal absolute values in an arbitrary way). Let $\chi\in \dR^{\vec E}$ be the vector with all entries equal to $1$. Observe that by construction $N$ is symmetric, $N \chi = (d-1) \chi$ and $M \chi = M^* \chi = \chi$. We deduce that 
\begin{equation}\label{eq:Perron}
B \chi= (d-1) \chi \quad \hbox{ and } \quad B^* \chi= (d-1) \chi, 
\end{equation} 
Hence, the Perron eigenvalue of $B$ is
$$
\lambda_1 = d-1.
$$
The Ihara-Bass formula asserts that if $G \in \cG_d(n)$ and $r= |\vec E| / 2 - |V| =  m / 2 - n$, 
\begin{equation}\label{eq:IharaBass}
\det( I_{\vec E} - B z ) = ( 1  - z^2) ^{r} \det(  I_V - A z + (d-1)z ^2 I_V ), 
\end{equation}
for a proof, we refer to \cite{MR1749978,MR2768284}. We use this formula as a dictionary between the spectra of $A$ and $B$. If $\si(A)$ and $\si(B)$ are the set of eigenvalues of $A$ and $B$, we get 
\begin{equation*}\label{eq:rosetta}
\si(B) = \BRA{\pm 1} \cup \BRA{ \lambda : \lambda^2 - \mu \lambda + ( d-1) = 0 , \mu \in \sigma(A) }.
\end{equation*}
Consequently, it is straightforward to check that if $\mu \in \sigma(A)$ with $|\mu| = 2 \sqrt{d-1} (1+\delta)$ and $\delta \geq 0$, then there exists a real $ \lambda \in \sigma(B)$ with $|\lambda| = \sqrt{d-1} ( 1 + \delta + \sqrt{ \delta ( 2 + \delta)}) \geq \sqrt{d-1} (1 + \sqrt\delta)$.  Hence,  from \eqref{eq:Poiloop}, Theorem \ref{th:Fr} is implied by the following statement.
\begin{theorem}\label{th:FrNB}
Let  $d \geq 3$ be an integer and $nd$ be even. Let $\si$ be uniformly distributed on $M(\vec E)$ and $\lambda_2$ be the second largest eigenvalue of $B$ in absolute value. For any $\veps >0$, 
$$
\lim_{n \to \infty} \dP \PAR{| \lambda_2 | \geq \sqrt{d-1} + \veps }= 0,
$$
where the limit is along any sequence going to infinity with $nd$ even.
\end{theorem}

The remainder of this section is dedicated to the proof of Theorem \ref{th:FrNB}. We remark that in order to prove \eqref{eq:QFried}, we will prove that for any $0 < a < 1$, there exists $c>0$ (depending on $d$ and $a$) such that for all $n \geq 3$ with $nd$ even,
\begin{equation}\label{eq:QFriedB}
\dP \PAR{  |\lambda_2| \geq  \sqrt{d-1} +c\frac{\log \log n }{ \log n}} \leq  n^{-a}.  
\end{equation}

\subsection{Path decomposition}
\label{subsec:PD}

In this subsection, we fix two positive integers $n,d$ and set $\vec E = [n] \times [d]$ as above. Let $\sigma \in M(\vec E)$ be a perfect matching of $\vec E$. We consider the multigraph $G = G(\sigma)$ and its non-backtracking matrix $B = B(\sigma)$ defined in \eqref{eq:defBMN}. Our aim is to derive a deterministic upper bound on  the second eigenvalue of $B$ (in forthcoming Proposition \ref{le:decompBl}).

In the following, we endow $\dC^n$ with the usual inner product and denote by $\perp$ the orthogonal complement. We start with an elementary algebraic lemma.  
\begin{lemma}\label{le:HR}
Let $R, S \in M_n (\dC)$ such that  $\mathrm{im}( S ) \subset \ker (R)$ and $\mathrm{im}(S^*) \subset \ker (R)$ where $S^*$ is the conjugate transpose of $S$. Then, if $\lambda$ is an eigenvalue of $S +R$ and is not an eigenvalue of $S$,
$$
|\lambda| \leq  \max_{x \in \ker (S), x \ne 0} \frac{\|(S + R)x\|_2}{\|x \|_2}.
$$
\end{lemma}

\begin{proof}
If $\lambda \ne 0$ is an eigenvalue of $S+R$ and is not eigenvalue of $S$ then it is an eigenvalue of $R$ (indeed, we have $\det ( S + R - \lambda )  = \det ( S - \lambda) \det ( I +  R (S - \lambda)^{-1} )$  and $R (S- \lambda)^{-1} = -  \lambda^{-1} R $ since $\mathrm{im}( S ) \subset \ker (R)$). Consequently, there exists $x \ne 0$ such that $|\lambda| \leq \| R x \|_2 / \|x \|_2$. Now, we write $x  = y + z$, with $y \in \ker (S)$ and $z \in  \ker (S)^\perp$. Since $\ker (S)^\perp =  \mathrm{im}(S^*)  \subset \ker (R)$, we get $ R x =  R y   = (S + R)y$.  Finally, $\|y \|_2  \leq \|x\|_2$ and we get $|\lambda| \leq \|  (S+ R) y  \|_2 / \| y \|_2$. 
\end{proof}

Consider again the non-backtracking matrix $B = B(\sigma)$. We fix a positive integer $\ell$. From \eqref{eq:Perron}, we may apply Lemma \ref{le:HR} to the symmetric matrix $S  = ( d-1) ^\ell \chi \chi^* / ( n d)$ and $R = B^\ell - S$. We find the inequality, 
\begin{equation}\label{eq:basicl2}
| \lambda_2 | \leq \sup_{x : \langle x , \chi \rangle  = 0 , \| x \|_2 = 1}   \| B^{\ell} x \|_2^{1/ \ell},
\end{equation}
where we recall that $\lambda_2$ is the second largest eigenvalue of $B$ in absolute value. The right-hand side of the above expression can be studied by an expansion of paths in the graph.  We introduce some definitions for the  sequences $ \gamma = (\gamma_1, \ldots, \gamma_{2\ell+1} ) \in \vec E^{2\ell+1}$ which appear when we express the entries of $B^{\ell}$ as a count of non-backtracking paths in the graph, see Figure \ref{fig:Gamma}.

\begin{definition}\label{def1}
For a positive integer $k$, let $ \gamma = (\gamma_1, \ldots, \gamma_{k} ) \in \vec E^{k}$, with $\gamma_t = ( v_t, i_t)$. 
\begin{enumerate}[-]
\item 
We define the set of {\em visited vertices} and of {\em unordered pairs of half-edges} of $\gamma$ to be, respectively, the sets  $V_\gamma = \{ v_{t}  :    t\in [k] \}$ and $\cE_\gamma = \{ \{ \gamma_{2t-1}, \gamma_{2t} \}  : 1 \leq t \leq k/2\}$. We denote by $G_\gamma$ the multigraph with vertex set $V_\gamma$ and edges given  by $\cE_\gamma$: where each element $ \{ (u,i) , (v,j) \} \in \cE_\gamma$ is viewed as an edge in $G_\gamma$ between $u$ and $v$.  
\item
The sequence $\gamma \in \vec E^{k}$ is a {\em non-backtracking path} if for all $t \geq 1$, $v_{2t+1} = v_{2t}$ and $\gamma_{2t+1} \ne \gamma_{2t}$ (that is $N_{\gamma_{2t} \gamma_{2t+1}} = 1$, where $N$ was defined by \eqref{eq:defN}). If $k = 2\ell +1$, the subset of non-backtracking  paths in $\vec E^{2\ell+1}$ is denoted by $\Gamma^{\ell}$.  If $e, f \in \vec E$, we denote by $\Gamma^{\ell}_{ef}$   paths in $\Gamma^{\ell}$ such that $\gamma_1 = e$, $\gamma_{2\ell +1} = f$, and similarly for $\vec E^{k} _{ef}$. 

\end{enumerate}
\end{definition}

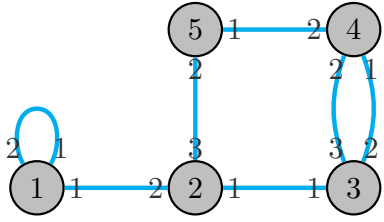
\begin{figure}[htb]
\begin{center}  
\resizebox{5.5cm}{!}{
\begin{tikzpicture}[main node/.style={circle, draw , fill = lightgray, text = black, thick}]
\node[main node]  at (0,0) (1) {1} ;
\node[main node] at (2,0) (2) {2} ;
\node[main node] at (4,0) (3) {3} ;
\node[main node] at (2,2 ) (5) {5} ;
\node[main node] at (4,2) (4) {4} ;

\draw[cyan, -,ultra thick] (1) to (2) ; 
\draw[ cyan, -,ultra thick] (2) to (3) ; 
 \draw[ cyan, -,ultra thick] [out = 70 , in = -70] (3) to (4) ; 
\draw[cyan, -,ultra thick] [out = 110 , in = -110]  (3) to (4) ;
 \draw[cyan,  -, ultra thick] (4) to (5) ;
 \draw[cyan,  -, ultra thick] (5) to (2) ;  
  \draw[ cyan, -,ultra thick] [out = 60 , in = 0] (1) to (0,1)[out = 180 , in = 120] to (1)  ;

  \node[text = black!80] at (0.3,0.5) {1} ; 
  \node[text = black!80] at (-0.3,0.5) {2} ; 
  \node[text = black!80] at (0.5,0) {1} ; 
  \node[text = black!80] at (1.5,0) {2} ; 
  \node[text = black!80] at (2.5,0) {1} ; 
  \node[text = black!80] at (3.5,0) {1} ; 
  \node[text = black!80] at (4.22,0.5) {2} ; 
 \node[text = black!80] at (4.22,1.5) {1} ; 
  \node[text = black!80] at (3.78,0.5) {3} ; 
 \node[text = black!80] at (3.78,1.5) {2} ; 
  \node[text = black!80] at (3.5,2) {2} ; 
  \node[text = black!80] at (2.5,2) {1} ; 
  \node[text = black!80] at (2,1.5) {2} ; 
  \node[text = black!80] at (2,0.5) {3} ; 

\end{tikzpicture}
}
{\small$$ \gamma =   (1,1) (1,2) (1,1)   (2,2)  (2,1)  ( 3,1)  (3,2)   (4,1)  (4,2)  (3,3)  (3,2) (4,1) (4,2) (5,1) (5,2) (2,3) (2,1)  (3,1)   $$} 
\vspace{-25pt}
\caption{A non-backtracking path $\gamma \in \vec E^{18}$  and its associated graph $G_\gamma$. We have $V_\gamma = [5]$ and $\cE_\gamma = \{\{ (1,1)(1,2) \}, \{ (1,1)(2,2) \} ,\{ (2,1)(3,1) \},\{ (3,2)(4,1)\} , \{ (4,2)(3,3) \},  \{ (4,2)(5,1) \}, \{ (5,2)(2,3) \}\}$.} \label{fig:Gamma}

\end{center}\end{figure}

 We use the convention that a product over an empty set is equal to $1$ and the sum over an empty set is $0$. By construction, since $B = MN$, where $N$ was defined in \eqref{eq:defN}, we find that
$$
(B^{\ell}) _{e f} =  \sum_{\gamma \in \vec E^{2\ell+1} _{e f}} \prod_{s=1}^{\ell} M_{\gamma_{2s-1} \gamma_{2s}} N_{\gamma_{2s} \gamma_{2s +1} } = \sum_{\gamma \in \Gamma^{\ell} _{e f}} \prod_{s=1}^{\ell} M_{\gamma_{2s-1} \gamma_{2s}},
$$
where $M =M(\sigma)$ is the permutation matrix associated to $\sigma$  defined by \eqref{eq:defM}.  Note that, in the above expression for $B^\ell$, we have pulled apart the configuration and combinatorial parts: the set $\Gamma^{\ell}$ does not depend on $\sigma$, only the summand depends on it.  We set 
\begin{equation}\label{eq:defW}
\underline M_{ef } = M_{ef}- \frac 1 {dn}.
\end{equation}
Observe that $\underline M$ is the orthogonal projection of $M$ onto $\chi^\perp$. Also, $N$ is symmetric and $\chi$ is an eigenvector, so it preserves $\chi^\perp$. Hence, setting  $\underline B = \underline M N$, we get from \eqref{eq:defBMN} that, if $x \in \chi^\perp$, 
\begin{equation}\label{eq:BBBk}
B^\ell x = \underline B^\ell x. 
\end{equation}
Moreover, using $(\underline B )^\ell = (\underline M N)^\ell$, we find
\begin{equation}\label{eq:defdefub}
(\underline B^{\ell}) _{e f} = \sum_{\gamma \in \Gamma^{\ell} _{e f}} \prod_{s=1}^{\ell} \underline M_{\gamma_{2s-1} \gamma_{2s}}.
\end{equation}
The matrix $\underline B$ will not be used in our analysis. As pointed in \S \ref{subsec:ovpf},  there are events of polynomially small probability which have a dominant influence on the expected value of $B^\ell$ or $\underline B^\ell$. We will first reduce  the above sum over $\Gamma^{\ell} _{e f}$ to a sum over a smaller subset. We will only afterward project onto $\chi^\perp$. This will create some extra remainder terms.

In the following definition, a {\em neighborhood
of radius $\ell$} in a multigraph is the subgraph spanned  by vertices at graphical distance at most $\ell$ from
some fixed vertex. Following \cite{MR2437174,BL,MNS,BLM}, we introduce a central definition. 
\begin{definition}\label{def2}
A multigraph $H$ is {\em tangle-free} if it contains at most one cycle  (loops and multiple edges count as cycles); $H$ is {\em $\ell$-tangle-free} if every neighborhood of radius $\ell$ in $H$ contains at most one cycle. Otherwise, $H$ is tangled or $\ell$-tangled. We say that $\gamma \in \vec E^k$ is tangle-free or tangled if $G_\gamma$ is. Finally,  we use $F^{\ell}$ and $F^{\ell} _{e f}$ to respectively  denote the subsets of tangle-free paths in $\Gamma^{\ell}$ and $\Gamma^{\ell}_{ef}$.
\end{definition}

For example, the path $\gamma$ in Figure \ref{fig:Gamma} is tangled. The following matrices  play a central role in the following analysis. 

\begin{definition}
For each integer $\ell \geq 1$, we introduce the matrices $B^{(\ell)} $ and $\underline B^{(\ell)}$ in $\dR^{\vec E}$, defined for all $e,f \in \vec E$ by 
\begin{eqnarray}\label{eq:defD1}
(B^{(\ell)} ) _{ef} &= &\sum_{\gamma \in F^{\ell} _{e f}} \prod_{s=1}^{\ell} M_{\gamma_{2s-1} \gamma_{2s}}\\
\label{eq:defD}
( \underline B^{(\ell)}) _{ef} &=& \sum_{\gamma \in F^{\ell} _{e f}}   \prod_{s=1}^{\ell} \underline M_{\gamma_{2s-1} \gamma_{2s}}.
\end{eqnarray}
For $\ell = 0$, $B^{(\ell)}= \underline B^{(\ell)}$ equal to the identity matrix. 
\end{definition}

Obviously, if $G$ is $\ell$-tangle-free then
\begin{equation}\label{eq:BkBk}
B^{\ell}   = B^{(\ell)}.
\end{equation}

Beware that it is not true that if $G$ is $\ell$-tangle-free  and $ \ell \geq 3$ then $\underline B^\ell = \underline B^{(\ell)}$ (in \eqref{eq:defdefub} and \eqref{eq:defD} the summand is the same but the sum in \eqref{eq:defdefub} is over a larger set). Nevertheless, as in \eqref{eq:BBBk}, we will now express $B^{(\ell)}x$ in terms of $\underline B^{(\ell)}x$ for all $x \in \chi^\perp$ plus some extra terms, culminating in \eqref{eq:decompBkx}. We start with the following telescopic sum decomposition:
\begin{lemma}Let $\ell$ be a positive integer. For any $e,f \in \vec E$, we have \begin{eqnarray}\label{eq:iopl}
(B^{(\ell)}) _{e f} &  =  & (\uB^{(\ell)} )_{ef}  + \frac{1}{dn} \sum_{k = 1}^\ell \sum_{\gamma \in F^{\ell} _{e f}}    p_k (\gamma),
\end{eqnarray}
where for all   $k \in[ \ell]$ and $\gamma \in \vec E^{ 2\ell+1}$ we have set
$$
p_k (\gamma) = p_k (\gamma ,\sigma) = \PAR{ \prod_{s=1}^{k-1} \underline M_{\gamma_{2s-1} \gamma_{2s}}} \PAR{\prod_{s = k+1}^\ell M_{\gamma_{2s-1} \gamma_{2s}}}.
$$
\end{lemma}
\begin{proof}
In \eqref{eq:defD1}, for all $\gamma$ in $F^{\ell} _{e f}$, we apply the identity, 
$$
\prod_{s=1}^\ell x_s = \prod_{s=1}^\ell y_s  + \sum_{k=1}^{\ell}\PAR{ \prod_{s=1}^{k-1} y_s } ( x_k - y_k) \PAR{\prod_{s = k+1}^{\ell} x_s}.
$$
to $x_s  = M_{\gamma_{2s-1} \gamma_{2s}}$ and $y_s =  \underline M_{\gamma_{2s-1} \gamma_{2s}}$. Since $x_s - y_s = 1/ (dn)$, it gives \eqref{eq:iopl}.
\end{proof}
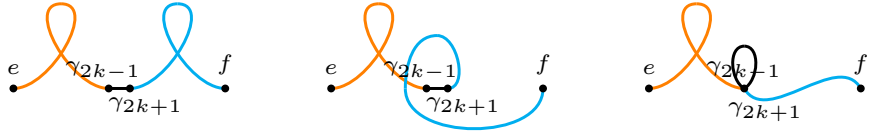
\begin{figure}[htb]
\begin{center}  
\resizebox{12cm}{!}{
\begin{tikzpicture}[main node/.style={circle,fill , text = black, thick}]

\draw[orange , -,thick] (0.0,0)  [out = 0, in = 0] to (0.45,0.8)  ; 
\draw[orange , -,thick] (0.45,0.8) [out = 180 , in = 180]   to (0.9,0)   ; 
\draw[ -,thick] (0.9,0) to (1.1,0)  ; 
\draw[ cyan,  -,thick] (1.1,0) [out = 0, in = 0]  to (1.55,0.8)  ; 
\draw[ cyan , -,thick] (1.55,0.8) [out = 180 , in = 180]  to (2,0) ;  

\node[above]  at (0,0) (1) {\tiny{$e$}} ;
\node[above]  at (2,0) (2) {\tiny{$f$}} ;
\node[above]  at (0.85,-0.05) (3)  {\tiny{$ \gamma_{2k-1}$}} ;
\node[below]  at (1.25,0.05) (4) {\tiny{$\gamma_{2k+1}$}} ;

\draw [fill] (0,0) circle [radius=0.03] ;
\draw [fill] (0.9,0) circle [radius=0.03] ;
\draw [fill] (1.1,0) circle [radius=0.03] ;
\draw [fill] (2,0) circle [radius=0.03] ;

\draw[orange , -,thick] (3.0,0)  [out = 0, in = 0] to (3.45,0.8)  ; 
\draw[orange , -,thick] (3.45,0.8) [out = 180 , in = 180]   to (3.9,0)   ; 
\draw[ -,thick] (3.9,0) to (4.1,0)  ; 
\draw[ cyan,  -,thick] (4.1,0) [out = 0, in = 0]  to (4,0.5)  ; 
\draw[ cyan , -,thick] (4,0.5) [out = 180 , in = 90]  to (3.7,0) ;  
\draw[ cyan , -,thick] (3.7,0) [out = -90 , in = -90]  to (5,0) ;  

\node[above]  at (3,0) {\tiny{$e$}} ;
\node[above]  at (5,0)  {\tiny{$f$}} ;
\node[above]  at (3.85,-0.05)  {\tiny{$ \gamma_{2k-1}$}} ;
\node[below]  at (4.25,0.05)  {\tiny{$\gamma_{2k+1}$}} ;

\draw [fill] (3,0) circle [radius=0.03] ;
\draw [fill] (3.9,0) circle [radius=0.03] ;
\draw [fill] (4.1,0) circle [radius=0.03] ;
\draw [fill] (5,0) circle [radius=0.03] ;

\draw[orange , -,thick] (6.0,0)  [out = 0, in = 0] to (6.45,0.8)  ; 
\draw[orange , -,thick] (6.45,0.8) [out = 180 , in = 180]   to (6.9,0)   ; 
\draw[ -,thick] (6.9,0)[out = 30, in = 0]  to (6.9,0.4)  ;
 \draw[ -,thick] (6.9,0.4)[out = 180, in = 150]  to (6.9,0)  ;
\draw[ cyan,  -,thick] (6.9,0) [out = -60, in = -240]  to (8,0) ;  

\node[above]  at (6,0) {\tiny{$e$}} ;
\node[above]  at (8,0)  {\tiny{$f$}} ;
\node[above]  at (6.9,-0.05)  {\tiny{$ \gamma_{2k-1}$}} ;
\node[below]  at (7.1,0.0)  {\tiny{$\gamma_{2k+1}$}} ;

\draw [fill] (6,0) circle [radius=0.03] ;
\draw [fill] (6.9,0) circle [radius=0.03] ;
\draw [fill] (8,0) circle [radius=0.03] ;

\end{tikzpicture}}
\caption{Tangle-free paths whose union is tangled.} \label{fig:Gamma3}
\end{center}\end{figure}

\vspace{-10pt}

We now rewrite \eqref{eq:iopl} as a sum of matrix products for lower powers of $\uB^{(k)}$ and $B^{(k)}$ up to some remainder terms.  Fix $k \in [\ell]$, we decompose a path $\gamma  = (\gamma_1, \ldots, \gamma_{2\ell +1} ) \in \Gamma^{\ell}$ as a path $\gamma'=  (\gamma_1, \ldots, \gamma_{2k -1} ) \in  \Gamma ^{k-1}$, a path $\gamma''= (\gamma_{2k-1}, \gamma_{2k} , \gamma_{2k+1} ) \in \Gamma^{1}$ and a path $\gamma''' = (\gamma_{2k+1}, \ldots, \gamma_{2 \ell +1}) \in \Gamma^{\ell - k}$. If the path $\gamma$ is in $F^{\ell}$ (that is, it is tangle-free), then  the three paths are tangle-free, but the converse is not necessarily true, see Figure \ref{fig:Gamma3}. This will be the origin of the remainder terms.  For each $k\in [\ell]$, we denote by $F^{\ell}_{k}$ the set of $\gamma \in \Gamma^{\ell}$ such that, with $\gamma', \gamma'', \gamma'''$ as above, $\gamma'  \in F^{k-1}$, $\gamma'' \in F^{1} = \Gamma^1$ and $\gamma''' \in F^{k-\ell}$. For $e,f \in \vec E$, let $F^\ell_{k,ef} = F^\ell_{k} \cap \vec E^{2\ell+1}_{ef}$. We have the inclusion $F^{\ell}  \subset F^{\ell}_{k}$. We write in \eqref{eq:iopl}
\begin{eqnarray*}
\sum_{\gamma \in F^{\ell} _{e f}}   p_k(\gamma)  = \sum_{\gamma \in F^{\ell} _{k,e f}} p_k(\gamma) - \sum_{\gamma \in F^{\ell} _{k,e f} \backslash F^{\ell} _{e f}}  p_k(\gamma).
\end{eqnarray*}

We observe that the cardinality of $\Gamma^1_{ef} = F^1_{ef}$ is $d-1$. If $\chi^*$ is the conjugate transpose of $\chi$, $\chi \chi^*$ is the matrix on $\dR^{\vec E}$ with all entries equal to $1$. The rule of matrix multiplication gives 
\begin{eqnarray*}
 \sum_{\gamma \in F^{\ell} _{k,e f}}  p_k(\gamma) & = &  \sum_{a,b \in \vec E} \sum_{\gamma'\in F^{k-1}_{ea}, \gamma'' \in  F^{1}_{ab}, \gamma'''\in F^{\ell-k}_{bf} }  \PAR{ \prod_{s=1}^{k-1} \underline M_{\gamma'_{2s-1} \gamma'_{2s}}}  \PAR{\prod_{s = 1}^{\ell-k} M_{\gamma'''_{2s-1} \gamma_{2s}}} \\
 & = & (d-1)( \uB^{(k-1)} \chi \chi^*  B^{(\ell - k)} )_{ef}. 
\end{eqnarray*}

For each $k \in [\ell]$, we introduce the matrix in $\dR^{\vec E}$, defined for all $e,f \in \vec E$ by 
\begin{equation}\label{eq:defR}
(R^{(\ell)}_{k} )_{ef}   =  \sum_{\gamma \in F^{\ell} _{k,e f} \backslash F^{\ell} _{e f}} p_k(\gamma).
\end{equation}

We deduce from \eqref{eq:iopl} that \begin{eqnarray*}
B^{(\ell)}   &= & \uB^{(\ell)}   +    \frac {d-1} {dn}  \sum_{k = 1} ^{\ell}  \uB^{(k-1)} \chi \chi^* B^{(\ell - k)}  -   \frac 1 {dn}     \sum_{k = 1}^{\ell} R^{(\ell)}_{k} .
\end{eqnarray*}

Observe that if $G$ is $\ell$-tangle free, then it is also $k$-tangle free for all $k \in[ \ell]$. Hence, from \eqref{eq:Perron} and \eqref{eq:BkBk}, we find $\chi^* B^{(\ell - k)}  = \chi^* B^{\ell - k}  =  (d-1)^{\ell -k} \chi^*$. Consequently, if $G$ is $\ell$-tangle free  and $\langle x , \chi \rangle = 0$,  we find
\begin{eqnarray}
 B^{\ell} x &   = &   \uB^{(\ell)} x    -  \frac 1{d n}   \sum_{k = 1}^\ell  R^{(\ell)}_{k} x.\label{eq:decompBkx}
\end{eqnarray}
We use the triangle inequality to estimate $\| B^{\ell} x \|_2$. From \eqref{eq:basicl2}, we deduce the main result of this subsection, which is the following
proposition. 

\begin{proposition}\label{le:decompBl}
Let $\ell \geq 1$ be an integer and $\sigma \in M(\vec E)$ be such that $G(\sigma)$ is $\ell$-tangle free. Then, if $\lambda_2$ is the second largest eigenvalue of the non-backtracking operator $B = B(\sigma)$, we have
$$
|\lambda_2 | \leq  \PAR{\|  \uB^{(\ell)} \|   +  \frac 1{d n}   \sum_{k = 1}^\ell \| R^{(\ell)}_{k} \|} ^{1/\ell}.$$
\end{proposition}

\subsection{Computation on the configuration model}

\label{subsec:CM}
The configuration model allows some explicit probabilistic computation.  In the remainder of this section, $\sigma$ is uniformly distributed on $M(\vec E)$, the set of matchings on $\vec E = [n] \times [d]$ and $G = G(\sigma)$ is the corresponding multigraph.  The next lemma states that $G$ is $\ell$-tangle free if $\ell$ is not too large. It is an already  known fact, see \cite[Lemma 2.1]{MR2667423}, it can also be extracted from \cite{MR2097332}. We give a proof for completeness.

\begin{lemma}\label{le:tangle}
Let $d \geq 3$ and $\ell$ be positive integers. Let $\sigma$ be uniformly distributed on $M(\vec E)$ with $\vec E = [n] \times [d]$. Then $G = G(\sigma)$ is $\ell$-tangle free with probability $1 - O ((d-1)^{4\ell} / n)$.
\end{lemma}

\begin{proof} We fix $v \in V$ and build a process which sequentially reveals the neighborhood of $v$. Informally, this process  updates a set $D_t$ of half-edges $e \in \vec E$ which are in the neighborhood of $v$ but whose matched half-edge $\sigma(e)$ has not been revealed yet. We start by the half-edges in $\vec E(v)$ (see \eqref{eq:defEv}) and then the half-edges which share a vertex  with an half-edge in $\sigma(\vec E(v))$ and so on. Formally, at stage $0$, we set $D_0 = \vec E(v)$.  At stage $t \geq 0$, if $D_t$ is not empty, take an element $e_{t+1}$ in $D_t$ which has been added at the earliest possible stage (we break ties with lexicographic order). Let $f_{t+1} = \si(e_{t+1}) = (u_{t+1},j_{t+1})$.  If  $f_{t+1}  \in D_{t}$, we set $D_{t+1} = D_t \backslash \{ e_{t+1}, f_{t+1} \}$, and, otherwise,  $$D_{t+1} =  \PAR{D_t \cup \vec E(u_{t+1} )  } \backslash \{ e_{t+1} , f_{t+1}\}.$$
 At some stage $\tau \leq d n$, $D_\tau$ is empty, and we have explored the connected component of $v$. Before stage 
$$ T = \sum_{k = 1}^{\ell -1} d (d-1)^{k-1} = O \PAR{ (d-1)^\ell} ,$$
 we have revealed the subgraph spanned by the vertices at distance at most  $\ell$ from $v$. Also, if $v$ has two distinct cycles in its $\ell$-neighborhood, then
$
S(v) = S_{\tau \wedge T}  \geq 2, 
$
where, for $t \geq 1$,
$$
S_t =  \sum_{s= 1}^{t} \veps_{s} \quad   \hbox{ and } \quad \veps_t = \IND ( f_{t} \in D_{t-1} ). 
$$
At stage $t \geq 0$, $2t$ values of $\sigma$ have been discovered (namely $\sigma(e_s)$ and  $\sigma(f_s)$ for $1 \leq s \leq t$) and 
$$
|D_t | = d + \sum_{s=1}^t (d-2)( 1 - \veps_s)  - 2 \sum_{s=1}^t \veps_s =d (t +1) - 2 t -d S_t. 
$$

Denote by $\cF_t$ the $\sigma$-algebra generated by $(D_0, \cdots, D_{t})$ and by $\dP_{\cF_t}$ the conditional probability distribution. Then, $\tau$ is a stopping time and, if $t < \tau \wedge T$, for some constant $c >0$, 
$$
\dP_{\cF_t} ( \veps_{t+1} = 1) = \frac{|D_t| - 1 }{nd - 2 t  -1}   \leq c \PAR{ \frac{ T }{n} } = q. 
$$
Hence, for any integer $k$, $\dP( S(v) \geq k)$ is at most the probability of $[k,\infty)$ for $\Bin(T,q)$, the binomial distribution with parameters $(T,q)$. The probability that $\Bin(T,q)$ is at least $k$ is at most $q^k { T \choose k} \leq q^k T^k$ . In particular, from the union bound,  
$$
\dP \PAR{ G \hbox{ is $\ell$-tangled}}\leq \sum_{v=1} ^n  \dP ( S(v) \geq 2)  \leq \sum_{v=1} ^n  q^2 T^2   = O \PAR{ \frac{ (d-1)^{4 \ell}  }{n}}.
$$
This concludes the proof of Lemma \ref{le:tangle}. 
\end{proof}

The next crucial proposition gives a precise estimate on the fact that the variables $\underline M_{ef}$ defined by \eqref{eq:defW} are weakly dependent. They are also approximately centered since for $e\ne f$, $\dE \underline M_{ef}= 1 / (dn-1)- 1 / (dn) = O (1 / (dn)^2)$.  We first introduce some new definitions which extend Definition \ref{def1}.
\begin{definition}\label{def:defcons}
Let $k$ be a positive integer, $\gamma = ( \gamma_1, \cdots , \gamma_{2k}) \in \vec E^{2k}$ and $E_\gamma = \{ \{ \gamma_{2t-1}, \gamma_{2t} \} :  t \in [k]\}$ its visited edge set. 

\begin{enumerate}[-]
\item
The {\em multiplicity} of an half-edge $e \in \vec E $ is $m_\gamma (e) = \sum_{t=1}^{2k} \IND ( \gamma_t  = e )$. 

\item
The {\em multiplicity} of an edge $\{e, f\} \in \cE_\gamma$ is $m_{\gamma} ( \{ e,f\} )  = \sum_{t=1}^k  \IND ( \{ \gamma_{2t-1}, \gamma_{2t}\} = \{e, f\} )$. 

\item
An edge  $\{e, f\} \in \cE_\gamma$ is {\em consistent} if  $m_\gamma (e) = m_{\gamma} (f) = m_{\gamma} (\{ e, f \})$.  It is inconsistent otherwise. 
\end{enumerate}
\end{definition}

Equivalently, an edge $\{e,f\}$ of $\gamma$ is consistent if its half-edges are distinct and if they are uniquely paired together: that is, $e \ne f$ and $\{ t : e \in \{ \gamma_{2t-1}, \gamma_{2t}  \} \} =  \{ t : f \in \{ \gamma_{2t-1}, \gamma_{2t}  \} \}  = \{ t : \{ e , f\} =\{ \gamma_{2t-1}, \gamma_{2t}  \} \} $. For example, in Figure \ref{fig:Gamma}, the edges $\{ (1,1)(1,2) \}, \{ (1,1)(2,2) \}$, $\{(4,2),(3,3)\}$, $\{(4,2),(5,1)\}$ are inconsistent. We are now ready to state the main result of this subsection. 
\begin{proposition}
\label{le:exppath}
Let $\vec E = [n] \times [d]$ with $nd$ an even positive integer. Let $\sigma$ be uniformly distributed on $M(\vec E)$. There exists a universal constant $c >0$ such that for any $ \gamma \in \vec E^{2k}$ with $1 \leq k \leq \sqrt{dn}$ and any $0 \leq k_0 \leq k$, we have, 
$$
\ABS{ \dE \prod_{t= 1} ^{k_0}  \underline M_{\gamma_{2t-1} \gamma_{2t} } \prod_{t= k_0+1} ^{k}  M_{\gamma_{2t-1} \gamma_{2t} }  } \leq c\, 2^{b} \PAR{ \frac{1 }{ dn }}^{a} \PAR{ \frac{  3k   }{ \sqrt{dn} }}^{a_1}, 
$$ 
where $a = | \cE_\gamma |$,  $b$  is the number of  $t \in [k_0]$ such that $\{ \gamma_{2t-1}, \gamma_{2t}\}$ is an inconsistent edge of multiplicity $1$ in $\cE_\gamma$,  and $a_1$ is the number of  $t \in [k_0]$ such that $\{ \gamma_{2t-1}, \gamma_{2t}\}$ is a consistent edge of multiplicity $1$ in $\cE_\gamma$.\end{proposition}

The important part in Proposition \ref{le:exppath} is the factor $( 3k   / \sqrt{dn} )^{a_1}$. It reflects  that the variables $\underline M_{ef}$ are nearly centered and weakly dependent when $k = o (\sqrt{dn})$. We will use the Pochhammer symbol, for non-negative integers $n, k$,
$$
(n)_k = \prod_{t = 0}^{k-1} (n -t).  
$$
Recall the convention that a product over an emptyset is equal to $1$.   We start with a technical lemma which bounds an expression which will be closely related to the expectation of product of distinct $\underline M_{ef}$ in  the proof of Proposition \ref{le:exppath}.
\begin{lemma}\label{le:bin}
Let $z\geq 1$, $k \geq 1$ be an integer, $0 < p , q < 1$ and $N$ be a $\BIN ( k , p)$ variable. If  $4 ( 1 -  p/ (q ( 1 - p))) ^2   \leq  zq k ^2 \leq 1$, we have
$$
\ABS{ \dE \prod_{n=0}^{N-1} \PAR{ z n -  \frac 1 q } } \leq  4 \PAR{  3 k   \sqrt{  z q /2}} ^k. 
$$
\end{lemma}
 
\begin{proof}
Let $f(x) =  \dE \prod_{n=0}^{N-1} \PAR{ z n -  1/x  }$,  $\delta  = - z q $ and $\veps = 1 -  p/ (q ( 1 - p))$.  By assumption, we have  $| \veps | \leq 1/2$. We write
$$
f(q) = \sum_{t = 0} ^k {k \choose t} p^{t} ( 1- p)^{k-t}  \prod_{n=0}^{t-1} \PAR{z n -  \frac 1 q  }  = \PAR{ 1 - p}^k   \sum_{t = 0} ^k {k \choose t} ( -1 + \veps)^{t}  \prod_{n=0}^{t-1} \PAR{ 1 + \delta n},
$$
(we note that the function $f$ can be expressed in terms of the confluent hypergeometric function $U(a,b,z)$, for definition, see  \cite[(13.2.7)]{DLMF},  the lemma is then a consequence of a known asymptotic in \cite[Chapter 13]{DLMF}. We will however give a full  proof).   We write
$$
  \prod_{n=0}^{t-1} \PAR{ 1 + \delta n} = 1 + \sum_{s=1}^{t-1} \delta^s \sum_{(s)}  \prod_{i=1}^s n_i  = 1 +  \sum_{s=1}^{t-1} \delta^s P_s(t),
$$
where $\sum_{(s)}$ is the sum over all $(n_i)_{1\leq i \leq s}$ all distinct and $1 \leq n_i \leq t-1$. We observe that $t \mapsto P_s(t)$ is a polynomial of degree $2s$ in $t$. Moreover, $P_s$ vanishes at integers $0 \leq t \leq s$ and for all integers $t \geq s+1$, we have 
\begin{equation*}\label{eq:bdPst}
0 \leq P_s (t) \leq \PAR{ \sum_{n=1}^{ t-1} n }^s \leq \PAR{ \frac {t^{2}}{  2}  }^ s.
\end{equation*}
Setting $P_0 (t) =1$, we get
\begin{eqnarray}\label{eq:fpabs}
|f(q)| \leq   \sum_{s=0}^{k-1}   |\delta|^s  \ABS{ \sum_{t=0} ^k  {k \choose t} \PAR{ -  1 + \veps}^{t}  P_s(t)}.
\end{eqnarray}

We will use some cancellations in the above sum. Indeed, consider the derivative of order $m$ of
$ ( 1+ x)^k = \sum_{t=0}^{k} {k \choose t} x^t$.  It vanishes at $x = -1$ for any $0 \leq m \leq k-1$. We get that for any $0 \leq m \leq k-1$,
$$
0 =  \sum_{t=0}^{k} {k \choose t} (-1)^t (t)_m.
$$
Since $Q_m ( x )  = (x)_m$ is a monic polynomial of degree $m$, the family $(Q_0, \ldots , Q_{k-1})$ is a basis of $\dR_{k-1}[x]$, the real polynomials of degree  at most $k-1$. Hence, by linearity that for any  $P \in \dR_{k-1} [x]$, 
\begin{equation}\label{eq:PPPP}
0 =  \sum_{t=0}^{k} {k \choose t} (-1)^t P(t).
\end{equation}

Since $P_s$ is a polynomial of degree $2s$, \eqref{eq:PPPP} can be used to cancel some terms in \eqref{eq:fpabs}. First, since $| \veps | \leq 1/2$, for any $s\geq 0$, we have 
\begin{eqnarray*}
 \ABS{ \sum_{t=0} ^k  {k \choose t}  \PAR{ -  1  +  \veps}^t     P_s(t)} \leq    \sum_{t=0}^k {k \choose t} \PAR{ 3/2 }^t k^{2s}2 ^{-s}   =     (5/2) ^{k} k^{2s}2 ^{-s}
\end{eqnarray*}
where we have used that  $\sum_{t=0}^k {k \choose t} (3/2)^t = (5/2)^k$ and $\ABS{  P_s(t)} \leq (k^{2}/ 2) ^s$. It follows that \begin{eqnarray*}
I & = & \sum_{s =  \lfloor \frac{k -1}{2} \rfloor +1 }^{k-1}   | \delta| ^s  \ABS{  \sum_{t=0} ^k  {k \choose t} \PAR{ -1 +  \veps}^t   P_s(t)}\\
 & \leq & \PAR{\frac 5 2   }^k  \sum_{s =  \lfloor \frac{k -1}{2} \rfloor +1 }^{k-1}  \PAR{ \frac{| \delta|  k ^2}{2} }^s  \\
& \leq & 2   \PAR{\frac 5 2  }^k  \PAR{ \frac{| \delta|  k ^2}{2} }^{k/2},
\end{eqnarray*}
where we have used that $|\delta| k ^2 /2 = z q k ^2 / 2 \leq 1/2$ and $\sum_{s \geq r}  x^k \leq 2x^r$ if $0 \leq x \leq 1/2$. We get 
\begin{equation}\label{eq:boundI}
I \leq 2 \PAR{ ( 5 / \sqrt 8)   k \sqrt{z q}}^k.
\end{equation}

For integer $0 \leq s \leq    (k -1)/2 $, we may  exploit \eqref{eq:PPPP} as follows. We use again the binomial identity
$$
\PAR{1 - \veps}^t  = \sum_{r=0}^{t} (-\veps)^r {t \choose r} =  T_{k,  s} (t) + R_{k,s} (t),
$$
where $T_{k,s} (t) = \sum_{r=0}^{k-1 - 2s} (-\veps)^r {t \choose r}$ is a polynomial in $t$ of degree $k-1 - 2s $. Using $ | \veps| \leq 1/2 \leq 1$, we find
$$
|R_ {k,s} (t)|   =    \ABS{ \sum_{r=k - 2s}^{t}   (-\veps)^r  {t \choose r} } \leq  |  \veps|^{k-2s}  \sum_{r=k - 2s}^{t}   {t \choose r}     \leq |  \veps| ^{k - 2 s}2^{t} .
$$
Moreover,  from \eqref{eq:PPPP},   for all integers $0 \leq  s \leq   (k -1)/2  $,
\begin{eqnarray*}
 \sum_{t=0} ^k  {k \choose t} (-1+\veps)^{t}  P_s(t) =    \sum_{t=0} ^k  {k \choose t} ( -1) ^t  R_{k,s} (t) P_s(t).
\end{eqnarray*}
Hence, since $|P_s(t)| \leq k ^{2s} 2^{-s}$, for all integers $0 \leq  s \leq  (k -1)/2 $,
$$
 \ABS{ \sum_{t=0} ^k  {k \choose t} (-1+\veps)^{t}  P_s(t)} \leq  \sum_{t=0} ^k  {k \choose t}   |  \veps| ^{k - 2 s} 2^{t} k^{2s} 2^{-s}  = 3^{k}  | \veps| ^{k - 2 s}  k^{2s} 2^{-s}.
$$
We deduce that 
\begin{eqnarray*}
J    &= &   \sum_{s=0}^{ \lfloor \frac{k -1}{2} \rfloor}     |\delta|^s  \ABS{  \sum_{t=0} ^k  {k \choose t} (-1+\veps)^{t}  P_s(t)} \\
& \leq &    \PAR{ 3  |\veps|}^{k}  \sum_{s=0}^{ \lfloor \frac{k -1}{2} \rfloor}   \PAR{ \frac{|\delta| k^2}{ 2 \veps^2}} ^s \\
& \leq  & 2 \PAR{ 3 |\veps|}^{k}   \PAR{ \frac{|\delta| k^2}{  2 \veps^{2} }} ^{\frac k 2 }, \label{eq:2rfr}
\end{eqnarray*}
where at the last step, we use the assumption that $|\delta| k^2 / ( 2 \veps^2) \geq 2 $ and $\sum_{s=0}^r x^s \leq 2 x^r$ if $x \geq 2$. So finally, $J$ is  bounded by $2  \PAR{  (3/\sqrt 2) k  \sqrt { z q}  } ^k$. From \eqref{eq:fpabs}, $|f(q)| \leq I + J $  and from \eqref{eq:boundI}, this concludes the proof of the lemma. 
\end{proof}

The following simple lemma  bounds the expected product of random variables in terms of the expected product of the random variables conditioned by the other variables. 
\begin{lemma}\label{le:condexp}
Let $T$, $(X_t)_{t \geq 1}, (x_t)_{t \geq 1}$ be random variables defined on a common probability space with $T$ a non-negative integer variable and $X_t$, $x_t$ real variables. Let $\cF_t = \sigma(T , (x_s)_{s},  (X_s)_{s \ne t})$ be the $\sigma$-algebra generated by all variables but $X_t$. We assume that for all $t \geq 1$,  $\dE \SBRA{ | X_t | \bigm| \cF_t } \leq x_t$. Then,
$$
\dE \ABS{  \prod_{t=1} ^T X_t } \leq \dE \prod_{t = 1}^T x_t .  
$$ 
\end{lemma}
\begin{proof}
By conditioning on the value of $T$, we may assume without loss of generality that $T$ is deterministic.  Let $\cG_t = \sigma ( (x_s)_s , (X_s)_{s < t})$, since $\cG_t \subset \cF_t$, we have
$$
 \dE \PAR{ \prod_{t=1}^T |X_t| } = \dE \PAR{ \prod_{t=1}^{T-1} |X_t| \,  \dE \SBRA{ |X_T| \bigm| \cG_T } } =  \dE \PAR{   \prod_{t=1}^{T-1} |X_t| \,  \dE \SBRA{ \dE \SBRA{ |X_T| \bigm|\cF_T } \bigm|  \cG_T } } 
$$
Applying our assumption, we find, since $x_T \in \cG_T$, 
$$
 \dE \PAR{  \prod_{t=1}^T |X_t| } \leq \dE \PAR{ x_T \prod_{t=1}^{T-1} |X_t|} .
$$
We then repeat the above step.
\end{proof}

\begin{proof}[Proof of Proposition \ref{le:exppath}] We will use that, if $k ( k \vee t)  \leq \alpha n$, $k \vee t \leq n/2$,  
\begin{equation}\label{eq:factpow}
(n)_k \geq   e^{-2\alpha} n^k  \quad  \hbox{ and } \quad (n - t )^k \geq e^{-2\alpha} n^k,
\end{equation}
(indeed, $(n)_k \geq (n - k)^k$ and $(n-t)^k =  n^{k} \exp ( k \log ( 1 - t /n)) \geq n^{k} \exp ( - 2 k t / n)$ since $\log (1 -x) \geq - x / (1 - x)$ for $0 \leq x  <1$). 

The proof relies on a conditional expectation argument. We set $\cE_\gamma = \{ y_1, \ldots , y_a  \}$ and $y_{t} = \{ e_t , f_t \}$. We also set $m = d n$ and
$$\vec E^* = \vec E \backslash \bigcup_{1 \leq t \leq a} \{ e_{t} , f_{t}  \}.$$
We have $| \vec E^* | \geq  m - 2 a$. The multiplicity of $y_t$ is equal to $p_t + q_t$, where $p_t$ is the multiplicity of $y_t$ in $(\gamma_1, \ldots, \gamma_{2k_0})$ and $q_t$ its multiplicity in $(\gamma_{2k_0 +1}, \ldots, \gamma_{2k})$.  We write 
$$
P =  \prod_{t= 1} ^{k_0}  \underline M_{\gamma_{2t-1} \gamma_{2t} } \prod_{t= k_0+1} ^{k}  M_{\gamma_{2t-1} \gamma_{2t} }  = \prod_{t=1}^a  \underline M_{e_tf_t} ^{p_t}  M_{e_tf_t} ^{q_t} . 
$$

Let $T$ be the set of $y_t = \{e_t , f_t \}$ such that $y_t$ is consistent, $p_t = 1$ and $q_t = 0$.  By assumption $|T| = a_1$. Note that if $t \in T$, $e_t \ne f_t$ and for all $s \ne t$, $\{ e_t,f_t \} \cap \{ e_s, f_s\} = \emptyset$. Let $T^*\subset T$ be the random subset of $t \in T$ such that $\sigma(e_t) \in \vec  E^* \cup \{ f_t \}$ and $\sigma(f_t) \in \vec E^* \cup \{ e_t \}$. In words, elements in $T^*$ are either  matched by $\sigma$ (that is, $\sigma(e_t) = f_t$) or their are matched outside $\gamma$ (that is, $\{ \sigma(e_t) , \sigma(f_t)\} \subset  \vec E^*$).  Similarly, let $S \subset T$ be the random subset of $t \in T$ such that $\{ \sigma(e_t) , \sigma(f_t) \} \cap  \{  e_s, f_s \} \ne \emptyset $ for some $s \in T \backslash \{t \}$. In words, elements in $S$ are matched by $\sigma$ to at least another element in $T$.

By construction, if $t \in S$, 
$$
 \underline M_{e_tf_t} ^{p_t}  M_{e_tf_t} ^{q_t} =  \underline M_{e_tf_t}   =  -\frac 1 m.
$$
We thus have 
$$
P = (-m)^{-|S|} P^* Q,
$$
where 
$$
P^*  = \prod_{t \in T^*}  \underline M_{e_tf_t}  \quad \hbox{ and } \quad Q =  \prod_{t\notin  S \cup T^* }   \underline M_{e_tf_t} ^{p_t}  M_{e_tf_t} ^{q_t} . 
$$

Now, we define $\cF$ to be the $\sigma$-algebra generated by the variables $T^*$ and $\sigma(e_t), \sigma(f_t) , t \notin  T^*$. We denote by $\dE_{\cF}$ the associated conditional expectation. By construction, the variables $S,T^*$ and $Q$ are  $\cF$-measurable. We get
\begin{equation}\label{eq:defPQP}
\ABS{ \dE \SBRA{ P}} = \ABS{ \dE \SBRA{(-m)^{-|S|} \, Q  \, \dE_{\cF} \SBRA{P^*} }} \leq   \dE \SBRA{ m^{-|S|} \, |Q| \,  \ABS{ \dE_{\cF} \SBRA{P^*} }} , 
\end{equation}
where the last step follows from Jensen's inequality.

We start by evaluating $ \dE_{\cF^*} \SBRA{P^*}$ in \eqref{eq:defPQP}. If $\hat N$ is the number of $t \in T^*$ such that $\sigma(e_t) \ne f_t$, we have 
\begin{eqnarray*}
 P^*   & =&  \PAR{ 1 - \frac 1 m}^{|T^* |  - \hat N} \PAR{ - \frac 1 m}^{\hat N}. 
\end{eqnarray*}
 We now determine the law of $\hat N$ given $\cF$. Let $\hat m = | \vec E^* | -   \sum_{t \notin T^*}( \IND_{ \si(e_t) \in \vec E^*} + \IND_{ \si(f_t) \in \vec E^* })$ be the cardinality of half-edges in $\vec E^*$ which have not yet been matched when  the values of $\sigma(e_t), \sigma(f_t) , t \notin  T^* $ have been revealed. We set, for  integers $ t \geq 0$, $k \geq 2$ even, 
$$
\lp k \rp_t = \prod_{s=0}^{t-1} (k - 2s) \quad \hbox{ and } \quad k!! = \prod_{s=0}^{k/2-1} (k -1 - 2s).
$$
Note that $k!!$ is the number of matchings of a set of size $k$. If $t \in T^*$ and $\sigma(e_t) \ne f_t$ then $\sigma(e_f), \sigma(f_t) \in \vec E^*$. Thus, given $\cF$, for $0 \leq x \leq |T^*|$, the number of matchings such that $\hat N = x$ is equal to ${|T^*| \choose x} (\hat m)_{2x} (\hat m-2x) !!  = {|T^*| \choose x} \lp \hat  m \rp_{  x} \hat m!!$. We deduce that
$$
\dP_{\cF} ( \hat N = x) = \frac{{|T^*| \choose x} \lp \hat  m \rp_{  x} }{Z}\quad \hbox{ with } \quad Z = \sum_{x = 0}^{|T^*|} {|T^*| \choose x} \lp \hat  m \rp_{x} . 
$$
First, since $\hat m \geq m - 4a$ and $a \leq  \sqrt{m}$, we obtain from \eqref{eq:factpow}, for some $c >0$, 
$$
Z \geq c \sum_{x = 0}^{|T^*|} {|T^*| \choose x}  m^{x}   = c ( 1+ m)^{|T^*|} \geq  c m^{|T^*|}. 
$$
From what precedes, we get, 
\begin{eqnarray*}
\dE_\cF \SBRA{P^*}  &=& \frac 1 Z \sum_{x = 0} ^{|T^*|} {|T^*| \choose x}  \lp \hat  m \rp_{x} \PAR{ 1 - \frac 1 m}^{|T^* |  - x} \PAR{ - \frac 1 m}^{x}\\
 & = &  \frac 1 Z \sum_{x = 0} ^{|T^*|} {|T^*| \choose x}  \prod_{y=0}^{x-1} ( 2y -  \hat m) \PAR{ 1 - \frac 1 m}^{|T^* |  - x} \PAR{ \frac 1 m}^{x} \\
& = &   \frac 1 Z \dE  \prod_{y=0}^{N-1} (2 y -  \hat m ),
\end{eqnarray*}
where $N$ has distribution $\Bin(|T^* | , 1 / m)$. By Lemma \ref{le:bin}, applied to $z = 2$,  $k =|T^* | $, $p = 1 / m$ and $q = 1 / \hat m $, we deduce that, for some $c >0$, 
\begin{equation}\label{eq:ddett}
\dE_\cF \SBRA{P^*} \leq c \PAR{ \frac{\veps}{m} }^{|T^*|} ,
\end{equation}
with $\veps =   3 a  / \sqrt m$ (since $3 |T^*| \sqrt{ z/2} \leq 3a$).

We now evaluate $|Q|$ in \eqref{eq:defPQP}. Let $\cF_t$ be the $\sigma$-algebra generated by $ \sigma(e_s), \sigma(f_s) , s \ne t$. For any  $ t \in [a]$, 
\begin{equation}\label{eq:MeftP}
\dP_{\cF_t} ( M_{e_t f_t} = 1 )  = \frac{ \IND( e_t \ne f_t)\IND( \Omega^c _t) } { m_t } \leq \frac{ 1} { m^*}, 
\end{equation}
where $m^* = m - 2 a +1$, $m_t = m - | \{ \sigma(e_s) , \sigma(f_s) : s \ne t \}| -1$ and $\Omega_t \in \cF_t$ is the event that for some $s \ne t$, $\{ \sigma(e_s), \sigma(f_s)\} \cap \{e_t,f_t\} \ne \emptyset$. 
We get, for $p \geq 2$ and $q\geq 0$,
\begin{equation} \label{eq:rdj20}
\dE_{\cF_t}  | \underline M_{e_t f_t} ^p M_{e_t f_t}^q |  \leq \dE_{\cF_t}   | \underline M_{e_t f_t} |^2   \leq\PAR{ 1 - \frac 1 m }^2  \frac{ 1} {m^*}  + \frac{1}{m^2}   \PAR{ 1 -  \frac{ 1} { m^*} } \leq \frac{1}{m^*}. 
\end{equation}

Similarly, if $q \geq 1$, 
\begin{equation} \label{eq:rdj01}
\dE_{\cF_t}  | \underline M_{e_t f_t}^p M_{e_t f_t}^q |  \leq \dE_{\cF_t}   M_{e_t f_t}    \leq \frac{1}{m^*}. 
\end{equation}

We also have the weak bound,
\begin{equation} \label{eq:rdj10}
\dE_{\cF_t} |   \underline M_{e_t f_t} |    \leq \PAR{ 1 - \frac 1 m }  \frac{ 1} {m^*}  + \frac{1}{m}   \PAR{ 1 -  \frac{ 1} { m^*} } \leq \frac{2}{m^*}. 
\end{equation}
This last bound can be improved for $t   \in T \backslash ( S \cup T^*)$ as follows (recall that $p_t = 1$, $q_t = 0$ for all $t \in T$). First, observe that the variables $S$ and $T^*$ are $\mathcal \cF_t$-measurable for any $t$. On the event $t \in T \backslash \{ S, T^*\}$, by construction, the event $\Omega_t$ holds and thus $M_{e_t f_t}  = 0$. It follows that if $t   \in T \backslash ( S \cup T^*)$, 
\begin{equation} \label{eq:rdj10*}
\dE_{\cF_t} |   \underline M^{p_t}_{e_t f_t} M_{e_t f_t}^{q_t}|    =\dE_{\cF_t} |   \underline M_{e_t f_t} | = \frac 1 m   \leq \frac{1}{m^*}. 
\end{equation}

We may estimate $\dE \SBRA{  | Q |  \bigm| (S,T^*) }$ as follows. If $y_t$ is such that $p_t \geq 2$, we use \eqref{eq:rdj20},  if  $q_t \geq 1$, we use \eqref{eq:rdj01}. If $y_t$ is an inconsistent edge such that $p_t = 1$ and $q_t = 0$, we use \eqref{eq:rdj10}. Finally, if $t \in T \backslash ( S \cup T^*)$,  we use \eqref{eq:rdj10*}. From Lemma \ref{le:condexp}, we find that
$$
\dE \SBRA{  | Q |  \bigm| (S,T^*) } =  \dE \SBRA{ \prod_{t\notin  S \cup T^* }   | \underline M_{e_tf_t} ^{p_t}  M_{e_tf_t} ^{q_t}| \bigm| (S,T^*) } \leq 2^b \PAR{\frac{1}{m^*} }^{a - |S| - |T^*|}.
$$

Putting this last bound together with \eqref{eq:ddett}, we deduce from \eqref{eq:defPQP} and \eqref{eq:factpow}  that for some $c >0$, 
\begin{eqnarray*}
\ABS{ \dE P   } &\leq &  c \, 2^b \, \dE \PAR{  \frac {1} m }^{a - |T^*| }   \PAR{ \frac{\veps  }{ m} }^{|T^*|}   \\
& = & c  \, 2^b \, \PAR{\frac{1}{m}}^{a} \veps^{a_1} \dE \veps^{-|T \backslash T^*|}.
\end{eqnarray*}

To conclude the proof, it thus remains to show that  $\dE \veps^{-|T \backslash T^*|} \leq c$ for some constant $c>0$. The event that $\{|T \backslash T^*| \geq x \}$ is contained in the event that there are $\lceil x / 2 \rceil$ pairs $\{s,t\}$, $s\ne t$, such  that $\{\sigma(e_s),\sigma(f_s) \cap \{ e_t , f_t \}\} \ne \emptyset$ (the latter can be further decomposed in the union of the four events, $\sigma(e_s) = e_t$, $\sigma(e_s) = f_t$, $\sigma(f_s) = e_t$ or $\sigma(f_s) = f_t$).  From the union bound, we get
$$
\dP \PAR{|T \backslash T^*| \geq x } \leq \PAR{ \frac{4 a^2}{m^*} }^{\lceil x / 2 \rceil}.
$$
Indeed, the factor $( 4 a^2 )^{\lceil x / 2 \rceil}$ accounts for the choices of the possible half-edges to be matched. The factor $(1/m^*)^{  \lceil x / 2 \rceil}$ is an upper bound on the probability that these half-edges are matched by $\sigma$ (from  Lemma \ref{le:condexp} and \eqref{eq:MeftP}). Since $2a \leq 2k \leq 2\sqrt{m}$ and $\lceil x / 2 \rceil \leq x/2 +1/2$, we get from \eqref{eq:factpow},
$$
\dP \PAR{|T \backslash T^*| \geq x } \leq c \PAR{ \frac{2 a}{\sqrt{m} } }^{x}.
$$
Recalling $\veps = 3a / \sqrt m$, we find 
$$
 \dE  \veps^{-|T \backslash T^*|} \leq \sum_{x=0}^\infty \veps^{-x} \dP ( |T \backslash T^*| \geq x) \leq c \sum_{x = 0}^{\infty}   \PAR{\frac 2 3}^{-x} = 3c.
$$
This concludes the proof of Proposition \ref{le:exppath}.
\end{proof}

\subsection{Path counting}

\label{subsec:PC}
In this subsection, we give upper bounds on the operator norms of $\uB^{(\ell)}$ and $R_\ell^{k}$ defined by \eqref{eq:defD} and \eqref{eq:defR}. We will use the high trace method and it will lead us to enumerate  some paths.

\subsubsection{Operator norm of $\uB^{(\ell)}$}

Here, we prove the following proposition. 

\begin{proposition} \label{prop:normDelta}
Let $d \geq 3$ and $1 \leq \ell \leq  \log n$ be integers.  Let $\sigma$ be uniformly distributed on $M(\vec E)$ and $\uB^{(\ell)} = \uB^{(\ell)}(\sigma)$ be defined as in \eqref{eq:defD}. Then, \whp
$$ \| \uB^{(\ell)} \| \leq ( \log n) ^{15} \PAR{ d -1} ^{\ell /2}.$$
\end{proposition}
Let $m$ be a positive integer. With the convention that $e_{2m + 1} = e_1$, we get 
\begin{eqnarray}
\|\uB^{(\ell)}  \| ^{2 m} = \|\uB^{(\ell)}{\uB^{(\ell)}}^*  \| ^{m} & \leq & \tr \BRA{ \PAR{  \uB^{(\ell)}{\uB^{(\ell)}}^*}^{m}  } \nonumber\\
& = & \sum_{e_1, \ldots, e_{2m}}\prod_{i=1}^{m}  (\uB^{(\ell)}) _{e_{2i-1} , e_{2 i}}(\uB^{(\ell)}) _{e_{2i+1} , e_{2 i}} \nonumber \\
& =  &  \sum_{\gamma }   \prod_{i=1}^{2m}  \prod_{t=1}^{\ell} \underline M_{\gamma_{i,2t-1}  \gamma_{i,2t}} ,  \label{eq:trDeltak}
\end{eqnarray}
where the sum is over all  $\gamma = ( \gamma_1, \ldots, \gamma_{2m})$ such that $\gamma_i = (\gamma_{i,1}, \ldots, \gamma_{i,2\ell+1}) \in F^{\ell}$ (that is, non-backtracking tangled-free path) and for all $i  \in [m]$, 
$$
\gamma_{2i,1} = \gamma_{2i+1, 1} \quad \hbox{ and } \quad  \gamma_{2i-1,2\ell+1} = \gamma_{2i, 2\ell+1},
$$
with the convention that $\gamma_{2m+1} = \gamma_{1}$. We set $\gamma_{i,t} = (v_{i,t}, j_{i,t})$. Note that the product \eqref{eq:trDeltak} does not depend on the value of $\gamma_{2i-1,2\ell+1} = \gamma_{2i, 2\ell+1}$, $i\in [ m]$. Moreover, if $\gamma_{2i-1,2\ell}$ and $\gamma_{2i,2\ell}$ are given, then $\gamma_{2i-1,2\ell+1} = \gamma_{2i, 2\ell+1}$ can either take $(d-1)$ possible values (if $j_{2i-1,2\ell} = j_{2i,2\ell})$ or $(d-2)$ possible values (if $j_{2i-1,2\ell} \ne j_{2i,2\ell})$. On the right-hand side of \eqref{eq:trDeltak}, for all $i \in [m]$, we sum over $\gamma_{2i-1,2\ell+1} = \gamma_{2i, 2\ell+1}$, and perform the change of variable, for all $(i,t) \in [m] \times  [2\ell]$, $\gamma'_{2i,t} = \gamma_{2i, 2 \ell +1 -t }$. Since, $\underline M$ is symmetric, we may rewrite the right-hand side of \eqref{eq:trDeltak} as follows: 
\begin{equation}\label{eq:trDeltak2}
\| \uB^{(\ell)}  \| ^{2 m}  \leq   \sum_{\gamma \in W_{\ell,m}}  q(\gamma)  \prod_{i=1}^{2m}  \prod_{t=1}^{\ell} \underline M_{\gamma_{i,2t-1}  \gamma_{i,2t}} , 
\end{equation}
where $W_{\ell,m}$ is the set of $\gamma = ( \gamma_1, \ldots, \gamma_{2m}) \in \vec E^{2\ell \times 2m}$ such that $\gamma_i = (\gamma_{i,1}, \ldots, \gamma_{i,2\ell})$ is a non-backtracking tangle-free path and for all $i\in[m]$, 
\begin{equation}\label{eq:defbound}
v_{2i,1} = v_{2i-1, 2\ell}  \quad \hbox{ and } \quad  \gamma_{2i+1,1} = \gamma_{2i, 2\ell},
\end{equation}
with the convention that $\gamma_{2m+1} = \gamma_{1}$ and  $\gamma_{i,t} = (v_{i,t}, j_{i,t})$, see Figure \ref{fig:fleur}. Finally, in \eqref{eq:trDeltak2}, we have set
\begin{equation}\label{eq:defqg}
q(\gamma) = \prod_{i=1} ^{m} \PAR{ d -1 - \IND_{j_{2i-1,2\ell} \ne j_{2i,1}} }\leq (d-1)^m.
\end{equation}

\begin{figure}[htb]
\begin{center}  
\resizebox{9cm}{!}{
\begin{tikzpicture}[main node/.style={circle,fill , text = black, thick}]
\node  at (0,0) (1) {} ;
\node at (2,0) (2) {} ;
\node[right] at (3,1.73205) (3) {{\tiny $v_{1,2\ell} = v_{2,1}$}}   ;
\node at (2,3.4641016) (4) {} ;
\node at (0,3.4641016) (5) {} ;
\node[left] at (-1,1.73205) (6) {{\tiny $v_{2i-1,2\ell} = v_{2i,1}$}}  ;

\node at (1,0.8660254) (a) {} ;
\node[below] at (1.8660254,1.3660254) (b) {{\tiny \hspace{10pt}$\gamma_{1,1} = \gamma_{12,2k}$}} ;
\node  at (1.8660254,2.23205) (c) {} ;
\node at (1,2.5980762) (d) {} ;
\node  at (0.1339746,2.23205) (e) {} ;
\node[below ]  at (0.1339746,1.3660254) (f) {{\tiny\hspace{-5pt} $\gamma_{2i,2k} = \gamma_{2i+1,1}$}}  ;

\draw[cyan , ->,thick] (1.8660254,1.3660254)  [out = 0 , in = 180] to (3) {} ;  \node at (2.3,1.46) {{\small $\gamma_1$}} ; 
\draw[cyan ,->,thick] (3)  [out = 180 , in = 0] to (1.8660254,2.23205) {} ;  \node at (2.3,2) {{\small $\gamma_2$}}  ; 
\draw[cyan ,->,thick] (1.8660254,2.23205)  [out = 60 , in = -120] to (4)   ; 
\draw[cyan ,->,thick] (4)  [out = -120 , in = 60] to (1,2.5980762)  ; 
\draw[cyan ,->,thick] (1,2.5980762)  [out = 120 , in = -60] to (5)   ; 
\draw[cyan ,->,thick] (5)  [out = -60 , in = 120] to (0.1339746,2.23205)   ; 
\draw[cyan ,->,thick] (0.1339746,2.23205) [out = 180 , in = 0] to (6)   ;  \node at (-0.339746,2) {{\small $\gamma_{2i-1}$}} ; 
\draw[cyan ,->,thick] (6)  [out = 0 , in = 180] to (0.1339746,1.3660254)  ; \node at (-0.339746,1.46) {{\small $\gamma_{2i}$}} ; 
\draw[cyan ,->,thick] (0.1339746,1.3660254)  [out = -120 , in = 60] to (1)   ; \node at (0,0.8) {{\small $\gamma_{2i+1}$}} ; 
\draw[cyan ,->,thick] (1)  [out = 60 , in = -120] to (1,0.8660254)   ; 
\draw[cyan ,->,thick] (1,0.8660254)  [out = -60 , in = 120] to (2) {} ; \node at (2.05,0.8)  {{\small $\gamma_{12}$}}  ; 
\draw[cyan ,->,thick] (2)  [out = 120 , in = -60] to (1.8660254,1.3660254)    ;

\end{tikzpicture}
}
\caption{A path $\gamma = (\gamma_1, \ldots , \gamma_{12})$ in $W_{\ell,6}$, each $\gamma_i$ is non-backtracking and tangle-free.} \label{fig:fleur}

\end{center}\end{figure}
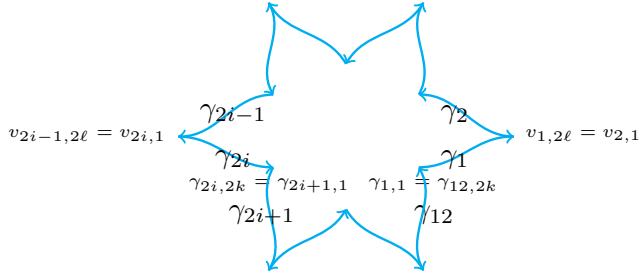

The proof of Proposition \ref{prop:normDelta} relies on an upper bound on the expectation of the right-hand side of \eqref{eq:trDeltak2}. First, for each $\gamma \in \vec E^{2\ell\times 2m}$, we define $G_\gamma$ as in Definition \ref{def1}: $V_\gamma = \cup_i V_{\gamma_i} = \{ v_{i,t} : (i,t) \in [2m] \times [2\ell] \} \subset [n]$ and $\cE_\gamma = \cup_i \cE_{\gamma_i} = \{ \{ \gamma_{i,2t-1}  , \gamma_{i,2t}  \} :   (i,t) \in [2m] \times [\ell] \}$ are the sets of visited vertices and visited pairs of half-edges along the path.  For $v \in V_\gamma$,  $\vec E_\gamma(v) = \{ \gamma_{i,t} : v_{i,t} = v \hbox{ for some $(i,t) \in [2m] \times [2\ell]$} \}\subset \vec E(v)$ is the set of visited half-edges pending at $v$.

In order to organize the terms on the right-hand side of \eqref{eq:trDeltak2}, we partition  $\vec E^{2\ell \times 2m}$ into isomorphism classes. For $\gamma,\gamma'\in \vec E^{2\ell \times 2m}$, we write $\gamma \sim \gamma'$ if there exist a permutation $\alpha \in S_n$ and permutations $(\beta_1, \ldots, \beta_{n} ) \in S_{d}^n$ such that, with $\gamma'_{i,t}  = (v'_{i,t} , j'_{i,t})$, for all $(i,t) \in [2m] \times [2\ell]$, $v'_{i,t} = \alpha ( v_{i,t})$ and $j'_{i,t}  = \beta_{v_{i,t}} ( j_{i,t})$.  We may define a canonical element in each isomorphic class as follows. We say that a path $\gamma \in \vec E^{2\ell \times 2m}$ is canonical if $V_\gamma = \{ 1, \ldots, |V_\gamma| \}$, for all $v \in V_\gamma$, $\vec E_\gamma(v) = \{ (v,1), \ldots, (v,|\vec E_\gamma(v)|)\}$  and the vertices in $V_\gamma$ and the the half-edges in $\vec E_\gamma(v)$  are visited in the lexicographic order ($x$ before $x+1$ and $(x,j)$ before $(x,j+1)$).  Note that $\gamma \in W_{\ell,m}$ and $\gamma'\sim \gamma$ implies that $\gamma'\in W_{\ell,m}$. Our first lemma bounds the number of paths in each isomorphism class.
\begin{lemma}\label{le:isopath}
Let $\gamma \in W_{\ell,m}$ with $|V_\gamma| = s $ and $|\cE_\gamma| = a$. If $g = a - s +1$, then $\gamma$  is isomorphic to at most 
$
  n^s \PAR{ d (d-1) }^{ s} (d-1)^{2g -1} 
$ 
paths in $W_{\ell,m}$.  
\end{lemma}

\begin{proof}
For $v \in V_\gamma$, let $d_v = |\vec E_\gamma(v)|$ and, for $t \in [d]$ integer, recall the Pochhammer symbol, $(d)_t = d (d-1) \cdots (d-t+1)$.  Observe that, if $s_t = \sum_{v\in V_\gamma} \IND ( d_v = t)$,  we have 
$$
\sum_{t \geq 1} s_t = s \quad \hbox{ and } \quad \sum_{t \geq 1} t s_t = 2 a.
$$
By construction, $\gamma$ is isomorphic to 
$$
(n)_s \prod_{v =1}^s (d)_{d_v}\leq n^s \prod_{t \geq 1} {(d)_t }^{s_t}
$$
distinct elements in $W_{\ell,m}$. However,
$$
\prod_{t \geq 1} {(d)_t }^{s_t} \leq d^{\sum_{t} s_t } (d-1) ^{\sum_{t\geq1} (t-1) s_t}  = d^s (d-1)^{2 a - s} = \PAR{d (d-1) }^s \PAR{d-1 }^{2 g -1}. 
$$
The conclusion follows.
\end{proof}

Our second lemma gives an  upper bound on the number of isomorphic classes. This lemma is  a variant of \cite[Lemma 17]{BLM}. It relies crucially on the fact that an element $\gamma \in W_{\ell,m}$ is composed of $2m$ tangle-free paths.  

\begin{lemma}\label{le:enumpath}
Let $\cW_{\ell,m} (s,a) $ be the subset of canonical paths with $|V_\gamma| = s$ and $|\cE_\gamma |= a$. Let $g = a - s +1$. If $g <0$, $\cW_{\ell,m} (s,a)$ is empty. If $g \geq 0$, we have 
$$
| \cW _{\ell,m} (s,a) | \leq   (4 \ell  m )^{6 m g  + 6 m}.
$$
\end{lemma}

\begin{proof}
For any $\gamma \in W_{\ell,m}$, due to the boundary conditions \eqref{eq:defbound}, the graph $G_\gamma$ is connected. Hence $ |V_\gamma| - 1 \leq |\cE_\gamma |$. It implies the first claim of the lemma.   In order to upper bound $| \cW_{\ell,m} ( s, a) | $, we find an efficient way  to encode the canonical paths $\gamma \in \cW_{\ell,m} ( s, a) $ (that is, find an injective map from $\cW_{\ell,m}(s,a)$ to a larger set whose cardinality is easily upper bounded).

For $(i,t) \in [2m]\times [\ell]$, let $x_{i,t} = ( \gamma_{i,2t-1}, \gamma_{i,2t} )$ and $y_{i,t} =  \{ \gamma_{i,2t-1}, \gamma_{i,2t} \} \in \cE_\gamma$ be the corresponding visited edge. We explore the sequence $(x_{i,t})$ in lexicographic order denoted by $\preceq$ (that is $(i,t)\preceq (i+1,t')$ and $(i,t)\preceq(i,t+1)$). By convention, we set $(i,\ell+1) = (i+1, 1)$ (if $i = 2m$, $(2m,\ell+1) = (1,1)$). We think of the index $(i,t)$ as a time.  We say that $(i ,t)$ is a {\em first time}, if $v_{i,2t}$ has not been seen before (that is $v_{i,2t} \ne v_{i', t'}$ for all $(i',t') \preceq (i,2t)$). The edge $y_{i,t}$ will then be called a {\em tree edge}. As its name suggests, the graph,  spanned by the  edges $ \{ \{ v_{i,2t-1}, v_{i,2t} \} : (i,t) \hbox{ first time}\}$ is a tree, it is a spanning tree of $G_\gamma$. An edge  $y_{i,t}$ which is not a tree edge, is called an {\em excess edge}, and  we say that $(i,t)$ is an {\em important time} (see Figure \ref{fig:impo}). Since every vertex in $V_\gamma$ different from $1$ has its associated tree edge,
\begin{equation}\label{eq:defchi}
\ABS{ \BRA{ y \in \cE_\gamma : \hbox{ $y$ is an excess edge}} } = a - s +1 = g.
\end{equation}

 Since $\gamma_i$ is non-backtracking in the sense of Definition \ref{def1}, the path $\gamma_i$ can be decomposed by the successive repetition of (i) a sequence of first times (possibly empty), (ii) an important time and (iii) a path on the tree defined so far (possibly empty). Note also that, if $(i,t)$ is a first time then  $\gamma_{i,2t} = ( m +1 , 1)$ and  $\gamma_{i,2t+1} = ( m +1 , 2)$ where $m$ is the number of previous first times (including $(i,t)$). Indeed, since $\gamma$ is canonical, $\gamma_{1,1} = (1,1)$ and every time that a new vertex, say $v$, is visited, the half-edge $(v,1)$ will be seen first.

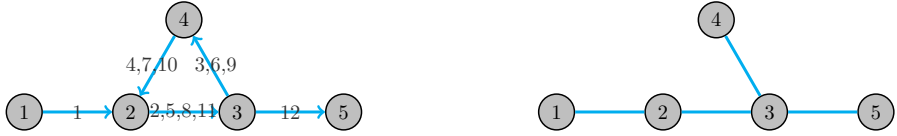
\begin{figure}[htb]
\begin{center}  
\resizebox{12cm}{!}{
\begin{tikzpicture}[main node/.style={circle, draw , fill = lightgray, text = black, thick}]
\node[main node]  at (0,0) (1) {1} ;
\node[main node] at (2,0) (2) {2} ;
\node[main node] at (4,0) (3) {3} ;
\node[main node] at (3,1.7320508 ) (4) {4} ;
\node[main node] at (6,0) (5) {5} ;

\draw[cyan, ->,ultra thick] (1) to (2) ; 
\draw[ cyan, ->,ultra thick] (2) to (3) ; 
\draw[ cyan, ->,ultra thick] (3) to (4) ;  
  \draw[cyan,  ->,ultra thick] (4) to (2) ;  
 \draw[cyan,  ->, ultra thick] (3) to (5) ;

  \node[text = black!80] at (1,0) {1} ; 
    \node[text = black!80] at (3,0) {2,5,8,11} ; 
    \node[text = black!80] at (5,0) {12} ;   
       \node[text = black!80] at (3.6,0.8660254) {3,6,9} ; 
    \node[text = black!80] at (2.4,0.8660254) {4,7,10} ;

  \node[main node]  at (10,0) (10) {1} ;
\node[main node] at (12,0) (20) {2} ;
\node[main node] at (14,0) (30) {3} ;
\node[main node] at (13,1.7320508 ) (40) {4} ;
\node[main node] at (16,0) (50) {5} ;

\draw[cyan, -,ultra thick] (10) to (20) ; 
\draw[cyan, -,ultra thick] (20) to (30) ; 
\draw[ cyan, -,ultra thick] (30) to (40) ;  
 \draw[cyan,  -, ultra  thick] (30) to (50) ;

\end{tikzpicture}
}
{\scriptsize$$ \gamma_1 =  (1,1)  (2,1)  (2,2)  ( 3,1)  (3,2)   (4,1)  (4,2)   (2,3) (2,2)  (3,1)  (3,2)  (4,1)  (4,2)  (2,3)  (2,2)  (3,1)  (3,2)  (4,1)  (4,2)   (2,3)  (2,2)  (3,1)  (3,2)  (5,1) $$} 
\vspace{-25pt}
\caption{A canonical path $\gamma_1$ (non-backtracking and tangle-free) and its associated spanning tree. The times $(1,t)$ with $t \in \{ 1,2,3,12\}$ are first times and  $ t= \{ 4,7, 10\} $ are important times, $(1,4)$ is the short cycling time, $(1,7), (1,10)$ are superfluous.  With the notation below, $t_1 = 4$, $\sigma = 2$, $\hat t = 12$, $\hat \tau = 13$.} \label{fig:impo}

\end{center}\end{figure}

We can thus build a first encoding of $\cW_{\ell,m}(s,a)$ as follows. If $(i,t)$ is an  important time, we mark the time $(i,t)$  by the vector $(\gamma_{i,2t},\gamma_{i,2\tau-1})$, where $(i,\tau)$ is the next time that $y_{i,\tau}$ will not be a tree edge of the tree constructed so far (by convention, if the path $\gamma_i$ remains on the tree, we set $\tau = \ell+1$). From \eqref{eq:defbound}, for $t=1$, we also add the {\em starting mark} $\gamma_{i,2\tau-1}$ where $(i,\tau)$ is as above the next time that $y_{i,\tau}$ will not be a tree edge of the tree constructed so far.  Since there is a unique non-backtracking path between two vertices of a tree, we can reconstruct $\gamma \in \cW_{\ell,m}$ from the starting marks and the position of the important times and their marks. This defines our first encoding.

The main issue with this encoding is that the number of important times could be large  (see Figure \ref{fig:impo}). This is where the hypothesis that each path $\gamma_i$ is tangle-free comes into play. We  partition important times into three categories, {\em short cycling}, {\em long cycling} and {\em superfluous} times. For each $i$, we consider the smallest time $(i,t_1)$ such that $v_{i,2t_1} \in \{ v_{i,1}, \ldots, v_{i,2t_1-1} \}$. If such time $t_1$ exists, the last important time $(i,t) \preceq (i,t_1)$ will be called the short cycling time.  Let $1 \leq \sigma \leq t_1$ be such that $v_{i,2t_1} = v_{i,2\sigma -1}$. By the tangle-free assumption, $C_i = (\gamma_{i,2\sigma-1},\cdots, \gamma_{i,2t_1})$ will be the unique cycle visited by $\gamma_i$.  We denote by $(i,\hat t)$, $\hat t \geq t_1$, the smallest time that $\gamma_{i,2\hat t-1}$ in not in $C_i$ (by convention $\hat t = \ell+1$ if $\gamma_i$ remains in $C_i$). We modify the mark of the short cycling time as  $(\gamma_{i,2t}, v_{i,2t_1}, \hat t ,\gamma_{i,2\hat \tau-1})$, where $(i,\hat \tau)$, $\hat \tau \geq \hat t$, is the next time that $y_{i,\hat \tau}$ will not be a tree edge of the tree constructed so far. Important times $(i,t')$ with $1 \leq t' < t$ or $\tau \leq t' \leq k$ are called long cycling times. The other important times are called superfluous. The key observation  is that for each $i \in [2 m]$, the number of long cycling times on $\gamma_i$ is bounded by  $g-1$ (since there is at most one cycle, no edge of $\cE_\gamma$ can be seen twice outside those of $C_i$, the $-1$ coming from the fact the short cycling time is an excess edge). 

We now have our second encoding. We can reconstruct $\gamma$ from the starting marks, the positions of the long cycling  and the short cycling times and their marks. For each $i$, there are at most $1$ short cycling time and $g-1$ long cycling times. There are at most $  \ell ^{2m g}$ ways to position them.  The number of distinct half-edges $\gamma_{i,t}$ in $\gamma$ is at most $h = 4 \ell m$. There are at most $h ^2$ different possible marks for a long cycling time and $ s h ^2 \ell$ marks for a short cycling time. Finally, there are $h$ possibilities for a starting mark. We deduce that    
$$
| \cW _{\ell,m} (s,a) | \leq    \ell ^{2 m g} ( h^2 ) ^{2m (g-1)} (s h^2 \ell )^{2m}   (h) ^{2m}. 
$$
Using $s \leq 2\ell m$, the last expression is generously bounded by the statement of the lemma. \end{proof}

For $\gamma \in W_{\ell,m}$, the average contribution of $\gamma$ in \eqref{eq:trDeltak2} is
\begin{equation}\label{eq:defmug}
\mu(\gamma) =   \dE \prod_{i=1}^{2m}  \prod_{t=1}^{\ell} \underline M_{\gamma_{i,2t-1}  \gamma_{i,2t}}.
\end{equation}
Observe that if $\gamma \sim \gamma'$ then $\mu(\gamma) = \mu(\gamma')$. Our final lemma uses  Proposition \ref{le:exppath} to estimate this average contribution. 

\begin{lemma}\label{le:meanpath}
There is a constant $c > 0$ such that, if $2 \ell m \leq \sqrt{ d n}$ and $\gamma \in W_{\ell,m} $ with $|V_\gamma| = s$, $|\cE_\gamma |= a$ and $g = a - s+1$, we have
$$
\ABS{ \mu(\gamma)} \leq  c^{g + m}  \PAR{\frac{ 1 }{dn} }^a   \PAR{\frac{ (6\ell m)^2 }{dn} } ^{  (a - 2g - (\ell+2) m)_+}.  
$$
\end{lemma}
\begin{proof}
Let $\cE'_1$ be the set of $y \in \cE_\gamma$ which are visited exactly once in $\gamma$, that is which are of multiplicity one in the sense of Definition \ref{def:defcons}.
We set $a'_1 =| \cE'_1 |$. Similarly, let $a_2$ the number of $y \in \cE_\gamma$ are visited at least twice.  We have 
$$
a'_1 + a_2 = a \quad  \hbox{ and } \quad a'_1 + 2 a_2 \leq 2\ell m. 
$$ 
Therefore, $a'_1 \geq 2( a - \ell m)$. Let  $\cE_1 $ be the subset of $ y \in \cE'_1 $ which are consistent and let $\cE_i $ the set of inconsistent edges (in the sense of Definition \ref{def:defcons}).  Using the terminology of the proof of Lemma \ref{le:enumpath}, a new inconsistent edge can appear at the the start of a non-empty sequence of first times or at a first visit of an excess edge. Every such step can create at most $2$ new inconsistent edges. From  \eqref{eq:defchi}, there are $g$ excess edges. Moreover, every non-empty sequence of
first times started in $\gamma_i$, $i \in [2m]$, either is followed by a first visit of an excess edge or ends $\gamma_i$. Hence, if $a_1 = |\cE_1|$ and $b = |\cE_i |$, we have $b \leq 4 g + 4m $ and
$ a_1 \geq a'_1 - 4 g - 4m$. So finally, $a_1 \geq 2 ( a  - 2 g   - (\ell+2)m)$. It remains to apply  Proposition \ref{le:exppath}. 
\end{proof}

\begin{proof}[Proof of Proposition \ref{prop:normDelta}]
For $n \geq 3$, we define 
\begin{equation}\label{eq:choicem}
m = \left\lfloor  \frac{ \log n }{13 \log (\log n)} \right\rfloor.
\end{equation}
The integer $m$ is positive for all $n$ large enough.  We claim that it is sufficient to prove that 
\begin{equation}\label{eq:boundS}
S = \sum_{\gamma \in W_{k,m} }   | \mu(\gamma) | \leq n  (c \ell  m )^{6 m} (d-1)^{\ell' m},
\end{equation}
where $\ell' = \ell+2$ and $\mu(\gamma)$ was defined in \eqref{eq:defmug}. Indeed, from \eqref{eq:trDeltak2}, we get for some new constant $c >0$, 
$$
\dE \|\uB^{(\ell)}  \| ^{2 m} \leq n  (c \ell  m )^{6 m} (d-1)^{\ell m}.
$$
Thus, from Markov inequality, for any $x \geq 1$,
\begin{equation}\label{eq:Markov}
\dP \PAR{ \|\uB^{(\ell)}  \| \geq n^{1/(2m)} (c \ell m) ^3 (d-1)^{\ell/2} x} \leq x^ {-2m}.
\end{equation}
For our choice of $m$, $n ^{  1 / (2m) } = o ( \log n )^{7}$ and $\ell m = o ( \log n) ^2$. Proposition \ref{prop:normDelta} follows.

We now prove \eqref{eq:boundS}. Using Lemma \ref{le:isopath}, Lemma \ref{le:enumpath} and Lemma \ref{le:meanpath}, we obtain,
\begin{eqnarray*}
S & \leq & \sum_{s = 1}^\infty \sum_{a = s - 1} ^{\infty} \max_{ \gamma \in \cW_{\ell,m} (s,a)} |\{ \gamma' \in W_{\ell,m} : \gamma' \sim \gamma \} | \times  |\cW_{\ell,m} (s,a) | \times  \max_{ \gamma \in \cW_{\ell,m} (s,a)} \mu(\gamma)  \\
& \leq &  \sum_{s=1}^{\infty} \sum_{a = s - 1} ^{\infty} n^s ( d (d-1) ) ^{s } (d-1)^{2g -1}  (4 \ell  m )^{6 m g + 6 m}  c^{g + m} \PAR{\frac{ 1 }{dn} }^a  \PAR{\frac{ (6\ell m)^2 }{dn} } ^{  (a - 2g - \ell ' m)_+} 
 \end{eqnarray*}
where $g = g(s,a) = a - s+ 1$. We perform the change of variable $a = s + g -1$, we get for some constant $c' >0$ and all $n$ large enough, 
\begin{eqnarray*}
S & \leq &   \sum_{s=1}^{\infty} \sum_{g=0} ^{\infty} \PAR{\frac{dn}{d-1}} \PAR{ c (4 \ell m)^6}^m  \PAR{d-1}^{s }\PAR{ \frac{ c (d-1)^{2}  (4 \ell  m )^{6m} }{dn}}^g    \PAR{\frac{ (6\ell m)^2 }{dn} } ^{  (s - g - 1 -  \ell ' m)_+}\\
& \leq & \sum_{s=1}^{\infty} \sum_{g=0} ^{\infty} n  (c'  \ell  m)^{6 m}    (d-1)^{s}  \PAR{\frac{ (c' \ell m) ^{6m} }{ n } }^{g} \PAR{\frac{ (6\ell m)^2 }{dn} } ^{  (s - g - 1 -  \ell ' m)_+},\\
& = &  S_1 + S_2 + S_3, 
 \end{eqnarray*}
where $S_1$ is the sum over $\{ 1 \leq s \leq \ell'm, g \geq 0 \}$, $S_2$ over $\{\ell' m+1 \leq s , 0 \leq g \leq s - 1 - \ell' m \}$, and $S_3$ over $\{ \ell' m+1 \leq s , g \geq s - \ell' m\}$. We have 
\begin{eqnarray*}
S_1 & = &  n  (c'  \ell  m)^{6 m}   \sum_{s=1}^{ \ell' m} (d-1)^{s} \sum_{g=0} ^{ \infty } \PAR{\frac{ (c' \ell m) ^{6m} }{ n } }^{g} \\
& \leq & 2 n (c'  \ell  m )^{6 m}  (d-1)^{\ell'm}   \sum_{g = 0} ^{ \infty } \PAR{\frac{  (c' \ell m )^{6m} }{ n } }^{g}.
\end{eqnarray*}
For our choice of $m$ in \eqref{eq:choicem}, for $n$ large enough, 
$$
\frac{   (c' \ell m) ^{6m} }{ n } \leq \frac{( \log n )^{12m} }{n} \leq n^{-1/13}. 
$$
In particular, the above geometric series converges . Hence, adjusting the value of $c$,  the right-hand side of \eqref{eq:boundS} is an upper bound for $S_1$. Similarly, with $\veps =   (6 \ell m)^2  / dn    = o(1)$, for some constant $c ' >0$, for $n$ large enough, 
\begin{eqnarray*}
S_2 & =  &   n   (c'  \ell  m)^{6 m}  \sum_{s= \ell' m+1}^{\infty} (d-1)^{s} \veps^{s-1 -\ell' m} \sum_{g = 0} ^{s - 1 - \ell' m } \PAR{\frac{(c' \ell m) ^{6m}  }{ \veps n } }^{g} \\
& \leq & 2 n (c' \ell'  m )^{6 m}     \sum_{s = \ell' m+1}^{ \infty} (d-1)^{s}   \PAR{\frac{(c' \ell m) ^{6m} }{ n } }^{ s- 1 - \ell'm} \\
& = & 2 n (c' \ell'  m )^{6 m}  (d-1)^{\ell 'm+1}\sum_{p = 0}^{ \infty} (d-1)^{p}   \PAR{\frac{ (c' \ell m) ^{6m} }{ n } }^{ p}.
\end{eqnarray*}
Again, the geometric series are convergent and the right-hand side of \eqref{eq:boundS} is an upper bound for $S_2$. Finally, 
\begin{eqnarray*}
S_3 & =  &  n   (c'  \ell  m)^{6 m}  \sum_{s= \ell'm+1}^{\infty} (d-1)^s  \sum_{g = s - \ell'm }^{\infty}  \PAR{\frac{ (c' \ell m) ^{6m} }{ n } }^{g} \\
& \leq & 2 n   (c'  \ell  m)^{6 m}    \sum_{s = \ell'm+1}^{ \infty} (d-1)^{s}   \PAR{\frac{  (c' \ell m) ^{6m}  }{ n } }^{ s - \ell'm} \\
& = & 2 n   (c'  \ell  m)^{6 m} (d-1)^{\ell'm}\sum_{p = 1}^{ \infty} (d-1)^{p}   \PAR{\frac{ (c' \ell m) ^{6m} }{ n } }^{ p}.
\end{eqnarray*}
The right-hand side of \eqref{eq:boundS} is an upper bound for $S_3$. This concludes the proof. 
\end{proof}

\begin{remark}\label{rk:QB} Markov inequality \eqref{eq:Markov} applied to $x = (\log n)^b$ actually implies that for any $a >0$, there exists $c >0$ such that the event $ \| \uB^{(\ell)} \| \leq  ( \log n) ^{c} \PAR{ d -1} ^{\ell /2}$ has probability at least $1 - n^{-a}$ for all $n$ large enough.
\end{remark}
\subsubsection{Operator norm of $R^{(\ell)}_{k}$}

\label{subsec:Rlk}
We now adapt the above subsection to the matrices $R^{(\ell)}_{k}$. 
\begin{proposition} \label{prop:normR}
Let $d \geq 3$, and $1 \leq \ell \leq  \log n$ be integers.  Let $\sigma$ be uniformly distributed on $M(\vec E)$ and for $k \in [\ell]$, let $R_k^{(\ell)} = R_k^{(\ell)}(\sigma)$ be defined as in \eqref{eq:defR}. Then, \whp
$$ \sum_{k=1}^\ell \| R^{(\ell)}_{k} \| \leq ( \log n) ^{30}\PAR{ d - 1}^{\ell}.$$
\end{proposition}
Let $m$ be a positive integer and $k \in [\ell]$.  Arguing as in  \eqref{eq:trDeltak}, we find
\begin{eqnarray}
\|R^{(\ell)}_{k}  \| ^{2 m} & \leq  &  \sum_{\gamma }  \prod_{i=1}^{2m}  \prod_{t=1}^{k-1} \underline M_{\gamma_{i,2t-1}  \gamma_{i,2t}} \prod_{t=k+1}^{\ell}  M_{\gamma_{i,2t-1}  \gamma_{i,2t}}  ,  \label{eq:trDeltak0}
\end{eqnarray}
where the sum is over all  $\gamma = ( \gamma_1, \ldots, \gamma_{2m})$ such that $\gamma_i = (\gamma_{i,1}, \ldots, \gamma_{i,2\ell+1}) \in F^{\ell}_k \backslash F^\ell$,  and for all $i \in [m]$, 
$$
\gamma_{2i,1} = \gamma_{2i+1, 1} \quad \hbox{ and } \quad  \gamma_{2i-1,2\ell+1} = \gamma_{2i, 2\ell+1},
$$
with the convention that $\gamma_{2m+1} = \gamma_{1}$. The product \eqref{eq:trDeltak0} does not depend on the value of $\gamma_{2i-1,2\ell+1} = \gamma_{2i,2\ell+1}$ for all $i\in[m]$. Moreover, given the values of $\gamma_{2i-1,2\ell}$, $\gamma_{2i, 2 \ell}$, the half-edge $\gamma_{2i-1,2\ell+1} = \gamma_{2i-1,2\ell+1}$ can take $(d-1)$ or $(d-2)$ possibles values.  On the right-hand side of \eqref{eq:trDeltak0}, for all $i\in[m]$, we sum over the values of $ \gamma_{2i-1,2\ell+1} =  \gamma_{2i-1,2\ell+1} $, and we perform the change of variable for all $(i,t) \in [m]\times[\ell]$, $\gamma'_{2i,t} = \gamma_{2i, 2\ell +1 -t}$, we find
\begin{eqnarray}
\| R^{(\ell)}_{k}  \| ^{2 m}  &\leq & \sum_{\gamma \in W^k_{\ell,m}} q (\gamma)  P_k (\gamma), \label{eq:trDeltak20}
\end{eqnarray}
where $q(\gamma)$ was defined in \eqref{eq:defqg} and $W^k_{\ell,m}$, $P_k(\gamma)$ are defined as follows. We set for $i \in [2m]$, 
\begin{equation}\label{eq:defki}
k_i = \left\{ \begin{array}{ll} k & \hbox{if $i$ odd} \\
 \ell  - k +1 & \hbox{if $i$ even}
\end{array}\right.
\end{equation}
 The set $W^k_{\ell,m}$ is the collection of  $\gamma = ( \gamma_1, \ldots, \gamma_{2m}) \in \vec E^{2\ell  \times 2m}$ such that for all $i \in [2m]$, $\gamma_i  = (\gamma_{i,1}, \ldots, \gamma_{i,2\ell})$ is non-backtracking and tangled but $$\gamma'_i = ( \gamma_{i,1}, \ldots , \gamma_{i,2k_i-2}) \quad  \hbox{ and } \quad \gamma''_i =  ( \gamma_{i,2k_i+1}, \ldots , \gamma_{i,2\ell})$$ are tangle-free.  We also have the boundary condition \eqref{eq:defbound} with  $\gamma_{i,t} = (v_{i,t}, j_{i,t})$. Finally, in \eqref{eq:trDeltak20}, for $\gamma \in W^k_{\ell,m}$, we have set 
$$P_k (\gamma) =  \prod_{i=1}^{2m}  \prod_{t=1}^{k_i-1}  M^{\veps_i}_{\gamma_{i,2t-1}  \gamma_{i,2t}}   \prod_{t=k_i+1}^{\ell} M^{\veps_i}_{\gamma_{i,2t-1}}, $$
where $M^{\veps_i} = \underline M$ if $i$ is odd and $M^{\veps_i} = M$ if $i$ is odd.

As in the previous subsection, for each $\gamma \in W^k_{\ell,m} \subset \vec E^{2\ell \times 2m}$, we associate the multigraph $G_\gamma$ introduced in Definition \ref{def1}.  We also partition  $W^k_{\ell,m}$ into isomorphism classes exactly as in the previous subsection. We define a canonical element in each isomorphic class thanks to the lexicographic order.

We note however that for all $\gamma \in W^k_{\ell,m}$, the scalar $P_k(\gamma)$ does not depend on the value of $(\gamma_{i,2k_i -1}, \gamma_{i,2k_i})$. We thus need to introduce a new multigraph for elements in $W^k_{\ell,m}$. This multigraph is $G^k_\gamma = \cup_{i} (G_{\gamma'_i} \cup G_{\gamma''_i})$ where $G_{\gamma'_i}$, $G_{\gamma''_i}$ are as in Definition  \ref{def1}. More precisely, the vertex set $G^k_\gamma$ of $V^k_\gamma = \cup_{i} (V_{\gamma'_i} \cup V_{\gamma''_i}) = \{ v_{i,t} : (i,t) \in [2m] \times [2\ell] : t  \notin \{2k_i-1,2k_i \}) \}$ and the set of visited pairs of half-edges is $\cE^k_\gamma = \cup_i (\cE_{\gamma'_i} \cup \cE_{\gamma''_i}) = \{ \{ \gamma_{i,2t-1}  , \gamma_{i,2t}  \} :   (i,t) \in [2m] \times [\ell], t \ne k_i \}$. 

\begin{lemma}\label{le:isopath0}
Let $\gamma \in W^k_{\ell,m}$ with $|V^k_\gamma| = s $ and $|\cE^k_\gamma| = a$. If $g = a - s +1$, then $\gamma$  is isomorphic to at most 
$
  n^s (d-1)^{4m} \PAR{ d (d-1) }^{ s} (d-1)^{2g -1} 
$ 
paths in $W^k_{\ell,m}$.    
\end{lemma}
\begin{proof}
The proof of lemma \ref{le:isopath} implies that $\gamma$ is isomorphic to at most 
$
  n^{s'} (d-1)^{2m} \PAR{ d (d-1) }^{ s} (d-1)^{2g' -1} 
$ 
where $s' = |V_\gamma|$, $a' = |\cE_\gamma|$ and $g' = a' - s' +1$. Since $v_{i,2t+1} = v_{i,2t}$, we have $V_\gamma = V^k _\gamma$ and thus $s' =s$. Also, $a' \leq a + 2m$. Hence, $g' \leq g + 2m$. The claim follows.
\end{proof}

We have the following upper bound on the number of isomorphism classes. This lemma is  a variant of \cite[Lemma 18]{BLM}.

\begin{lemma}\label{le:enumpath0}
Let $\cW^k_{\ell,m} (s,a) $ be the subset of canonical paths with $|V_\gamma^k| = s$, $|\cE_\gamma ^k|= a$. Let $g = a - s +1$. If $g \leq 0$, $\cW^k_{\ell,m} (s,a)$ is empty. If $g \geq 1$, we have 
$$
| \cW^k_{\ell,m} (s,a) | \leq  (4 \ell  m )^{12 m g  + 16 m}.
$$
\end{lemma}

\begin{proof}
Let $\gamma \in W_{\ell,m} ^k$. By assumption, for each $i \in [2m]$, $\gamma_i$ is tangled and non-backtracking. It follows that either $G_{\gamma'_i} \cup G_{\gamma''_i}$ is a connected graph with a cycle or both $G_{\gamma'_i}$ and $G_{\gamma''_i}$ contain a cycle (see Figure \ref{fig:Gamma3}).  Notably,   any connected component of $G^k_{\gamma}$ has a cycle, it follows that $|V^k_\gamma|  \leq |\cE^k_\gamma| $. It gives the first claim.

For the second claim, we adapt the proof of Lemma \ref{le:enumpath} and use the same terminology. For $i \in [2m]$, we define for $t \in [\ell]\backslash\{k_i\}$, $x_{i,t} = ( \gamma_{i,2t-1}, \gamma_{i,2t} )$. We then explore the sequence $(x_{i,t})$, $(i,t) \in T = \{ (i,t) \in [2m] \times [\ell] : t \ne k_i \}$ in lexicographic order. We denote by $(i,t)_-$ the preceding element in $T$ for the lexicographic order (with the convention $(1,1)_- = (1,0)$). For $(i,t) \in T$, we set $y_{i,t} = \{  \gamma_{i,2t-1}, \gamma_{i,2t} \}$. For each $(i,t) \in T$, we build a growing spanning forest $F_{i,t}$ of the graph visited so far as follows. The forest $F_{1,0}$ has a single vertex $\gamma_{1,1} = (1,1)$. By induction, for $(i,t) \in T$,  if the addition of $y_{i,t}$ to $F_{(i,t)_-}$  creates a cycle, we set $F_{i,t} = F_{(i,t)_-}$, and we say that  $y_{i,t}$ is an excess edge. Otherwise, we say that $(i,t)$ is a first time, that $y_{i,t}$ is a tree edge, and we define $F_{i,t}$ as the union of $F_{(i,t)_-}$ and $y_{i,t}$.

Let $p$  the number of connected components of $G^k_\gamma$. Since $F_{2m,\ell}$ is a spanning forest of $G^k_\gamma$, there are $a -s + p  = g + p-1$ excess edges. Let $g_i$ be the number of excess edges in the $i$-th component. We have $\sum_i g_i = g + p-1$. Besides, each connected component of $G_\gamma$ has a cycle, and thus $g_i \geq 1$. It follows that there are at most  $g$ excess edges in each connected component of $G_{\gamma}$.

We may now repeat the proof of Lemma \ref{le:enumpath}. The only difference is that, for each $i$, we use that $\gamma'_i$ and $\gamma''_i$ are tangled free, it gives short cycling times and long cycling times for both $\gamma'_i$ and $\gamma''_i$.  We also need a starting mark for $\gamma''_i$ equal to $(\gamma_{i,2k_i-1},\gamma_{i,2k_i}, \gamma_{i,2\tau-1})$ where  $(i,\tau)$ is the next time that $y_{i,\tau}$ will not be a tree edge of the forest $F_{(i,k_i+1)_-}$ constructed so far. Then, for each $i$, there are at most $2$ short cycling times and $2 (g-1)$ long cycling times (since each connected component has most $g$ excess edges). There are at most $ \ell  ^{4m g}$ ways to position these times. The number of distinct half-edges $\gamma_{i,t}$ in $\gamma$ is at most $h = 4 \ell m$. Arguing as in the proof of Lemma \ref{le:enumpath}, we get that    
$$
| \cW^k_{\ell,m} (s,a) | \leq \ell^{4 m g} ( s h^2 \ell )^{4m}    ( h^2 ) ^{4m (g-1)} (2a) ^{2m} ( h^3  )^{2m} , 
$$
where the factor $(h^3 )^{2m}$ accounts for the extra  starting marks of $\gamma''_i$.  Using $s \leq 2\ell m$, we obtain the claimed statement. \end{proof}

For $\gamma \in W^k_{\ell,m}$, the average contribution of $\gamma$ in \eqref{eq:trDeltak2} is
\begin{equation*}\label{eq:defmug0}
\mu_k(\gamma) =  \dE P_k (\gamma) =  \dE \prod_{i=1}^{2m}  \prod_{t=1}^{k_i-1}  M^{\veps_i}_{\gamma_{i,2t-1}  \gamma_{i,2t}}   \prod_{t=k_i+1}^{\ell} M^{\veps_i}_{\gamma_{i,2t-1}}.
\end{equation*}
Note that if $\gamma \sim \gamma'$ then $\mu_k (\gamma) = \mu_k (\gamma')$. 
\begin{lemma}\label{le:meanpath0}
There is a universal constant $c > 0$ such that, if $6 \ell m \leq \sqrt{ d n}$ and $\gamma \in W^k_{\ell,m} $ with $|V^k_\gamma| = s$, $|\cE^k_\gamma |= a$ and $g = a - s +1$, we have
$$
\ABS{ \mu_k(\gamma)} \leq  c^{g+m}  \PAR{\frac{ 1 }{dn} }^a.  
$$
\end{lemma}
\begin{proof} 
We adapt the proof of Lemma \ref{le:meanpath}. Let $\cE_i $ be the set of inconsistent edges of $\cup_i (\gamma'_i, \gamma''_i)$. Using the terminology of Lemma \ref{le:enumpath0}, a new inconsistent edge can appear at the the start of a non-empty sequence of first times or at a first visit of an excess edge. Every such step can create at most $2$ new inconsistent edges. Moreover, every non-empty sequence of
first times started in $\gamma'_i$ or $\gamma''_i$, $i \in [2m]$, either is followed by a first visit of an excess edge or ends $\gamma'_i$ or $\gamma''_i$. There are at $g +p -1 $ excess edges where $p$ is number of connected components of $G^k_\gamma$. By construction $p \leq 2m$. Hence, we find $|E_i| \leq 8 (g +p -1)+ 8m \leq 4 g +24m $. It remains to apply  Proposition \ref{le:exppath}. 
\end{proof}

\begin{proof}[Proof of Proposition \ref{prop:normR}]
We repeat the proof of Proposition  \ref{prop:normDelta}. For $n \geq 3$, we define 
\begin{equation}\label{eq:choicem0}
m = \left\lfloor  \frac{ \log n }{25 \log (\log n)} \right\rfloor.
\end{equation}
Since $\ell \leq \log n$, for this choice of $m$, $\ell m = o ( \log n) ^2$. We will  prove that for some constant $c>0$, for all $ k \in [\ell]$,
\begin{equation}\label{eq:boundS0}
S_k = \sum_{\gamma \in W^k_{\ell,m} }   | \mu_k(\gamma) | \leq (c \ell  m )^{28 m} (d-1)^{2\ell m},
\end{equation}
Then, from \eqref{eq:trDeltak20}, it implies 
$$
\dE  \sum_{k=1} ^\ell \| R^{(\ell)}_{k}  \| ^{2 m} \leq \ell (d-1)^{m}   (c \ell  m )^{28 m} (d-1)^{2\ell m}.
$$  
It remains to use Markov inequality to conclude.

We now check that  \eqref{eq:boundS0} holds. Using Lemma \ref{le:isopath0}, Lemma \ref{le:enumpath0} and Lemma \ref{le:meanpath0}, we obtain, with $g = g(a,s) = a - s+1$,
\begin{eqnarray*}
S_k & \leq & \sum_{s = 1}^{2\ell m} \sum_{a = s } ^{\infty} \max_{ \gamma \in \cW^k_{\ell,m} (s,a)} |\{ \gamma' \in W^k_{\ell,m} : \gamma' \sim \gamma \} | \times  |\cW^k_{\ell,m} (s,a) | \times  \max_{ \gamma \in \cW^k_{\ell,m} (s,a)} \mu_k(\gamma)  \\
& \leq &  \sum_{s=1}^{2\ell m} \sum_{a = s} ^{ \infty } n^s ( d (d-1) ) ^s (d-1)^{4m} (d-1)^{2g-1}   (4 \ell  m )^{12 m g + 16 m} c^{g +  m } \PAR{\frac{ 1}{dn} }^a     \nonumber\\
& = &   \sum_{s=1}^{2 \ell m} \sum_{h=0} ^{ \infty }   (d-1)  ^s (d-1)^{4m}   (d-1)^{2h+1}    (4 \ell  m )^{12 m h + 28 m} c^{h + 1 + m} \PAR{\frac{ 1}{dn} }^{ h}  , 
 \end{eqnarray*}
where at the last tine, we have performed the change of variable $h = a - s = g -1$. We find that for some new constant $c' >0$, for all $n$ large enough,
\begin{eqnarray*}
S_k &\leq &   (c'  \ell  m )^{28 m}   \sum_{s=1}^{ 2 \ell  m } (d-1)^{s} \sum_{h = 0} ^{\infty} \PAR{\frac{ c (d-1)^2 (4 \ell m) ^{12 m} }{ dn } }^{h}
\end{eqnarray*}
For our choice of $m$ in \eqref{eq:choicem0}, we have, for $n$ large enough, 
$
   (4 \ell m) ^{12 m} /  n   \leq n^{-1/25}. 
$
Hence, the above geometric series converges  and  the right-hand side of \eqref{eq:boundS0} is an upper bound for $S_k$. 
\end{proof}

\begin{remark}\label{rk:QR} Markov inequality and \eqref{eq:boundS0}  imply that for any $a >0$, there exists $c >0$ such that the event $\sum_k \| R_k^{(\ell)} \| \leq  ( \log n) ^{c} \PAR{ d -1} ^{\ell}$ has probability at least $1 - n^{-a}$ for all $n$ large enough.
\end{remark}

\begin{remark}\label{rk:IR}
A careful treatment in Lemma \ref{le:meanpath0} of edges visited once allows to prove that \whp $ \| R^{(\ell)}_{k} \| \leq ( \log n) ^{c}\PAR{ d - 1}^{\ell - k/2}.$ This refinement seems useless.
\end{remark}
\subsection{Proof of Theorem \ref{th:FrNB}}

\label{subsec:end}

All ingredients are finally gathered.  We fix some $0 < \kappa < 1/4$ and consider an integer sequence $\ell = \ell(n)$ such that $\ell \sim \kappa \log_{d-1} n$. By Lemma \ref{le:tangle} and Proposition \ref{le:decompBl}, if $\Omega$ is the event that $G(\sigma)$ is $\ell$-tangle free, 
\begin{eqnarray*}
\dP \PAR{|\lambda_2 | \geq \sqrt {d-1} + \veps } & = &\dP \PAR{|\lambda_2 | \geq \sqrt {d-1} + \veps \, ; \, \Omega} + o(1) \\
& \leq & \dP \PAR{ J^ { 1 /  \ell}  \geq   \sqrt {d-1} + \veps  } + o(1), 
\end{eqnarray*}
where $J = \|  \uB^{(\ell)} \|   +  \frac 1{d n}   \sum_{k = 1}^\ell \| R^{(\ell)}_{k} \|.$
On the other end, by Propositions \ref{prop:normDelta}-\ref{prop:normR}, \whp
\begin{align*}
J  &\leq ( \log n )^{15} (d-1)^{\ell/2}   +  \frac {(\log n)^{30} } {d n}  \PAR{ d - 1}^{\ell} \\
&\leq ( \log n )^{15} (d-1)^{\ell/2} + o(1),
\end{align*} since $(d-1)^\ell = n^{\kappa+o(1)} $. Finally, since $\ell \leq \log n$, $(\log n)^{15 / \ell} = 1 + O ( \log \log n / \log n )$. This concludes the proof.

\begin{remark}
For the proof of \eqref{eq:QFriedB}, we take  $0< \kappa < a/4$, and use Remarks \ref{rk:QB}-\ref{rk:QR}. We then choose $\veps = c \log n /(\log \log n)$ with $c$ large enough in the above argument. 
\end{remark}
\section{New eigenvalues of random lifts}

\label{sec:introbis}

We now introduce the model of random lifts and present an analog of  Theorem \ref{th:Fr} in this new context. Notation is independent of the previous section, but they are kept similar to help the reader.

Let us first introduce an abstract terminology for graphs. Let $V$ and $\vec E$ be countable sets. Elements of $V$ are called \emph{vertices} and elements of $\vec E$ are \emph{half-edges} or \emph{directed edges}. We assume that $\vec E$ is a set of even cardinality and that $\vec E$  comes with a matching  $\iota : \vec E \to \vec E$ ($\iota^2(e) = e$ and $\iota(e) \ne e$ for all $e \in \vec E$). This defines an equivalence classes on $\vec E$, $e \sim f$ iif $e = \iota(f)$ with two elements in each equivalence class. An equivalence class is called an \emph{edge}, the edge set is denoted by $E$. Finally, there is map $o : \vec E \to V$. We interpret $o (e)$ as the origin vertex of the  directed edge $e$ and $t(e) = o( \iota(e))$ as the end vertex of $e$. The quadruple $G  = (V,\vec E,\iota,o)$ will be called a \emph{graph}.  In words, $G$ is the multigraph with vertex set $V$ and edge set $E$, where an edge connects  the origin vertices $o(e)$ with $e$ one of the two directed edges in  the equivalence class of the edge, see Figure \ref{fig:ExLift}. This definition allows loops, which would correspond to directed edges $e \in \vec E$ such that $o(e) = t(e)$, and multiple edges, corresponding to $e \ne f \in \vec E$ such that $(o(e),t(e)) = (o(f),t(f))$.

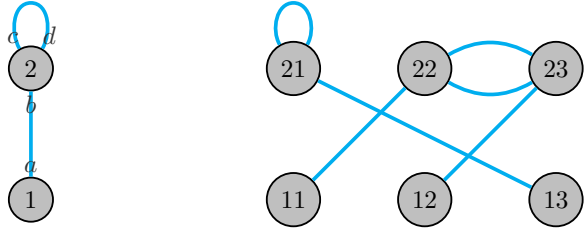
\begin{figure}[htb]
\begin{center}  
\resizebox{8cm}{!}{
\begin{tikzpicture}[main node/.style={circle, draw , fill = lightgray, text = black, thick}]
\node[main node]  at (0,0) (1) {1} ;
\node[main node] at (0,2) (2) {2} ;

\draw[cyan, -,ultra thick] (1) to (2) ; 
\draw[ cyan, -,ultra thick] [out = 60 , in = 0] (2) to (0,3)  [out = 180 , in = 120]  to (2)  ;

\node[text = black!80] at (0,0.5) {$a$} ; 
\node[text = black!80] at (0,1.45) {$b$} ; 
\node[text = black!80] at (0.28,2.5) {$d$} ; 
\node[text = black!80] at (-0.28,2.47) {$c$} ;

\node[main node]  at (4,0) (11) {11} ;
\node[main node] at (4,2) (21) {21} ;
\node[main node]  at (6,0) (12) {12} ;
\node[main node] at (6,2) (22) {22} ;
\node[main node]  at (8,0) (13) {13} ;
\node[main node] at (8,2) (23) {23} ;

\draw[ cyan, -,ultra thick] [out = 60 , in = 0] (21) to (4,3)  [out = 180 , in = 120]  to (21)  ;
\draw[ cyan, -,ultra thick] [out = 30 , in = 150] (22) to (23)  ; \draw[ cyan, -,ultra thick] [out =-30 , in = -150] (22) to (23)  ;

\draw[cyan, -,ultra thick] (11) to (22) ; 
\draw[cyan, -,ultra thick] (12) to (23) ; 
\draw[cyan, -,ultra thick] (13) to (21) ; 

\end{tikzpicture}
}
\caption{Left: the graph $G = (V, E, \iota, o)$ with $V = [2]$, $\vec E = \{a,b,c,d\}$, $\iota(a) = b$, $\iota(c)= d$, $o(a) = 1$, $o(b) = o(c) = o(d) = 2$. Right: a $3$-lift of $G$ with $\sigma_a = (1 \, 2 \, 3)$, $\sigma_c = (1) (2\,3)$ (permutations written in cycle decomposition).} \label{fig:ExLift}

\end{center}\end{figure}

The adjacency matrix  $A$ of $G$ is the symmetric matrix indexed on  $V$ defined for all $u,v \in V$ by
$$
A_{uv}  = \sum_{ e \in \vec E }\IND ( (o(e),t(e) ) = (u,v)). 
$$

The non-backtracking matrix $B$  of $G$ is the matrix  indexed on $\vec E$ defined for all $e, f \in \vec E$ by 
$$
B_{e f}  = \IND ( t(e) = o(f) ) \IND ( f \ne \iota(e)). 
$$

We have the matrix identity
\begin{equation}\label{eq:defBMM}
B = S N,
\end{equation}
where $S_{ef} = \IND_{  f = \iota (e)}$ and $N_{ef} = \IND_{ o(e) = o(f) , e \ne f}$.

We set $r = |\vec E| $  and let  
\begin{equation*}\label{eq:eigBX}
\rho_1 \geq  \dots \geq |\rho_r|
\end{equation*}
be the eigenvalues of $B$ with multiplicities. We shall assume that $G$ is connected and $|\vec E | / 2 = |E| > |V|$. Then, the matrix $B$ is irreducible and its Perron eigenvalue $\rho_1$ is larger than $1$, see \cite{FrKo}.

We now define $n$-lifts and random $n$-lifts of a base graph $G$ (also known as $n$-coverings), for an example, see Figure \ref{fig:ExLift}.

\begin{definition}
Let $G= (V,\vec E,\iota,o)$ be a graph. For integer $n \geq 1$, let  $S_n^{G}$ be the family of permutations $\sigma = (\sigma_e)_{e \in \vec E}$ such that $\sigma_{\iota(e)} = \sigma^{-1}_{e}$ for all $e \in \vec E$. A {\em $n$-lift of $G$} is  a graph $G_n = (V_n, \vec E_n, \iota_n,o_n)$ such that
$$V_n = V \times [n] \quad  \hbox{ and } \quad \vec E_n  = \vec E \times [n],$$
and, for some $\sigma \in S_n^G$, for all $(e,i) \in \vec E_n$
$$
\iota_n (e,i) = (\iota(e),\sigma_e(i))  \quad \hbox{ and } \quad o_n(e,i) = (o(e),i).
$$
We write $G_n = G_n(\sigma)$ for the $n$-lift associated to $\sigma \in S_n^G$.  We say that $G_n$ is a {\em random $n$-lift} if $\sigma  = (\sigma_e)_{e \in \vec E}$ is uniformly distributed on $S_n^{G}$ (that is, the permutations $\sigma_e$, $e \in \vec E$, are uniform on $S_n$ and independent for all $e \ne \iota(e)$).
\end{definition}

For some positive integer $n$, let $\sigma \in S_n^{G}$ and $G_n = G_n (\sigma)$ be the $n$-lift as above. Let $B_n = B_n(\sigma)$ be the non-backtracking matrix of $G_n$.  From \eqref{eq:defBMM}, we have 
\begin{equation}\label{eq:defBMMn}
B_n = S_n N_n,
\end{equation}
where for all $\Be = (e,i) , \Bf = (f,j) \in \vec E_n$, $$(S_n)_{\Be \Bf} =  \IND_{e = \iota(f) } \IND_{  \sigma_e (i) = j}\quad \hbox{ and } \quad (N_n)_{\Be \Bf} =  \IND_{  o(e) = o(f) , e \ne  f} \IND_{ i = j} = (N \otimes I_n)_{\Be\Bf},$$ with $\otimes$ being the usual tensor product, $N$ as in \eqref{eq:defBMM} and $I_n$ being the identity matrix.

We consider the vector subspace $H$ of $\dR^{\vec E_n}$, 
\begin{equation}\label{eq:defH}
H = \SPAN ( \chi_e : e \in \vec E) ,
\end{equation} 
where for $e \in \vec E$, $\chi_e \in \dR^{\vec E_n}$ is given by $\chi_e  ( f, i)  = \IND_{ e = f}$. In words, $H$ is the vector space of vectors $x \in \dR^{\vec E_n}$ which are constant on each edge of $G$: $x(e,i) = x(e,j)$ for all $i,j\in [n]$ and $e \in \vec E$. The dimension of $H$ is $r = | \vec E|$. 

From \eqref{eq:defBMMn}, it is straightforward to check that  $B_n H \subset H$,  $B_n^*  H \subset H$ and the restriction of $B_n$ to $H$ is $B$. It follows that the spectrum of $B_n$ contains the spectrum of $B$ (with multiplicities). We will denote the  $n r -r$ new eigenvalues of $B_n$ by $\lambda_i$ with
\begin{equation}\label{eqM:defneweig}
|\lambda_1| \geq \cdots \geq |\lambda_{nr -r}|. 
\end{equation}
They are the eigenvalues of the restriction of $B_n$ to $H^\perp$ (they depend implicitly on $\sigma \in \sigma^G_n$).  The main result of this section is the following. 

\begin{theorem}\label{thM:main}
Let $G = (V,\vec E,\iota,o)$ be a finite connected graph with $|\vec E|/2 = |E| > |V|$.  Let $n$ be a positive integer and  $G_n$ a random $n$-lift of $G$. Let $B$ and $B_n$ be the non-backtracking matrices of $G$ and $G_n$. If the new eigenvalues of $B_n$  are denoted as in \eqref{eqM:defneweig} and $\rho_1$ is the Perron eigenvalue of $B$,  then, for any $\veps > 0$,  
$$
\lim_{n \to \infty} \dP \PAR{ |\lambda_1|   \geq \sqrt{\rho_1}+ \veps} = 0.
$$ 
\end{theorem}

Under an extra assumption on the graph $G$, Theorem \ref{thM:main} is contained in \cite[Theorem 0.2.6]{FrKo}.
Arguing as above, if $A_n$ is the adjacency matrix of $G_n$, then the spectrum of $A_n$ contains the spectrum of $A$.  We may similarly define the new eigenvalues of $A_n$ as the remaining eigenvalues. If $G$ is a $d$-regular multigraph, then the Ihara-Bass formula \eqref{eq:IharaBass} remains valid. As a by-product of Theorem \ref{thM:main}, we deduce the following corollary, 
\begin{corollary}[Friedman and Kohler \cite{FrKo}]\label{corM:main}
Let $G = (V,\vec E,\iota,o)$ be a finite $d$-regular graph with $d \geq 3$. Let $A_n$ be the adjacency matrix of $G_n$, a random $n$-lift of $G$.  If $\mu_1$ is the largest  new eigenvalues of $A_n$ in absolute value, then, for any $\veps > 0$,  
$$
\lim_{n \to \infty} \dP \PAR{ |\mu_1|   \geq 2 \sqrt{d-1}+ \veps} = 0. 
$$
\end{corollary}

If $d$ is even and $G$ is a bouquet of $d/2$-loops then by standard contiguity results, Corollary \ref{corM:main} implies Friedman's Theorem \ref{th:Fr}. Also, coming back to Theorem \ref{thM:main}, Angel, Friedman and Hoory \cite{AFH} proved that $\sqrt{ \rho_1}$ is the spectral radius of  the non-backtracking operator of $T$, the universal covering tree of $G$.  In \cite{MR1978881}, Friedman conjectures similarly that, for a random $n$-lift, the new eigenvalues of the adjacency matrix have absolute value bounded by $\rho + o(1)$ where $\rho$ is the spectral radius of the adjacency operator of $T$. In the case of $d$-regular multigraphs,  $T$ is an infinite $d$-regular tree and $\rho = 2 \sqrt {d-1}$. It follows that Corollary \ref{corM:main} is consistent with the conjecture. For the general case, the best known result is due to Puder \cite{puder} who proves the upper bound $\sqrt3 \rho +o(1)$ for the new eigenvalues. For a large class of local operators, one may expect that the new eigenvalues of a random $n$-lift will be bounded by $\rho + o(1)$   where $\rho$ is the spectral radius of this operator on $T$. This phenomenon can be thought as the analog of the weak Ramanujan property for arbitrary graphs. Theorem \ref{thM:main} proves this phenomenon for the specific case of the non-backtracking operator.

\section{Proof of  Theorem \ref{thM:main}}

\label{secM:proof}

Theorem \ref{thM:main} is proved exactly as Theorem \ref{th:FrNB}. The main noticeable difference will appear in the path counting argument. Notation are as close as possible to Section \ref{sec:proofF}. 

\subsection{Path decomposition}
\label{subsecM:PD}

In this subsection, we fix $\sigma \in S_n^G$ and consider its associated $n$-lift, $G_n= G_n(\sigma) = (V_n,\vec E_n, \iota_n,o_n)$. Let  $B_n$ be its non-backtracking matrix.  From now on, elements in $\vec E_n$ will be called half-edges.

Let $H$ be as in \eqref{eq:defH} and $P$ be the orthogonal projection onto $H$. Since $B_n H \subset H$ and $B^*_n H \subset H$, we have the decomposition $B_n = P B_n P + ( I - P) B_n (I-P)$. Let $\ell$ be a positive integer. We apply Lemma \ref{le:HR} to  
$
S=  P B_n^\ell P   
$ and $R = B_n^{\ell} - S = (I - P) B_n^\ell (I-P)$.  We find that for the largest new eigenvalue  of $B_n$, 
\begin{equation}\label{eqM:basicl2}
| \lambda_1|^\ell \leq \sup_{x \in H^\perp , \| x \|_2 = 1}   \| B_n^{\ell} x \|_2 .
\end{equation}

We now adapt Definitions \ref{def1}-\ref{def2} to our new setting.

\begin{definition}\label{defM1}
For a positive integer $k$, let $ \gamma = (\gamma_1 , \ldots, \gamma_k) $ with $\gamma_t  = ( e_t,  i_t) \in \vec E_n$.
\begin{enumerate}[-]
\item 
The set of visited vertices, edges and pairs of half-edges are denoted by  $V_\gamma= \{ o_n(\gamma_t) : t \in [k]\}$ and $\cE_\gamma =\{ \{\gamma_{2t-1} ,  \gamma_{2t} \}   : 1 \leq t \leq k/2\}$. The multigraph associated to $\gamma$ is $ G_\gamma$, its vertex set is $V_\gamma$ and to each element $ \{\Be, \Bf \} \in \cE_\gamma$, we associate an edge in $G_\gamma$ between $o_n(\Be)$ and $o_n(\Bf)$ (this can be formalized in the abstract graph terminology). 
\item 
If for all $t \geq 1$, $e_{2t } = \iota(e_{2t-1})$, $ o_n(\gamma_{2t+1}) = o_n(\gamma_{2t})$ and $\gamma_{2t+1} \ne \gamma_{2t}$, the sequence $\gamma$ will be called a {\em non-backtracking path} (that is, $S_{e_{2t} e_{2t-1}} = (N_n)_{\gamma_{2t} \gamma_{2t+1}} = 1$ with $S,N_n$ as in \eqref{eq:defBMM}-\eqref{eq:defBMMn}). 
If  $k = 2 \ell +1$, the set of non-backtracking paths of length $\ell$ is $\Gamma^{\ell}$.  If $\Be, \Bf \in \vec E_n$, we denote by $\Gamma^{\ell}_{\Be \Bf}$  paths in $\Gamma^\ell$ such that $\gamma_1 = \Be$, $\gamma_k  = \Bf$.
The sets $F^{\ell}$ and $F^{\ell} _{\Be \Bf}$ will denote the subsets of tangle-free paths in $\Gamma^{\ell}$ and $\Gamma^{\ell}_{\Be \Bf}$ (see Definition \ref{def2}).
\end{enumerate}
\end{definition}

Importantly, the definitions of  non-backtracking path and tangled paths do not depend on  $\sigma \in S_n^G$. Note also, in Definition \ref{defM1}, that if $\gamma \in \Gamma^{\ell}$, then $(e_1, \ldots, e_{2\ell +1}) \in \vec E ^{2\ell +1}$ is a proper non-backtracking path on the graph $G$:
$$
\prod_{t=1}^{\ell} S_{e_{2t-1} e_{2t}} N_{e_{2t} e_{2t+1}}   = 1.
$$
where $S,N$ are defined in \eqref{eq:defBMM}.

For $e   \in \vec E$, let $M_{e}$  be the permutation matrix associated to $\sigma_{e}$, defined for all $i,j \in [n]$, by
\begin{equation}\label{eqM:defM}
(M_{e})_{i j} =\IND ( \sigma_e(i) =  j) .
\end{equation}
Since $\sigma_{\iota(e)} = \sigma^{-1}_e$, we have $M_{\iota(e)} = M^*_e = M_e ^{-1}$.  If $\Be = (e,i)$, $\Bf  = (\iota(e) ,j)$, we also set 
$$
M_{\Be \Bf} = M_{\Bf \Be} = (M_{e} )_{i j}.  
$$

Let $\ell$ be a positive integer. From \eqref{eq:defBMMn}, for all positive integers $\ell$ and $\Be,\Bf \in \vec E_n$,
$$
(B_n^{\ell}) _{\Be \Bf} = \sum_{\gamma \in \Gamma^{\ell} _{\Be \Bf}} \prod_{s=1}^{\ell} M_{\gamma_{2s-1} \gamma_{2s}},
$$

Moreover, if $G_n$ is $\ell$-tangle-free, we have
\begin{equation}\label{eqM:BkBk}
B_n^{\ell}   = B_n^{(\ell)},
\end{equation}
where we have set for $\Be,\Bf \in \vec E_n$,
$$
(B_n^{(\ell)} ) _{\Be \Bf} = \sum_{\gamma \in F^{\ell} _{\Be \Bf}} \prod_{s=1}^{\ell} M_{\gamma_{2s-1} \gamma_{2s}}.
$$

Let $\chi \in \dR^n$ be the vector with all coordinates equal to $1$ and $\chi^*$ its conjugate transpose of $\chi$. We introduce for $e \in \vec E$, $\Be = (e,i)$, $\Bf  = (\iota(e) ,j)$ in $\vec E_n$,
\begin{equation*}\label{eqM:defW}
\underline M_{e} = M_e - \frac {\chi \chi^*} n \quad \hbox{ and } \quad \underline M_{\Be \Bf} =  \underline M_{\Bf  \Be}  = (\underline M_{e})_{i j},
\end{equation*}
($\underline M_e$ is the orthogonal projection of $M_e$ onto $\chi^\perp$). We define the matrix on $\dR^{\vec E_n}$, defined for all $\Be,\Bf \in \vec E_n$ by
\begin{equation}\label{eqM:defD}
(\underline B_n^{(\ell)} ) _{\Be \Bf} = \sum_{\gamma \in F^{\ell} _{\Be \Bf}} \prod_{s=1}^{\ell} \underline M_{\gamma_{2s-1} \gamma_{2s}}.
\end{equation}
 We use the same telescopic sum decomposition than in \eqref{eq:iopl}, we find 
\begin{eqnarray}\label{eqM:iopl}
(B^{(\ell)}_n) _{\Be \Bf} &  =  & (\uB_n^{(\ell)} ) _{\Be\Bf}   + \frac 1 n  \sum_{\gamma \in F^{\ell} _{\Be\Bf}}  p_k(\gamma),
\end{eqnarray}
where, for all $\gamma \in F^{\ell}$, 
$$
p_k (\gamma) =    \sum_{k = 1}^\ell \prod_{s=1}^{k-1} \underline M_{\gamma_{2s-1} \gamma_{2s}} \prod_{k+1}^\ell M_{\gamma_{2s-1} \gamma_{2s}}
$$

We rewrite \eqref{eqM:iopl} as a sum of matrix products of $\uB^{(k)}_n$ and $B_n^{(k)}$ up to some remainder terms.  Fix $k \in [\ell]$, we decompose a path $\gamma  = (\gamma_1, \ldots, \gamma_{2\ell +1} ) \in \Gamma^{\ell}$ as a path $\gamma'=  (\gamma_1, \ldots, \gamma_{2k -1} ) \in  \Gamma ^{k-1}$, a path $\gamma''= (\gamma_{2k-1}, \gamma_{2k} , \gamma_{2k+1} ) \in \Gamma^{1}$ and a path $\gamma''' = (\gamma_{2k+1}, \ldots, \gamma_{2 \ell +1}) \in \Gamma^{\ell - k}$.  For each $k\in [\ell]$, we denote by $F^{\ell}_{k}$ the set of $\gamma \in \Gamma^{\ell}$ such that, with $\gamma', \gamma'', \gamma'''$ as above, $\gamma'  \in F^{k-1}$, $\gamma'' \in F^{1} = \Gamma^1$ and $\gamma''' \in F^{k-\ell}$. For $\Be,\Bf \in \vec E_n$, we set $F^\ell_{k,\Be \Bf}  \cap \Gamma^{\ell}_{\Be \Bf}$. We have the inclusion $F^{\ell} \subset F^{\ell}_{k}$. We define 
\begin{equation}\label{eqM:defR}
(R_k^{(\ell)} )_{\Be \Bf}   =  \sum_{\gamma \in F^{\ell} _{k, \Be \Bf} \backslash F^\ell_{\Be \Bf}}   p_k(\gamma).
\end{equation}

For $\Be,\Bf \in \vec E_n$, with $\Be = (e,i)$, $\Bf = (f,j)$, we observe that  the cardinality of $\Gamma^1_{\Be \Bf} = F^1_{\Be \Bf}$ is $B_{e f} = (B \otimes \chi \chi^*)_{\Be \Bf}$. The rule of matrix multiplication gives 
\begin{eqnarray*}
 \sum_{\gamma \in F^{\ell} _{k,\Be \Bf}}  p_k(\gamma)  & = &( \uB_n^{(k-1)} ( B \otimes \chi \chi^* )  B_n^{(\ell - k)})_{\Be \Bf}. 
\end{eqnarray*}

From \eqref{eqM:iopl}, we find that  
\begin{eqnarray*}
B^{(\ell)}_n   &= &\uB^{(\ell)}_n    +    \frac 1 {n}  \sum_{k = 1} ^{\ell} \uB^{(k-1)}_n  ( B \otimes \chi \chi^* )   B^{(\ell-k)}_n   -   \frac 1 {n}     \sum_{k = 1}^{\ell} R^{(\ell)}_k .
\end{eqnarray*}

Observe that if $x \in H^\perp$, 
$$
(B \otimes \chi \chi^* ) x = 0 \quad \hbox{ and } \quad B_n H^{\perp} \subset H^{\perp}. 
$$
Hence, if $G_n$ is $\ell$-tangle free and $ x  \in H^\perp$,  then, from \eqref{eqM:BkBk},  we find
\begin{eqnarray*}
  B^{(\ell)}_n  x  &   = &  \uB^{(\ell)}_n   x   -    \frac 1 {n}     \sum_{k = 1}^{\ell} R^{(\ell)}_k x.
\end{eqnarray*}
Putting this last inequality in \eqref{eqM:basicl2}, we arrive at the following statement. 

\begin{proposition}\label{leM:decompBl}
Let $\ell \geq 1$ be an integer and $\sigma \in S_n^G$ be such that $G_n = G_n(\sigma)$ is $\ell$-tangle free. Then, if $\lambda_1$ is the largest new eigenvalue of $B_n = B_n (\sigma)$, we have
$$
|\lambda_1 | \leq  \PAR{ \| \uB^{(\ell)}_n   \|   +    \frac 1 {n}     \sum_{k = 1}^{\ell} \| R^{(\ell)}_k \|}^{1/\ell}.$$
\end{proposition}

\subsection{Computation on random lifts}

\label{subsecM:CM}

In the remainder of this section, $G_n= G_n(\sigma)$ where $\sigma$ is uniform on $S_n^G$.  Our first lemma checks that \whp$G_n$ is $\ell$-tangle free if $\ell$ is not too large. 
The maximal degree in the base graph $G$ is defined as  
\begin{equation}\label{eq:deftheta}
d  = \max_{ v \in V} \ABS{\BRA{ e \in \vec E : o(e) = v }}.
\end{equation}

\begin{lemma}\label{leM:tangle}
Assume that $d \geq 3$. Let $n$ and $\ell$ be positive integers. Let $\sigma$ be uniformly distributed on $S_n^G$. Then $G_n = G_n(\sigma)$ is $\ell$-tangle free with probability $1 - O ((d-1)^{4\ell} / n)$.
\end{lemma}

\begin{proof} For $\bm v = (v,i) \in V_n$, we set  $\vec E_n (\bm v) = \{  \Be \in \vec E_n : o_n (\Be ) = \bm v\} = \{ (e,i) : o(e) = v \}$. Note that $|\vec E_n(\bm v ) | \leq d$. The proof is a simple adaptation of Lemma \ref{le:tangle}. We fix $\bm v  \in V_n$, and we explore its neighborhood sequentially. We start with $D_0 = \vec E_n (\bm v) $.  At stage $t \geq 0$, if $D_t$ is not empty, take an element $\Be_{t+1}  = (e_{t+1},i_{t+1})$ in $D_t$ with $o_n(\Be_{t+1})$ at minimal graph distance  from $\bm v$ (we break ties with lexicographic order). We set $\Bf_{t+1} = \iota_n (\Be_{t+1}) = ( \iota(e_{t+1}) , \sigma_{e_{t+1}} (i_{t+1}))$ .  If  $\Bf_{t+1}  \in D_{t}$, we set $D_{t+1} = D_t \backslash \{ \Be_{t+1}, \Bf_{t+1} \}$, and, otherwise,  $$D_{t+1} =  \PAR{D_t \cup  \vec E_n( o_n(\Bf_{t+1} ))  } \backslash \{ \Be_{t+1} , \Bf_{t+1}\}.$$ 
At stage $\tau \leq n r$, $D_\tau$ is empty, and we have explored the connected component of $\bm v$. Before stage 
$$ T = \sum_{s = 1}^{\ell -1} d (d-1)^{s-1} = O \PAR{ (d-1)^\ell} ,$$
 we have revealed the subgraph spanned by the vertices at distance at most  $\ell$ from $\bm v$. Also, if $\bm v$ has two distinct cycles in its $\ell$-neighborhood, then
$
S(\bm v) = S_{\tau \wedge T}  \geq 2 . 
$
where, for $t \geq 1$,
$$
S_t =  \sum_{s= 1}^{t} \veps_{s} \quad   \hbox{ and } \quad \veps_t = \IND ( \Bf_{t} \in D_{t-1} ). 
$$
At stage $t \geq 0$, for any $e \in \vec E$, at most $t$ values of $\sigma_e$ have been discovered and $
|D_t | \leq d + (d-1)(t-1)$. 

Let $\cF_t$ be the $\sigma$-algebra generated by $(D_0, \cdots, D_{t})$ and $\dP_{\cF_t}$ be its conditional probability distribution. Then, $\tau$ is a stopping time. Also, if $t < \tau \wedge T$, let $A_t = \{ (\iota(e_{t+1}),i) \in D_t : i \in [n]\}$   and $n_t \leq t$ be the number of $s \leq t$ such that $\Bf_s$ or $\Be_s$ is of the form $(  \iota(e_{t+1}),i)$, $i \in [n]$. We find
$$
\dP_{\cF_t} ( \veps_{t+1} = 1) = \frac{|A_{t}|  }{n - n_t}   \leq c \PAR{ \frac{ T }{n} } = q. 
$$
Hence, arguing as in Lemma \ref{le:tangle}, from the union bound,  
$$
\dP \PAR{ G \hbox{ is $\ell$-tangled}}\leq \sum_{\bm v \in V_n}  \dP ( S(\bm v) \geq 2)  \leq \sum_{\bm v \in V_n} q^2 T^2   = O \PAR{ \frac{ (d-1)^{4 \ell}  }{n}}.
$$
This concludes the proof of Lemma \ref{leM:tangle}.  
\end{proof}

%
%
%
%

As in Definition \ref{def:defcons}, for $\gamma = ( \gamma_1, \cdots , \gamma_{k}) \in \vec E_n^{k}$, with $\gamma_t = (e_t, i_t)$, we say that  an edge $\{\Be, \Bf \}   \in \cE_\gamma$ is {\em consistent}, if $\{ t :   \Be \in \{ \gamma_{2 t -1}, \gamma_{2t} \} \} = \{ t :  \Bf \in \{ \gamma_{2t-1}, \gamma_{2t}\} \}  = \{ t :   \{ \Be , \Bf \} =  \{ \gamma_{2t-1}, \gamma_{2t}\}\}$.  It is inconsistent otherwise.  The {\em multiplicity} of $y = \{ \Be , \Bf \} \in \cE_\gamma$ is $\sum_t \IND ( \{\gamma_{2t-1} , \gamma_{2t} \}  = y )$. We have the following analog of Proposition \ref{le:exppath}.
\begin{proposition}
\label{leM:exppath}
Let  $n$ be a  positive integer and let $\sigma$ be uniformly distributed on $S^G_n$. Let $1 \leq k_0 \leq  k \leq \sqrt{n}$ be integers and $ \gamma \in \vec E_n^{2k}$ such that  for all $t \in [k]$,  $e_{2t} = \iota(e_{2t-1})$, where $\gamma_{t}  = (e_t , i_t)$.  For some universal constant $c >0$, we have, 
$$
\ABS{ \dE \prod_{t= 1} ^{k_0}  \underline M_{\gamma_{2t-1} \gamma_{2t}} \prod_{t= k_0+1} ^{k}  M_{\gamma_{2t-1} \gamma_{2t}}} \leq c \, 2^{b} \PAR{ \frac{1 }{ n }}^{a} \PAR{ \frac{  3k   }{ \sqrt{n} }}^{a_1}, 
$$ 
where $a = | E_\gamma |$,  $b$  is the number of  $t \in [k_0]$ such that $\{ \gamma_{2t-1}, \gamma_{2t}\}$ is an inconsistent edge of multiplicity $1$ in $\cE_\gamma$,  and $a_1$ is the number of  $t \in [k_0]$ such that $\{ \gamma_{2t-1}, \gamma_{2t}\}$ is a consistent edge of multiplicity $1$ in $\cE_\gamma$.\end{proposition}

Setting $m = dn$ in Proposition \ref{le:exppath}, the two propositions are similar. This is no surprise, the set of matchings is a subset of the set of permutations. The two proofs are nearly identical . Note however the slight difference between the definitions of consistency (for matchings, for an edge $\{ e, f\}$ to be consistent there is the extra condition $e \ne f$).

\begin{proof}  {\em Step 1: reduction to a single edge of the base graph}. Let $E$ be the set of edges of $\vec E$ (recall that an edge is an equivalence class of two half-edges). For each edge $e \in E$, we choose a distinguished half-edge.   Let $\vec E^+$ be the set of distinguished half-edges. Then, since $e_{2t} = \iota(e_{2t-1})$, up to reversing $\gamma_{2t-1}$ and $\gamma_{2t}$, we may assume without loss of generality that for all $t$, $e_{2t-1} \in \vec E^+$.  Since the random permutations $(\sigma_{e})_{e \in \vec E^+}$ are independent,   we have
$$
\dE \prod_{t= 1} ^{k_0}  \underline M_{\gamma_{2t-1} \gamma_{2t}} \prod_{t= k_0+1} ^{k}  M_{\gamma_{2t-1} \gamma_{2t} } = \prod_{e \in \vec E^+}  \dE \prod_{1 \leq t \leq k_0 \, :\,  e_{2t-1}  = e}  \underline M_{\gamma_{2t-1} \gamma_{2t} } \prod_{k_0  +1 \leq t\leq  k \, : \,   e_{2t-1} = e }  M_{\gamma_{2t-1} \gamma_{2t}  }
$$
 It thus suffices to prove the statement for $\gamma \in \vec E_n^{2k}$ such that for all $t \geq 1$, $e_{2t-1} = e $. We make this assumption in the remainder of the proof.

\noindent  {\em Step 2: case of a single edge of the base graph}.  We may now repeat the proof of Proposition \ref{leM:exppath}. For ease of notation, we set $\gamma_{2t-1} = (e , i_t)$, $\gamma_{2t} = (\iota (e), j_t)$, $\sigma = \sigma_e$,  $M = M_e$ and $\underline M = \underline M_e$. We have $ E_\gamma = \{ y_1, \ldots , y_a  \}$ with $y_{t} = \{ \Be_t , \Bf_t  \}$, $\Be_t = (e , i_{t})$, $\Bf_t = (\iota(e), j_t)$. 
 We set
$$ I = \{ i_t : t \in [a] \} \quad \hbox{ and } \quad J = \{ j_t : t \in [a] \},$$
and $I^c = [n] \backslash I$, $J^c = [n] \backslash J$. We have $| I | \vee |J|  \leq  a$. The multiplicity of $y_t$ is equal to $p_t + q_t$, where $p_t$ is the multiplicity of $y_t$ in $(\gamma_1, \ldots, \gamma_{2k_0})$ and $q_t$ its multiplicity in $(\gamma_{2k_0 +1}, \ldots, \gamma_{2k})$.  We write 
$$
P =  \prod_{t= 1} ^{k_0}  \underline M_{\gamma_{2t-1} \gamma_{2t} } \prod_{t= k_0+1} ^{k}  M_{\gamma_{2t-1} \gamma_{2t} }  = \prod_{t=1}^a  \underline M_{i_tj_t} ^{p_t}  M_{i_tj_t} ^{q_t} . 
$$

Let $T$ be the set of $y_t = \{e_t , f_t \}$ such that $y_t$ is consistent, $p_t = 1$ and $q_t = 0$.  By assumption $|T| = a_1$. Note that for any $t \in T$, $i_t \ne i_s$ and $j_t \ne j_s$ for all $s \ne t$. Let $T^*\subset T$ be the random subset of $t \in T$ such that $\sigma(i_t) \in J^c \cup \{ j_t \}$ and $\sigma^{-1} (j_t) \in I^c\cup \{ i_t \}$. Similarly, let $S \subset T$ be the random subset of $t \in T$ such that $\sigma(i_t) = j_s$ or  $\sigma(i_s) = j_t$ for some $s \in T \backslash \{t \}$.

By construction, if $t \in S$, 
$$
 \underline M_{i_t j_t} ^{p_t}  M_{i_t j_t} ^{q_t} =  \underline M_{i_t j_t}   =  - \frac {1}{n}.
$$
We thus have 
$$
P = (-n)^{-|S|} P^* Q,
$$
where 
$$
P^*  = \prod_{t \in T^*}  \underline M_{i_tj_t}  \quad \hbox{ and } \quad Q =  \prod_{t\notin  S \cup T^* }   \underline M_{i_tj_t} ^{p_t}  M_{i_t j_t} ^{q_t} . 
$$

Let $\cF$  be the $\sigma$-algebra generated by the variables $T^*$ and $\sigma(i_t), \sigma^{-1}(j_t) , t \notin  T^*$. We denote by $\dE_{\cF}$ its conditional expectation. By construction, the variables $S,T^*$ and $Q$ are  $\cF$-measurable. From Jensen's inequality, we get
\begin{equation}\label{eqM:defPQP}
\ABS{ \dE \SBRA{ P}} = \ABS{ \dE \SBRA{(-n)^{-|S|} \, Q  \, \dE_{\cF} \SBRA{P^*} }} \leq   \dE \SBRA{ n^{-|S|} \, |Q| \,  \ABS{ \dE_{\cF} \SBRA{P^*} }}. 
\end{equation}

We start by evaluating $ \dE_{\cF^*} \SBRA{P^*}$ in \eqref{eqM:defPQP}. If $\hat N$ is the number of $t \in T^*$ such that $\sigma(i_t) \ne j_t$, we have 
\begin{eqnarray*}
 P^*   & =&  \PAR{ 1 - \frac 1 n}^{|T^* |  - \hat N} \PAR{ - \frac 1 n}^{\hat N}. 
\end{eqnarray*}
 We now determine the law of $\hat N$ given $\cF$. Let $\hat n = |J ^c| -   \sum_{t \notin T^*}\IND_{ \si(i_t) \in J ^c} $ be the cardinality of elements in $J^c$ whose pre-image has been revealed when the values of $\sigma(i_t), \sigma^{-1}(j_t) , t \notin  T^* $ have been revealed. By counting all possibilities for the values of $\sigma(i_t)$, $t \in T^*$, we find for all $0 \leq x \leq |T^*|$, 
$$
\dP_{\cF} ( \hat N = x) = \frac{{|T^*| \choose x} ( \hat  n )_{  x} }{Z}\quad \hbox{ with } \quad Z = \sum_{x = 0}^{|T^*|} {|T^*| \choose x} ( \hat  n )_{x} ,
$$
where we have used the Pochhamer symbol, $(\hat n)_x = \hat n(\hat n -1) \cdots (\hat n - x + 1)$.  First, since $\hat n \geq n - 2a$ and $a \leq  \sqrt{n}$, we obtain from \eqref{eq:factpow}, for some $c >0$, 
$$
Z \geq c \sum_{x = 0}^{|T^*|} {|T^*| \choose x}  n^{x}   = c ( 1+ n)^{|T^*|} \geq  c n^{|T^*|}. 
$$
From what precedes, we get, 
\begin{eqnarray*}
\dE_\cF \SBRA{P^*}  &=& \frac 1 Z \sum_{x = 0} ^{|T^*|} {|T^*| \choose x}  (n )_{x} \PAR{ 1 - \frac 1 n}^{|T^* |  - x} \PAR{ - \frac 1 n}^{x}\\
 & = &  \frac 1 Z \sum_{x = 0} ^{|T^*|} {|T^*| \choose x}  \prod_{y=0}^{x-1} ( y -  \hat n) \PAR{ 1 - \frac 1 n}^{|T^* |  - x} \PAR{ \frac 1 n}^{x} \\
& = &   \frac 1 Z \dE  \prod_{y=0}^{N-1} ( y -  \hat n ),
\end{eqnarray*}
where $N$ has distribution $\Bin(|T^* | , 1 / n)$. By Lemma \ref{le:bin}, applied to $z = 1$,  $k =|T^* | $, $p = 1 / n$ and $q = 1 / \hat n$, we deduce that, for some $c >0$, 
\begin{equation}\label{eqM:ddett}
\dE_\cF \SBRA{P^*} \leq c \PAR{ \frac{\veps}{n} }^{|T^*|} ,
\end{equation}
with $\veps =   3 a  / \sqrt n$ (since $3 |T^*| \sqrt{ z /2 } \leq 3a$).

We now evaluate $|Q|$ in \eqref{eq:defPQP}. Let $\cF_t$ be the $\sigma$-algebra generated by $ \sigma(i_s), \sigma^{-1}(j_s) , s \ne t$. For any  $ t \in [ a]$, 
\begin{equation}\label{eqM:MeftP}
\dP_{\cF_t} ( M_{i_t j_t} = 1 )  = \frac{\IND( \Omega^c _t) } { n_t } \leq \frac{ 1} { n^*}, 
\end{equation}
where $n^* = n - a +1$, $n_t = n - | \{  \sigma(i_s) : s \ne t \}|$ and $\Omega_t \in \cF_t$ is the event that for some $s \ne t$, $\sigma(i_s) = j_t$ or $\sigma^{-1}(j_s) = i_t$. 
We get, for $p \geq 2$ and $q\geq 0$,
\begin{equation} \label{eqM:rdj20}
\dE_{\cF_t}  | \underline M_{i_t j_t} ^p M_{i_t j_t}^q |  \leq \dE_{\cF_t}   | \underline M_{e_t f_t} |^2   \leq\PAR{ 1 - \frac 1 n }^2  \frac{ 1} {n^*}  + \frac{1}{n^2}   \PAR{ 1 -  \frac{ 1} { n^*} } \leq \frac{1}{n^*}. 
\end{equation}

Similarly, if $q \geq 1$, 
\begin{equation} \label{eqM:rdj01}
\dE_{\cF_t}  | \underline M_{i_t j_t}^p M_{e_t f_t}^q |  \leq \dE_{\cF_t}   M_{e_t f_t}    \leq \frac{1}{n^*}. 
\end{equation}

We also have the weak bound,
\begin{equation} \label{eqM:rdj10}
\dE_{\cF_t} |   \underline M_{i_t j_t} |    \leq \PAR{ 1 - \frac 1 n }  \frac{ 1} {n^*}  + \frac{1}{n}   \PAR{ 1 -  \frac{ 1} { n^*} } \leq \frac{2}{n^*}. 
\end{equation}

Finally, observe that the variables $S$ and $T^*$ are $\mathcal \cF_t$-measurable for any $t$. On the event $t \in T \backslash \{ S, T^*\}$, we have $M_{i_t j_t}  = 0$. It follows that if $t   \in T \backslash ( S \cup T^*)$, 
\begin{equation} \label{eqM:rdj10*}
\dE_{\cF_t} |   \underline M^{p_t}_{i_t j_t} M_{i_t j_t}^{q_t}|    =\dE_{\cF_t} |   \underline M_{i_t j_t} | = \frac 1 n  \leq \frac{1}{n^*}. 
\end{equation}

We then estimate $\dE \SBRA{  | Q |  \bigm| (S,T^*) }$ as follows. If $y_t$ is such that $p_t \geq 2$, we use \eqref{eqM:rdj20},  if  $q_t \geq 1$, we use \eqref{eqM:rdj01}. If $y_t$ is an inconsistent edge such that $p_t = 1$ and $q_t = 0$, we use \eqref{eqM:rdj10}. Finally, if $t \in T \backslash ( S \cup T^*)$,  we use \eqref{eqM:rdj10*}. From Lemma \ref{le:condexp}, we find that
$$
\dE \SBRA{  | Q |  \bigm| (S,T^*) } =  \dE \SBRA{ \prod_{t\notin  S \cup T^* }   | \underline M_{i_tj_t} ^{p_t}  M_{i_t j_t} ^{q_t}| \bigm| (S,T^*) } \leq 2^b \PAR{\frac{1}{n^*} }^{a - |S| - |T^*|}.
$$

Putting this last bound together with \eqref{eqM:ddett}, we deduce from \eqref{eqM:defPQP} and \eqref{eq:factpow}  that for some $c >0$, 
\begin{eqnarray*}
\ABS{ \dE P   } &\leq &  c \, 2^b \, \dE \PAR{  \frac {1} m }^{a - |T^*| }   \PAR{ \frac{\veps  }{ m} }^{|T^*|}   \\
& = & c  \, 2^b \, \PAR{\frac{1}{m}}^{a} \veps^{a_1} \dE \veps^{-|T \backslash T^*|}.
\end{eqnarray*}

To conclude the proof, we prove that  $\dE \veps^{-|T \backslash T^*|} \leq c$ for some constant $c>0$. The event that $\{|T \backslash T^*| \geq x \}$ is contained in the event that there are $\lceil x / 2 \rceil$ ordered pairs $(s,t)$, $s\ne t$, such  that $\sigma(i_s) = j_t $.  From the union bound, we get
$$
\dP \PAR{|T \backslash T^*| \geq x } \leq \PAR{ \frac{ a^2}{n^*} }^{\lceil x / 2 \rceil}.
$$
Indeed, the factor $( a^2 )^{\lceil x / 2 \rceil}$ accounts for the choices of pair $(s,t)$. The factor $(1/n^*)^{  \lceil x / 2 \rceil}$ is an upper bound on the probability that  $\sigma(i_s) = j_t$ (from  Lemma \ref{le:condexp} and \eqref{eqM:MeftP}). Since $a \leq k \leq \sqrt{n}$ and $\lceil x / 2 \rceil \leq x/2 +1/2$, we get from \eqref{eq:factpow},
$$
\dP \PAR{|T \backslash T^*| \geq x } \leq c \PAR{ \frac{ a}{\sqrt{n} } }^{x}.
$$
Recalling $\veps = 3a / \sqrt m$, we find 
$$
 \dE  \veps^{-|T \backslash T^*|} \leq \sum_{x=0}^\infty \veps^{-x} \dP ( |T \backslash T^*| \geq x) \leq c \sum_{x = 0}^{\infty}   \PAR{\frac 1 3}^{-x} = \frac{3c}{2}.
$$
 This concludes the proof of Proposition \ref{leM:exppath}. \end{proof}

\subsection{Path counting}

\label{subsecM:PC}
In this subsection, we give upper bounds on the operator norms of $\uB_n^{(\ell)}$ and $R^{(\ell)}_{k}$ defined by \eqref{eqM:defD} and \eqref{eqM:defR}. As in Subsection \ref{subsec:PC}, we use the high trace method.

\subsubsection{Operator norm of $\uB_{n}^{(\ell)}$}

We denote by $\| \cdot \|_1$ the $\ell^1$-norm in $\dR^{\vec E}$ and, for $e \in \vec E$, we define the unit vector $\delta_e ( f) = \IND_{e = f}$. From Gelfand's Formula,
$$
\lim_{k \to \infty} \| B^k  \delta_e \|^{1/\ell}_1 = \lim_{k \to \infty} \| (B^*)^k  \delta_e \|^{1/\ell}_1 = \rho_1, 
$$
where $\rho_1 >1$ is the Perron eigenvalue of $B$. In this paragraph, we consider a scalar $\rho >  \rho_1$. Then, there exists a constant $c_\rho \geq 1$ such that for all integers $k \geq 1$ and all $e \in \vec E$, 
\begin{equation}\label{eq:defrho}
\| (B^*)^k  \delta_e \|_1 \leq c_\rho \rho^{k}. 
\end{equation}
We will prove the following proposition. 

\begin{proposition} \label{propM:normDelta}
Let  $\rho > \rho_1$ and $1 \leq \ell  \leq  \log n$ be an integer.  Let $\sigma$ be uniformly distributed on $S^G_n$ and  let $\uB^\ell _n = \uB^\ell_n (\sigma)$ be defined as in \eqref{eqM:defD}.  
Then, \whp  
$$ \| \uB_n^{(\ell)} \| \leq ( \log n) ^{20}   \rho ^{\ell/2}.$$
\end{proposition}
To be precise, in Proposition \ref{propM:normDelta}, $\ell = \ell(n)$ may depend on $n$ but $\rho$ is constant. Beware that what is hidden behind {\em\whp}depends on $\rho$ and $G$. For the main part, we repeat the proof of Proposition \ref{prop:normDelta}. Let $m$ be a positive integer.  Arguing as in \eqref{eq:trDeltak},
\begin{eqnarray}
\|\uB_n^{(\ell)}  \| ^{2 m}  & \leq & \tr \BRA{ \PAR{  \uB_n^{(\ell)}{\uB_n^{(\ell)}}^*}^{m}  } \nonumber\\
& =  &  \sum_{\gamma }   \prod_{i=1}^{2m}  \prod_{t=1}^{\ell} \underline M_{\gamma_{i,2t-1}  \gamma_{i,2t}} ,  \label{eqM:trDeltak}
\end{eqnarray}
where the sum is over all  $\gamma = ( \gamma_1, \ldots, \gamma_{2m})$ such that $\gamma_i = (\gamma_{i,1}, \ldots, \gamma_{i,2\ell+1}) \in F^{\ell}$ (that is, non-backtracking tangle-free path) and for all $i \in [m]$, 
$$
\gamma_{2i,1} = \gamma_{2i+1, 1} \quad \hbox{ and } \quad  \gamma_{2i-1,2\ell+1} = \gamma_{2i, 2\ell+1},
$$
with the convention that $\gamma_{2m+1} = \gamma_{1}$. The product \eqref{eqM:trDeltak} does not depend on the value of $\gamma_{2i-1,2\ell+1} = \gamma_{2i, 2\ell}$, $i \in[ m]$. Moreover, if $\gamma_{2i-1,2\ell}$ and $\gamma_{2i,2\ell}$ are given and $d_{i,\ell}$ is the degree of $o_n(\gamma_{2i-1,2\ell})  = o_n(\gamma_{2i,2\ell})$, then $\gamma_{2i-1,2\ell+1} = \gamma_{2i, 2\ell+1}$ can take $(d_{i,\ell} - 1 - \IND_{\gamma_{2i-1,2\ell} \ne \gamma_{2i, 2\ell}} )$ possibles values. Hence, by setting $\gamma'_{2i,t} = \gamma_{2i, 2 \ell +1 -t }$ on the right-hand side of \eqref{eqM:trDeltak}, we get 
\begin{equation}\label{eqM:trDeltak2}
\| \uB_n^{(\ell)}  \| ^{2 m}  \leq   \sum_{\gamma \in W_{\ell,m}}  q(\gamma) \prod_{i=1}^{2m}  \prod_{t=1}^{\ell} \underline M_{\gamma_{i,2t-1}  \gamma_{i,2t}} , 
\end{equation}
where $W_{\ell,m}$ is the set of $\gamma = ( \gamma_1, \ldots, \gamma_{2m}) \in   \vec E_n^{2 \ell \times 2m}$ such that for all $i \in [2m]$, $\gamma_i = (\gamma_{i,1}, \ldots, \gamma_{i,2\ell})$ is non-backtracking and tangle-free, and for all $i \in [m]$, 
\begin{equation}\label{eqM:defbound}
o_n(\gamma_{2i,1}) = o_n(\gamma_{2i-1, 2\ell})  \quad \hbox{ and } \quad  \gamma_{2i+1,1} = \gamma_{2i, 2\ell},
\end{equation}
with the convention that $\gamma_{2m+1} = \gamma_{1}$. Finally,
\begin{equation}\label{eqM:defqg}
q(\gamma)   = \prod_{i=1}^{m} ( d_{i,\ell} - 1 - \IND_{\gamma_{2i-1,2\ell} \ne \gamma_{2i, 2\ell}}  ) \leq (d-1)^{m}.
\end{equation}

For $\gamma \in \vec E_n^{2 \ell \times 2m}$, we define $V_\gamma$, $E_\gamma$, $G_\gamma$ as in Definition \ref{defM1}.  For $\gamma, \gamma' \in \vec E_n^{2 \ell \times 2m}$, we consider the isomorphism class $\gamma \sim \gamma'$, if there exists permutations $(\alpha_v)_{v \in V} \in  S_n ^{V} $ such that, with $\gamma_{i,t} = (e_{i,t},j_{i,t})$, $\gamma'_{i,t} = (e'_{i,t},j'_{i,t})$,  for all $(i,t) \in [2m] \times [2\ell]$, $e'_{i,t} = e_{i,t}$ and $j'_{i,t} = \alpha_{o(e_{i,t})} ( j_{i,t})$.  We may define a canonical element in each isomorphic class as follows. For $v  \in V$, we define $V_\gamma (v) = \{ j_{i,t} : o(e_{i,t}) = v , (i,t) \in [2m] \times [2\ell] \}$.  We say that a path $\gamma \in \vec E_n^{2 \ell \times 2m}$ is canonical if for all $v \in V$, $V_\gamma (v)  = \{ (v,1), \ldots, (v,|V_\gamma(v)|) \}$  and the elements of $V_\gamma(v)$  are visited in the lexicographic order. Note that $\gamma \in W_{\ell,m}$ and $\gamma' \sim \gamma$ implies that $\gamma' \in W_{\ell,m}$. The number of elements in each isomorphic class is easily upper bounded.

\begin{lemma}\label{leM:isopath}
Let $\gamma \in \vec E_n^{2 \ell \times 2m}$ with $|V_\gamma| = s $. Then $\gamma$  is isomorphic to at most 
$
  n^s
$ 
elements in $\vec E_n^{2 \ell \times 2m}$.  
\end{lemma}

\begin{proof}
For $v \in V$, let $s_v = |V_\gamma (v)|$. By construction, $\sum_v s_v = s$ and $\gamma$ is isomorphic to 
$
\prod_v  n ! /(n - s_v) !   \leq \prod_v n^{s_v} = n ^s
$
elements.
\end{proof}

We now upper bound the number of isomorphic classes in $W_{\ell,m}$. The next lemma contains the main noticeable difference with Subsection \ref{subsec:PC}.

\begin{lemma}\label{leM:enumpath}
Let $\cW_{\ell,m} (s,a) $ be the subset of canonical paths in $W_{\ell,m}$ with $|V_\gamma| = s$ and $|E_\gamma |= a$. Let $g = a - s +1$. If $g < 0$ then $\cW_{\ell,m} (s,a)$ is empty. Otherwise,  there exists a constant $c$ depending on $\rho$ and $G$ such that, 
$$
| \cW _{\ell,m} (s,a) | \leq  \rho^{s} (c \ell  m )^{8 m g + 10 m}.
$$
\end{lemma}

\begin{proof}
Let $\gamma \in W_{\ell,m}$ with $|V_\gamma| = s$, $|E_\gamma| = a$ We set $\gamma_{i,t} = (e_{i,t}, j_{i,t})$ and $v_{i,t} = o(e_{i,t})$. From \eqref{eqM:defbound}, $G_\gamma$ is connected and the first statement follows. For the bound on $\cW _{\ell,m} (s,a) $, we start by recalling some definitions used in the proof of Lemma \ref{le:enumpath}. For $ (i,t) \in [2m] \times [\ell]$, let $x_{i,t} = ( \gamma_{i,2t-1}, \gamma_{i,2t} )$ and $y_{i,t} =  \{ \gamma_{i,2t-1}, \gamma_{i,2t} \} \in \cE_\gamma$. We explore the sequence $(x_{i,t})$ in lexicographic order denoted by $\preceq$. We say that $(i ,t)$ is a {\em first time}, if $v_{i,2t}$ has not been seen before (that is $v_{i,2t} \ne v_{i', t'}$ for all $(i',t') \preceq (i,2t)$). The edge $y_{i,t}$ will then be called a {\em tree edge}. By construction, the graph with edge set $\{\{v_{i,2t-1}, v_{i,2t} \} : (i,t) \hbox{ first time}\}$ is a tree. An edge  $y_{i,t}$ which is not a tree edge, is called an {\em excess edge}, and we then say that $(i,t)$ is an {\em important time}. We denote the sequence of important times by $(i,t^i_{q})$, with $ q \in[ q_i]$ and set $t^i_{0} = 0$, $t^i_{q_i+1} = \ell+1$. We observe that between $(i,t^i_{q-1})$ and $(i,t^i_q)$, there is a $t^i_{q-1} < \tau^i_{q-1} \leq t^i_{q}$ such that  $y_{i,t}$ is a tree edge for all $t^i_{q-1}<  t < \tau^i_{q-1}$ and then $(i,t)$ is a first time for all $\tau^i_{q-1} \leq t <  t^i_{q}$. We set $s^i_{q-1} = t^i_q - \tau^i_{q-1}$.   Since every vertex in $V_\gamma$ different from $v_{1,1}$ has its associated tree edge,
\begin{equation*}\label{eqM:defchi}
 \ABS{ \BRA{ y \in \cE_\gamma : \hbox{ $y$ is an excess edge}} } = a - s +1 = g,
\end{equation*}
and
\begin{equation}\label{eqM:sumsiq}
\sum_{i = 1}^{2m}\sum_{q=0}^{q_i} s^i_{q} = s-1.
\end{equation}
We mark the important times $(i,t^i_q)$, $q \in [q_i]$  by the vector $(\gamma_{i,2t^q_i},\gamma_{i,2\tau^i_q-1},\veps_{i,q})$ where, by convention, for $q = q_i$, $\gamma_{i,2\ell+1} = \gamma_{i+1, 1}$ and $\veps_{i,q} = (e_{i,2\tau^i_q},  \ldots, e_{i,2 t^i_{q+1} -1})$ is the projection of the  path $(\gamma_{i,2\tau^i_q}, \cdots, \gamma_{i,2 t^i_{q+1} -1})$  on $G$. Similarly, for $q=0$, we add a starting mark $(\gamma_{i,2\tau^i_0-1},\veps_{i,0})$ with $\veps_{i,0} = (e_{i,2\tau^i_0 },  \cdots, e_{i,2 t^i_{1} -1})$.  We observe two facts (i) there is a unique non-backtracking path between two vertices of a tree, and (ii) if $(i,t)$ is a first time then  $\gamma_{i,2t} = ( e_{i,2t} , m+1)$ and $\gamma_{i,2t+1} = (e_{i,2t+1} , m+1)$, where $m$ is the number of  first times $(i',t') \preceq (i,t)$, such that $v_{i',2t'-1} = v_{i,2t}$ (since $\gamma$ is canonical). It follows that we can reconstruct $\gamma \in \cW_{\ell,m}(s,a)$ from the starting marks, the position of the important times and their marks. It gives a first encoding.

This encoding may have large number of important times. To improve it, we partition important times into three categories, {\em short cycling}, {\em long cycling} and {\em superfluous} times. For each $i$, consider the first time $(i,t_1)$ such that $v_{i,2t_1} \in \{ v_{i,1}, \ldots, v_{i,2t_1-1} \}$. If such time exists, the short cycling time $(i,t)$ is the last important time $(i,t) \preceq (i,t_1)$. We have $t = t^i_q$ for some $1 \leq q \leq q_i$.  Let $1 \leq \sigma \leq t_1$ be such that $v_{i,2t_1} = v_{i,2\sigma -1}$. By assumption, $C_i = (\gamma_{i,2\sigma-1},\cdots, \gamma_{i,2t_1})$ will be the unique cycle visited by $\gamma_i$. We denote by $(i,\hat t) \succeq (i,t)$ the first time that $\gamma_{i,2\hat t-1}$ in not in $C_i$ (by convention $\hat t = \ell+1$ if $\gamma_i$ remains on $C_i$). We modify the mark of the short cycling time $(i,t) = (i,t^i_q)$ and redefine it as $(\gamma_{i,2t^q_i}, \gamma_{i,2t_1}, \hat  t, \gamma_{i,2\hat t  -1},\gamma_{i,2\hat \tau^i_q  -1} ,\hat \veps_{i,q})$, where $(i,\hat \tau^i_q)$, $\hat \tau^i_q \geq \hat t$, is the first time that $y_{i,\hat \tau^i_q}$ is not on the tree constructed so far and  $\hat \veps_{i,q} = (e_{i,2\hat\tau^i_q},  \cdots, e_{i,2 t^i_{q+1} -1})$. We also redefine $s^i_q $ as $t^i_{q+1} - \hat \tau^i_q$. Important times $(i,t)$ with $1 \leq t < \sigma$ or $\tau \leq t \leq \ell$ are called long cycling times. The other important times are called superfluous. Then, for each $i \in [2 m]$, the number of long cycling times $(i,t)$ is bounded by  $g-1$ (since there is at most one cycle, no edge of $\cE_\gamma$ can be seen twice outside those of $C_i$, the $-1$ coming from the fact the short cycling time is an excess edge).  We note also that, if $(i,q)$ is a short cycling time and $(i,q')$ is the next long cycling time then $\hat \tau^i_q$ and $\hat \veps_{i,q}$ are equal to $\tau^i_{q'-1}$ and $\veps_{i,q'-1}$. Moreover, $s^i_p = 0$ for all $q \leq p < q' -1$.

It gives our second encoding. We can reconstruct $\gamma$ from the starting marks, the positions of the long cycling  and the short cycling times and their marks. For each $i$, there are at most $1$ short cycling time and $ g-1$ long cycling times. There are at most $ \ell ^{2m g }$ ways to position them. There are at most $\ell^{2m (g+1)}$ possibilities for the $\tau_i^q$'s (of the starting marks, long and short cycling times).  There are $r$ possible choices of $e_{1,1}$.  The number of distinct $\gamma_{i,t}$ in $\gamma$ is at most $h = 4 \ell m$. Also for any integer $s\geq 1$ and $e \in \vec E$,
$
\| (B^*)^s \delta_e \|_1
$ is equal to the number proper of non-backtracking walks $(e_1, \ldots , e_{2s+1})$ in $G$ with $e_1 = e$.  Thus, there are at most $h^2 \|(B^*)^{s^i_q} \delta_{e_{i,2\tau^i_q-1}}\|_1$ different possible marks for the long cycling time $t^i_q$ and $h^4\ell \|(B^*)^{s^i_q} \delta_{e_{i,2\hat \tau^i_q-1}}\|_1$ possible marks for the short cycling time $t^i_q$. Similarly, there are at most $h\|(B^*)^{s^0_q} \delta_{e_{i,2\hat \tau^i_0-1}}\|_1$ possibilities for the $i$-th starting mark.  Finally, from \eqref{eq:defrho} and \eqref{eqM:sumsiq},
$$
\prod_{i = 1}^{2m} \prod_{ q = 0} ^{q_i} \|(B^*)^{s^i_q} \delta_{e_{i,2  \tau^i_q-1}}\|_1 \leq c_\rho^{2 m (g+1)}\rho^{s -1}.
$$
We deduce that    $| \cW _{\ell,m} (s,a) | $ is at most
$$
 r  \ell ^{2m g} \ell^{2m (g+1)}\PAR{h^{2}}^ {2m(g-1)} \PAR{h^4\ell}^{2m} h^{2m} c_\rho^{2 m (g+1)}\rho^{s -1}. 
$$
This concludes the proof by setting $c$ large enough. \end{proof}

For $\gamma \in W_{\ell,m}$, the average contribution of $\gamma$ in \eqref{eqM:trDeltak2} is
\begin{equation}\label{eqM:defmug}
\mu(\gamma) =   \dE \prod_{i=1}^{2m}  \prod_{t=1}^{\ell} \underline M_{\gamma_{i,2t-1}  \gamma_{i,2t}}.
\end{equation}

Note that if $\gamma \sim \gamma'$ then $\mu(\gamma) = \mu(\gamma')$. The proof of Lemma \ref{le:meanpath} gives immediately the following. 
\begin{lemma}\label{leM:meanpath}
There is a constant $c > 0$ such that, if $2 \ell m \leq \sqrt{  n}$ and $\gamma \in W_{\ell,m} $ with $|V_\gamma| = s$, $|\cE_\gamma |= a$ and $g = a - s+1$, we have
$$
\ABS{ \mu(\gamma)} \leq  c^{g + m}  \PAR{\frac{ 1 }{n} }^a   \PAR{\frac{ (6\ell m)^2 }{n} } ^{  (a - 2g - (\ell+2) m)_+}.  
$$
\end{lemma}

\begin{proof}[Proof of Proposition \ref{propM:normDelta}]
For $n \geq 3$, we define 
\begin{equation}\label{eqM:choicem}
m = \left\lfloor  \frac{ \log n }{17 \log (\log n)} \right\rfloor.
\end{equation}
For $n$ large enough, $m$ is positive, $n ^{  1 / (2m) } = o ( \log n )^{9}$ and $\ell m = o ( \log n) ^2$. Hence, arguing as in the proof of Proposition  \ref{prop:normDelta}, from \eqref{eqM:trDeltak2} and Markov inequality, it suffices to prove that 
\begin{equation}\label{eqM:boundS}
S = \sum_{\gamma \in W_{\ell,m} }   | \mu(\gamma) | \leq n  (c \ell m )^{10  m} \rho^{ \ell' m},
\end{equation}
where $\ell' = \ell +2$ and $\mu(\gamma)$ was defined in \eqref{eqM:defmug}. Using Lemma \ref{leM:isopath}, Lemma \ref{leM:enumpath} and Lemma \ref{leM:meanpath}, we find, with $g = g(a,s) = a - s +1$, for some new constant $c'>0$, all $n$ large enough,
\begin{eqnarray*}
S & \leq &  \sum_{s=1}^{\infty} \sum_{a = s - 1} ^{ \infty } n^s    \rho^{s} (c \ell  m )^{8 m g + 10  m} c^{g + m}  \PAR{\frac{ 1 }{n} }^a   \PAR{\frac{ (6\ell m)^2 }{n} } ^{  (a -2g  - \ell' m)_+} \nonumber\\
& \leq &   \sum_{s=1}^{\infty} \sum_{g=0} ^{\infty} n  (c'  \ell  m)^{10 m}    \rho^{s}  \PAR{\frac{ (c' \ell m) ^{8m} }{ n } }^{g} \PAR{\frac{ (6\ell m)^2 }{n} } ^{  (s - g - 1 -  \ell ' m)_+},\\
& \leq &  S_1 + S_2  + S_3 \label{eqM:Sboundddd}
 \end{eqnarray*}
where $S_1$ is the sum over $\{ 1 \leq s \leq \ell' m, g \geq 0 \}$, $S_2$ over $\{\ell' m+1 \leq s , 0 \leq g \leq s - 1 - \ell' m \}$, and $S_3$ over $\{ \ell' m+1 \leq s , g \geq s - \ell' m\}$. Since $\rho > 1$, we have
\begin{eqnarray*}
S &\leq &     (c' \ell   m)^{10 m}   \sum_{s=1}^{ \ell' m} \rho^{s} \sum_{g=0} ^{ \infty } \PAR{\frac{  (c' \ell m) ^{8m} }{ n } }^{g} \\
& \leq & \PAR{\frac{\rho}{ 1-\rho}}  n (c'  \ell  m )^{10 m} \rho^{\ell'm} \sum_{g=0} ^{ \infty } \PAR{\frac{  (c' \ell m) ^{8m} }{ n } }^{g}.
\end{eqnarray*}
For our choice of $m$ in \eqref{eq:choicem}, for $n$ large enough, 
$$
\frac{   (c \ell m) ^{8m} }{ n } \leq \frac{( \log n )^{16m} }{n} \leq n^{-1/17}. 
$$
In particular, the above geometric series converges  and, adjusting the value of $c$,  the right-hand side of \eqref{eqM:boundS} is an upper bound for $S_1$.  The treatment of $S_2$ and $S_3$ is exactly parallel to the treatment of $S_2$ and $S_3$ in  the proof of Proposition  \ref{prop:normDelta} with $\rho$ replacing $d-1$.  This concludes the proof. 
\end{proof}

\subsubsection{Operator norm of $R^{(\ell)}_{k}$}

We now repeat the argument for $R^{(\ell)}_{k}$. This a routine extension of the previous paragraph and Subsection \ref{subsec:Rlk}. 
\begin{proposition} \label{propM:normR}
Let  $\rho > \rho_1$ and $1 \leq \ell  \leq  \log n$ be an integer.  Let $\sigma$ be uniformly distributed on $S^G_n$ and  for $k \in [\ell]$, let $R_k^{(\ell)} = R_k^{(\ell)} (\sigma)$ be defined as in \eqref{eqM:defR}.  
Then, \whp  
$$ \sum_{k=1} ^\ell \| R_k^{(\ell)} \| \leq ( \log n) ^{40}   \rho ^{\ell}.$$
\end{proposition}
Let $m$ be a positive integer and $k \in [\ell]$. Arguing as in  \eqref{eq:trDeltak20}-\eqref{eqM:trDeltak2}, we find
\begin{eqnarray}
\| R^{(\ell)}_{k}  \| ^{2 m}  &\leq & \sum_{\gamma \in W^k_{\ell,m}} q (\gamma)  P_k (\gamma), \label{eqM:trDeltak20}
\end{eqnarray}
where $q(\gamma)$ was defined in \eqref{eqM:defqg} and $W^k_{\ell,m}$, $P_k(\gamma)$ are defined as follows. Let $k_i \in \{ k , \ell - k +1 \}$ be as in \eqref{eq:defki}. The set $W^k_{\ell,m}$ is the collection of  $\gamma = ( \gamma_1, \ldots, \gamma_{2m}) \in \vec E_n^{2\ell  \times 2m}$ such that for all $i \in [2m]$, $\gamma_i  = (\gamma_{i,1}, \ldots, \gamma_{i,2\ell})$ is non-backtracking and tangled but,  $$\gamma'_i = ( \gamma_{i,1}, \ldots , \gamma_{i,2k_i-2}) \quad  \hbox{ and } \quad \gamma''_i =  ( \gamma_{i,2k_i+1}, \ldots , \gamma_{i,2\ell})$$ are tangle-free.  We also have the boundary condition \eqref{eqM:defbound}. Finally, in \eqref{eqM:trDeltak20}, for $\gamma \in W^k_{\ell,m}$, we have set 
$$P_k (\gamma) =  \prod_{i=1}^{2m}  \prod_{t=1}^{k_i-1}  M^{\veps_i}_{\gamma_{i,2t-1}  \gamma_{i,2t}}   \prod_{t=k_i+1}^{\ell} M^{\veps_i}_{\gamma_{i,2t-1}}, $$
where $M^{\veps_i} = \underline M$ if $i$ is odd and $M^{\veps_i} = M$ if $i$ is odd.

As in the previous subsection, for each $\gamma \in W^k_{\ell,m} \subset \vec E_n^{2\ell \times 2m}$, we associate a multigraph $G_\gamma$ as in Definition \ref{defM1}.  We also partition  $W^k_{\ell,m}$ into isomorphism classes exactly as in the previous subsection. We define a canonical element in each isomorphic class thanks to the lexicographic order.  We also introduce the multigraph $G^k_\gamma = \cup_{i} (G_{\gamma'_i} \cup G_{\gamma''_i})$ where $G_{\gamma'_i}$, $G_{\gamma''_i}$ are as in Definition  \ref{defM1}. More precisely, the vertex set $G^k_\gamma$ of $V^k_\gamma = \cup_{i} (V_{\gamma'_i} \cup V_{\gamma''_i}) = \{ o_n(\gamma_{i,t}) : (i,t) \in [2m] \times [2\ell] : t  \notin \{2k_i-1,2k_i \}) \}$ and the set of visited pairs of half-edges is $\cE^k_\gamma = \cup_i (\cE_{\gamma'_i} \cup \cE_{\gamma''_i}) = \{ \{ \gamma_{i,2t-1}  , \gamma_{i,2t}  \} :   (i,t) \in [2m] \times [\ell], t \ne k_i \}$.

Since $o_n(\gamma_{i,2t}) = o_n(\gamma_{i,2t+1})$, we have $V^k_\gamma = V_\gamma$. Thus, Lemma \ref{leM:isopath} implies the following lemma.

\begin{lemma}\label{leM:isopath0}
Let $\gamma \in  W^k_{\ell,m}$ with $|V^k_\gamma| = s $. Then $\gamma$  is isomorphic to at most 
$
  n^s
$ 
elements in $ W^k_{\ell,m}$.  
\end{lemma}

We need an upper bound on the number of isomorphism classes in $W^k_{\ell,m}$.

\begin{lemma}\label{leM:enumpath0}
Let $\cW^k_{\ell,m} (s,a) $ be the subset of canonical paths in $W^k_{\ell,m}$ with $|V_\gamma| = s$, $|\cE_\gamma |= a$. Let $g = a - s +1$. If $g \leq 0$, $\cW^k_{\ell,m} (s,a)$ is empty. Otherwise, there exists a constant $c$ depending on $\rho$ and $G$ such that, 
$$
| \cW^k_{\ell,m} (s,a) | \leq  \rho^s (c  \ell  m )^{16 m g + 22 m}.
$$
\end{lemma}

\begin{proof}
The first claim is proved as in Lemma \ref{le:enumpath0}: for all $\gamma \in W_{\ell,m} ^k$, each connected component of $G_\gamma^k$ has a cycle.

For the second claim, We adapt the proof of Lemma \ref{leM:enumpath}, using the extra input of the proof of Lemma \ref{le:enumpath0}. For $i \in [2m]$, we define for $t \in [\ell]\backslash\{k_i\}$, $x_{i,t} = ( \gamma_{i,2t-1}, \gamma_{i,2t} )$. We then explore the sequence $(x_{i,t})$, $(i,t) \in T = \{ (i,t) \in [2m] \times [\ell] : t \ne k_i \}$ in lexicographic order. We denote by $(i,t)_-$ the preceding element in $T$ for the lexicographic order (with the convention $(1,1)_- = (1,0)$). For $(i,t) \in T$, we set $y_{i,t} = \{  \gamma_{i,2t-1}, \gamma_{i,2t} \}$. For each $(i,t) \in T$, we build a growing spanning forest $F_{i,t}$ of the graph visited so far as follows. The forest $F_{1,0}$ has a single vertex $\gamma_{1,1} = (1,1)$. By induction, for $(i,t) \in T$,  if the addition of $y_{i,t}$ to $F_{(i,t)_-}$  creates a cycle, we set $F_{i,t} = F_{(i,t)_-}$, and we say that  $y_{i,t}$ is an excess edge. Otherwise, we say that $(i,t)$ is a first time, that $y_{i,t}$ is a tree edge, and we define $F_{i,t}$ as the union of $F_{(i,t)_-}$ and $y_{i,t}$.

Arguing exactly as in the proof of Lemma \ref{le:enumpath0},  we find that there are at most  $g$ excess edges in each connected component of $G^k_{\gamma}$.

We may now repeat the proof of Lemma \ref{leM:enumpath}. The only difference is that, for each $i$, we use that $\gamma'_i$ and $\gamma''_i$ are tangled free, it gives short cycling times and long cycling times for both $\gamma'_i$ and $\gamma''_i$.  We also need a starting mark for $\gamma''_i$ equal to $(\gamma_{i,2k_i-1},\gamma_{i,2k_i}, \gamma_{i,2\tau-1}, \veps''_{i,0})$ where  $(i,\tau)$ is the next time that $y_{i,\tau}$ will not be a tree edge of the forest $F_{(i,k_i+1)_-}$ constructed so far and $\veps''_{i,0} = (e_{i,2 \tau }, \ldots , e_{i,2t''-1})$ is the projection of $(\gamma_{i,2 \tau}, \ldots , \gamma_{i,2t''-1})$ on $G$ and $(i,t'')$ is the next important time.  If $p$ is the number of connected components in $G^k_\gamma$. The analog of \eqref{eqM:sumsiq} is
\begin{equation*}\label{eqM:sumsiq0}
\sum_{i = 1}^{2m}\PAR{ \sum_{q=0}^{q'_i} (s')^i_{q} + \sum_{q=0}^{q''_i} (s'')^i_{q} } = s-p \leq s - 1,
\end{equation*}
where  $q'_i$ and $q''_i$ is the number of important times in $\gamma'_i$ and $\gamma''_i$, and $(s')^i_q$ and $(s'')^i _q$ are the number of new vertices between two successive important times in $\gamma'_i$ and $\gamma''_i$.

Then, for each $i$, there are at most $2$ short cycling times and $2 (g-1)$ long cycling times (since each connected component has most $g$ excess edges). There are at most $ \ell  ^{4m g}$ ways to position these times. The number of distinct half-edges $\gamma_{i,t}$ in $\gamma$ is at most $h = 4 \ell m$.  Arguing as in Lemma \ref{leM:enumpath}, we get that  $| \cW^k_{\ell,m} (s,a) | $ is upper bounded by 
$$
 r \ell ^{4 m g} \ell^{4 m (g+1)}\PAR{h^{2}}^ {4m(g-1)} \PAR{h^4\ell}^{4m} h^{2m}c_\rho^{4 m (g +1)}\rho^{s -1}   ( r^2 h^2  )^{2m},
$$
where the factor $( r^2 h^2)^{2m}$ accounts for the extra starting marks of $\gamma''_i$, $i \in [2m]$ (there are at most $r$ possibilities for $\gamma_{i,2k_i-1}$ since $o_n(\gamma_{i,2k_i-1}) = o_n(\gamma_{i,2k_i})$, $r h$ possibilities for $\gamma_{i,2k_i}$ and $h$ possibilities for $\gamma_{i,2\tau-1}$).  Taking $c$ large enough, we obtain the claimed statement. \end{proof}

For $\gamma \in W^k_{\ell,m}$, the average contribution of $\gamma$ in \eqref{eqM:trDeltak2} is
\begin{equation*}\mu_k(\gamma) = \dE P_k (\gamma) = \dE \prod_{i=1}^{2m}  \prod_{t=1}^{k_i-1}  M^{\veps_i}_{\gamma_{i,2t-1}  \gamma_{i,2t}}   \prod_{t=k_i+1}^{\ell} M^{\veps_i}_{\gamma_{i,2t-1}}.
\end{equation*}
If $\gamma \sim \gamma'$ then $\mu_k (\gamma) = \mu_k (\gamma')$. A straightforward extension of Lemma \ref{le:meanpath0} gives the following. 
\begin{lemma}\label{leM:meanpath0}
There is a universal constant $c > 0$ such that, if $6 \ell m \leq \sqrt{  n}$ and $\gamma \in W^k_{\ell,m} $ with $|V^k_\gamma| = s$, $|\cE^k_\gamma |= a$ and $g = a - s +1$, we have
$$
\ABS{ \mu_k(\gamma)} \leq  c^{g+m}  \PAR{\frac{ 1 }{n} }^a.  
$$
\end{lemma}

\begin{proof}[Proof of Proposition \ref{propM:normR}]
For $n \geq 3$, we define 
\begin{equation}\label{eqM:choicem0}
m = \left\lfloor  \frac{ \log n }{33 \log (\log n)} \right\rfloor.
\end{equation}
Since $\ell \leq \log n$, for $n$ large enough, $m$ is positive and $\ell m = o ( \log n) ^2$. It suffices to prove that for some constant $c >0$ and all $k \in [\ell]$, for all $n$ large enough
\begin{equation}\label{eqM:boundS0}
S_k = \sum_{\gamma \in W^k_{\ell,m} }   | \mu_k(\gamma) | \leq (c \ell  m )^{38 m} \rho^{2 \ell m},
\end{equation}
Indeed, from \eqref{eqM:trDeltak20}, this implies that 
$$
\dE \sum_{k=1} ^\ell \| R^{(\ell)}_{k}  \| ^{2 m} \leq \ell (d-1)^{m} (c \ell  m )^{38 m} \rho^{2 \ell m}.
$$ 
It then remains to use Markov inequality.

We now prove \eqref{eqM:boundS0}. Using  Lemma \ref{leM:isopath0}, Lemma \ref{leM:enumpath0} and Lemma \ref{leM:meanpath0}, we obtain, with $g = a - s +1$,
\begin{eqnarray*}
S_k & \leq &  \sum_{s=1}^{2 \ell m} \sum_{a = s} ^{ \infty } n^s \rho ^s    (c \ell  m )^{16 m g + 22 m} c^{g+ m } \PAR{\frac{ 1 }{n} }^a \\
& = &  \sum_{s=1}^{2 \ell m} \sum_{h = 0} ^{ \infty }  \rho ^s    (c \ell  m )^{16 m h + 38 m} c^{ h+ m +1} \PAR{\frac{ 1 }{n} }^h 
 \end{eqnarray*}
where, at the last line, we have done the change of variable $a = h +s$. Since $\rho >1$, we get for some  $c ' >0$, for all $n$ large enough, 
\begin{eqnarray*}
S_k & \leq  &    (c' \ell  m )^{38 m}   \sum_{s=1}^{ 2 \ell m } \rho^{s} \sum_{h=0} ^{\infty} \PAR{\frac{ c (c \ell m) ^{16 m} }{ n } }^{h} \\
& \leq &  \PAR{ \frac{\rho}{\rho-1}} (c' \ell  m )^{38 m}  \rho^{2 \ell m}  \sum_{h=0} ^{\infty} \PAR{\frac{ c (c \ell m) ^{16 m} }{ n } }^{h}.
\end{eqnarray*}
For our choice of $m$ in \eqref{eqM:choicem0}, we have, for $n$ large enough, 
$
   (c \ell m) ^{16 m} /  n   \leq n^{-1/33}. 
$
Hence, the above geometric series converges. Adjusting the value of the constant $c$, the right-hand side of \eqref{eqM:boundS0} is an upper bound for $S_k$. \end{proof}

\subsection{Proof of Theorem \ref{thM:main}}

\label{subsecM:end}

All ingredients are finally gathered. Recall that $d \geq 3$ defined in \eqref{eq:deftheta} is the largest degree of $G$. Let $1 < \rho_1 \leq d-1$ be the Perron eigenvalue of $B$. We fix  $\veps > 0$ and some  $0 < \kappa < 1/4$. We consider a sequence $\ell = \ell (n)$ with $\ell \sim \kappa \log_{d-1} n$. We may take $\rho \leq d-1$ in \eqref{eq:defrho} such that $\sqrt \rho \leq \sqrt \rho_1 + \veps$.  By Lemma \ref{leM:tangle} and Proposition \ref{leM:decompBl}, if $\Omega$ is the event that $G_n$ is $\ell$-tangle free, 
\begin{eqnarray*}
\dP \PAR{|\lambda_1 | \geq \sqrt{\rho_1}  +  2 \veps} & \leq &\dP \PAR{|\lambda_1 | \geq \sqrt {\rho} + \veps ; \Omega} + o(1) \\
& \leq & \dP \PAR{ J^ { 1 /  \ell}  \geq   \sqrt {\rho} + \veps  } + o(1), 
\end{eqnarray*}
where $J = \|  \uB^{(\ell)}_{n} \|  + \frac 1{ n}   \sum_{k = 1}^\ell \| R_{k}^{(\ell)} \|.$
However, by Propositions \ref{propM:normDelta}-\ref{propM:normR}, \whp
\begin{align*}
J  \leq ( \log n )^{15} \rho^{\ell/2}   +  \frac {(\log n)^{40}  }{ n} \rho^{\ell}  \leq ( \log n )^{15} \rho^{\ell/2} + o(1),
\end{align*} since $\rho^\ell \leq  n^{\kappa +o(1)} $. Finally, for our choice of $\ell$, $(\log n)^{15 / \ell} = 1 + O ( \log \log n / \log n )$.

\bibliographystyle{abbrv}
\bibliography{bib}

\bigskip
\noindent
 Charles Bordenave \\
 Institut de Math\'ematiques de Marseille. CNRS and Aix-Marseille University. \\
39, rue F. Joliot Curie. 13453 Marseille Cedex 13.  France. \\
\noindent
{E-mail:} \href{charles.bordenave@univ-amu.fr}{charles.bordenave@univ-amu.fr} \\

\end{document}